%% file: main.tex
\documentclass[bj]{imsart}
\usepackage{amsfonts}
\usepackage{amssymb, amsmath,yhmath}
\usepackage{enumitem}
\usepackage[mathscr]{eucal}
\usepackage{mathrsfs}
\usepackage{graphics}
\usepackage{tensor}
\usepackage{graphicx}
\usepackage{verbatim}
\usepackage{cases}
\usepackage{float}
\usepackage{url}
\usepackage[font=small,format=plain,labelfont=bf,up,textfont=it,up]{caption}

\usepackage[latin1]{inputenc}
\usepackage{tikz}
\usetikzlibrary{trees}

\startlocaldefs

    \newcommand{\argmin}{\mathop{\rm argmin}}

    \newcommand{\wh}{\widehat}

    \newcommand{\wt}{\widetilde}

\newtheorem{theorem}{Theorem}[section]
\newtheorem{lemma}{Lemma}[section]
\newtheorem{proposition}{Proposition}[section]
\newtheorem{definition}{Definition}[section]
\newtheorem{corollary}{Corollary}[section]

\numberwithin{equation}{section}
\newtheorem{remark}{Remark}[section]

\newenvironment{proof}[1][\sc Proof.]{\begin{trivlist}
\item[\hskip \labelsep {\bfseries #1}]}{\end{trivlist}}

\newcommand{\qed}{\nobreak \ifvmode \relax \else
      \ifdim\lastskip<1.5em \hskip-\lastskip
      \hskip1.5em plus0em minus0.5em \fi \nobreak
      \vrule height0.75em width0.5em depth0.25em\fi}
\endlocaldefs

\allowdisplaybreaks

\begin{document}
\begin{frontmatter}
\title{Asymptotic Confidence Regions for Density Ridges}
\runtitle{Ridge Confidence Regions}
\author{\fnms{Wanli} \snm{Qiao}\ead[label=e1]{wqiao@gmu.edu}\thanksref{t2}}
  \address{ Department of Statistics, George Mason University, 4400 University Drive, MS 4A7, Fairfax, VA 22030, USA.
           \printead{e1}}
%
%
%
 \thankstext{t2}{This work is partially supported by NSF grant DMS 1821154.} 
\begin{abstract} \noindent 
We develop large sample theory including nonparametric confidence regions for $r$-dimensional ridges of probability density functions on $\mathbb{R}^d$, where $1\leq r<d$. We view ridges as the intersections of level sets of some special functions. The vertical variation of the plug-in kernel estimators for these functions constrained on the ridges is used as the measure of maximal deviation for ridge estimation. Our confidence regions for the ridges are determined by the asymptotic distribution of this maximal deviation, which is established by utilizing the extreme value distribution of nonstationary $\chi$-fields indexed by manifolds.
\end{abstract}
%
%


\end{frontmatter}

\section{Introduction}

A ridge in a data cloud is a low-dimensional geometric feature that generalizes the concept of local modes, in the sense that ridge points are local maxima constrained in some subspace. In the literature ridges are also called filaments, or filamentary structures, which usually exhibit a network-like pattern. They are widely used to model objects such as fingerprints, fault lines, road systems, and blood vessel networks. The vast amount of modern cosmological data displays a spatial structure called Cosmic Web, and ridges have been used as a mathematical model for galaxy filaments (Sousbie et al. 2008).

The statistical study on ridge estimation has recently attracted much attention. See Genovese et al. (2009, 2012, 2014), Chen et al. (2015), and Qiao and Polonik (2016). One of the fundamental notions under ridge estimation is that ridges are {\em sets}, and most of the above statistical inference work focuses on the maximal (or global) deviation in ridge estimation, that is, how the estimated ridge captures the ground truth as a whole. This requires an appropriately chosen measure of global deviation. For example, the Hausdorff distance is used in Genovese et al. (2009, 2012, 2014) and Chen et al. (2015), while Qiao and Polonik (2016) use the supremum of ``trajectory-wise" Euclidean distance between the true and estimated ridge points, where trajectories are driven by the second eigenvectors of Hessian. Both distances measure the deviation of ridge estimation in the space where the sets live in, which we call the horizontal variation (HV).  

In this manuscript we develop large sample theory for the nonparametric estimation of density ridges, which in particular includes the construction of confidence regions for density ridges. Our methodology is based on the measure of global deviation in ridge estimation from a different perspective. Briefly speaking, we treat ridges as intersections of special level sets, and use the measure of maximal deviation in {\em levels}, which we call vertical variation (VV). 

We give the mathematical definition of ridges. Let $\nabla f(x)$ and $\nabla^2 f(x)$ be the gradient and Hessian of a twice differentiable probability density function $f$ at $x\in\mathbb{R}^{d}$ with $d \geq 2$. Let $v_1(x),\cdots,v_d(x)$ be unit eigenvectors of $\nabla^2 f(x)$, with corresponding eigenvalues $\lambda_1(x)\geq\lambda_2(x)\geq \cdots \geq \lambda_{d}(x)$. For $r=1,2,\cdots,d-1$, write $V(x) = (v_{r+1}(x),\cdots, v_{d}(x)).$ The $r$-ridge $\mathcal{M}^r$ induced by $f$ is defined as the collection of points $x$ that satisfies the following two conditions: 
\begin{align}
&V(x)^T\nabla f(x) = 0,\label{condition1}\\
& \lambda_{r+1}(x) < 0.\label{condition2}
\end{align}
We fix $r\geq 1$ in this manuscript and denote the ridge by $\mathcal{M}$. This definition has been widely used in the literature (e.g., Eberly, 1996). A ridge point $x$ is a local maximum of $f$ in a $(d-r)$-dimensional subspace spanned by $v_{r+1}(x),\cdots, v_{d}(x)$. This geometric interpretation can be seen from the fact that $v_i^T\nabla f$ and $\lambda_i$ are the first and second order directional derivatives of $f$ along $v_i$, respectively. In fact, if we take $r=0$, then conditions (\ref{condition1}) and (\ref{condition2}) just define the set of local maxima, which is the 0-ridge. Condition (\ref{condition1}) indicates that an $r$-ridge is contained in the intersection of $(d-r)$ level sets of the functions $v_i^T\nabla f$, $i=r+1,\cdots,d$, and is an $r$-dimensional manifold with co-dimension $(d-r)$ under some mild assumptions (e.g. see assumption ({\bf F2}) below). 
%

Given a sample $X_1,\cdots,X_n$ of $f$, the ridge $\mathcal{M}$ can be estimated using a plug-in approach based on kernel density estimators (KDE). Let $\wh f \equiv \wh f_{n,h}$ be the KDE of $f$ with bandwidth $h>0$ (see (\ref{kde})), and let $\wh v_1(x),\cdots,\wh v_d(x)$ be unit eigenvectors of $\nabla^2 \wh f(x)$, with corresponding eigenvalues $\wh\lambda_1(x)\geq\wh\lambda_2(x)\geq \cdots \geq \wh\lambda_{d}(x)$. Also write $\wh V(x) = (\wh v_{r+1}(x),\cdots, \wh v_{d}(x))$. Then a plug-in estimator for $\mathcal{M}$ is $\wh{\mathcal{M}}$, which is the set defined by plugging in these kernel estimators into their counterparts in conditions (\ref{condition1}) and (\ref{condition2}). See Figure~\ref{illustration} for example.
%
%
%
Genovese et al. (2009, 2012, 2014) and Chen et al. (2015) focus on the estimation of ridges induced by the smoothed kernel density function $f_h\equiv\mathbb{E}\wh f$, instead of the true density $f$. Such ridges, denoted by $\mathcal{M}_h$, depend on the bandwidth $h$ and are called surrogates. Focusing on $\mathcal{M}_h$ instead of $\mathcal{M}$ avoids the well-known bias issue in nonparametric function and set estimation. 

\begin{figure}[ht]
\begin{center}
\includegraphics[height=6cm]{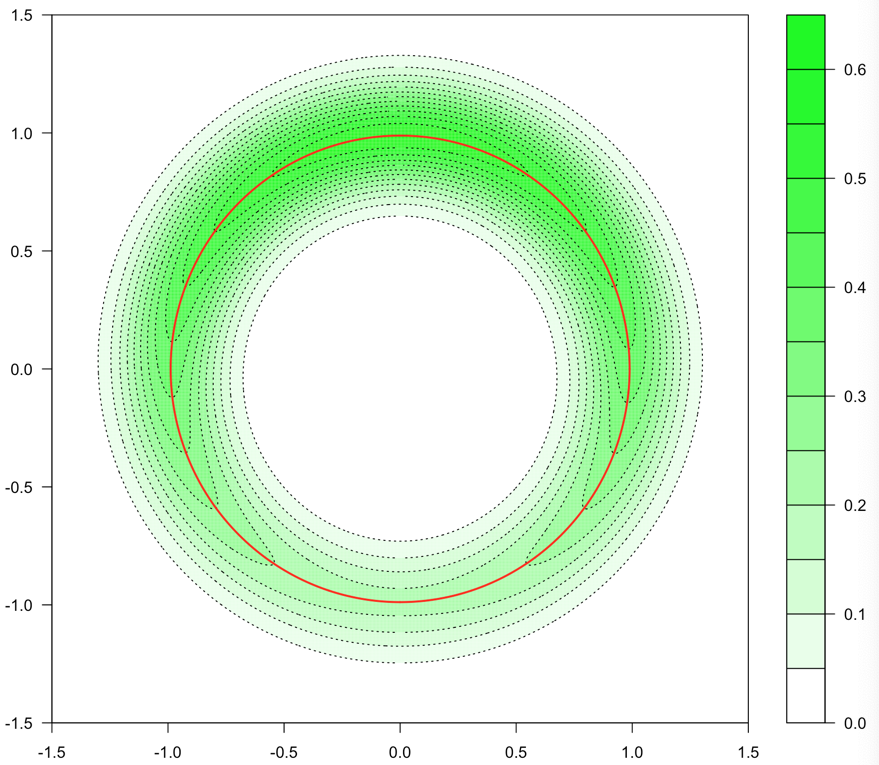}
\includegraphics[height=6cm]{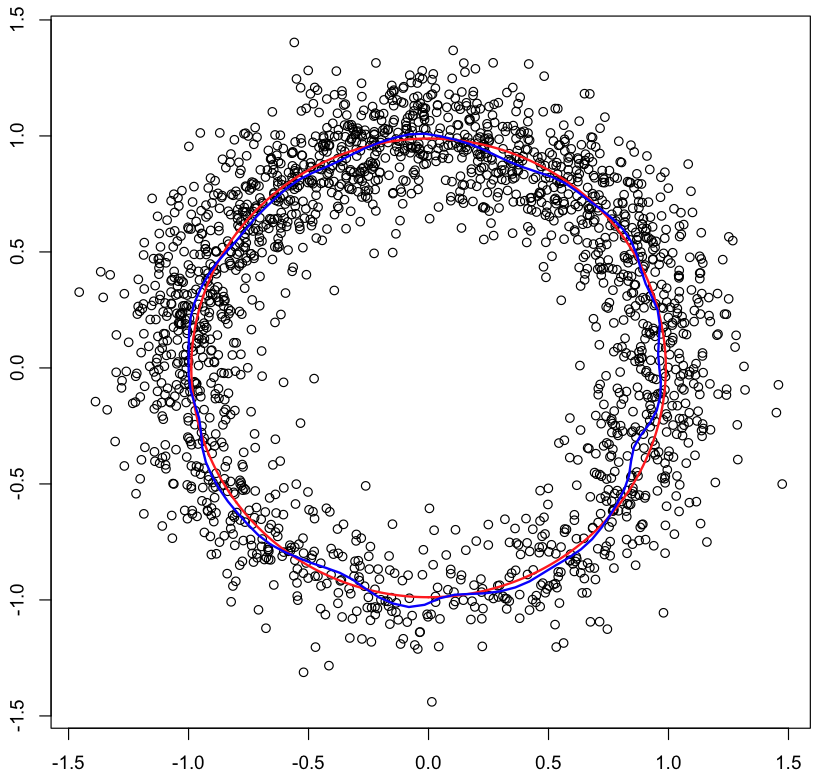}
\caption{Left: contour plot of a density function, where the red curve is a ridge and the dotted lines are contour lines; Right: simulated data points from the density function and the estimated ridge (blue curve).}
\label{illustration}
\end{center}
\end{figure}

In this manuscript we consider confidence regions for both $\mathcal{M}$ and $\mathcal{M}_h$ in the form of
\begin{align}\label{confregform}
\wh C_{n,h}(a_n,b_n) = \left\{x:  \sqrt{nh^{d+4}}\left\| Q_n(x) \wh V(x)^T \nabla \wh f(x) \right\| \leq a_n, \text{and}\; \wh\lambda_{r+1}(x) < b_n \right\},
\end{align}
where $a_n>0$, $b_n\in \mathbb{R}$ and $Q_n(x)$ is a normalizing matrix. Here determining $Q_n$, $a_n$ and $b_n$ is critical to guarantee that $\wh C_{n,h}(a_n,b_n)$ has a desired asymptotic coverage probability for $\mathcal{M}$ or $\mathcal{M}_h$ as $n\rightarrow\infty$ and $h\rightarrow0$. The basic idea for our VV approach is as follows. We consider density ridges as the intersection of the zero-level sets of the functions $V^T\nabla f$ and a sublevel set of $\lambda_{r+1}$. When we use plug-in estimators for these functions, we allow their values to vary in a range (specified by $a_n$ and $b_n$), which implicitly defines a neighborhood near $\wh{\mathcal{M}}$. The shape of this neighborhood is envisioned as a tube around $\wh{\mathcal{M}}$ with varying local radii. This tube is geometrically different from the one with constant radius based on the asymptotic distribution of $d_H(\wh{\mathcal{M}},\mathcal{M}_h)$, which is the Hausdorff distance (belonging to HV) between $\wh{\mathcal{M}}$ and $\mathcal{M}_h$. As seen from its definition given in (\ref{condition1}) and (\ref{condition2}), ridge estimation mainly involves the estimation of the density gradient and Hessian. Between these two major components, the rate $\sqrt{nh^{d+4}}$ in (\ref{confregform}) follows from the rate of convergence of the Hessian, which is $1/\sqrt{nh^{d+4}}$ (ignoring the bias). Note that the rate of convergence of the gradient is $1/\sqrt{nh^{d+2}}$, which is much faster than that of the Hessian, and makes the Hessian estimation a dominant component in ridge estimation. We note in passing that this statement also applies to the asymptotic properties of $d_H(\wh{\mathcal{M}},\mathcal{M}_h)$ (see Chen et al., 2015).  

The asymptotic validity of the confidence regions for $\mathcal{M}_h$ and $\mathcal{M}$ in the form of (\ref{confregform}) can be verified in the following steps, which are also the main results we will show in the manuscript. First note that if we write $B_n(x)=\| Q_n(x) \wh V(x)^T \nabla \wh f(x) \|$, then $\mathcal{M}_h \subset \wh C_{n,h}(a_n,b_n)$ is equivalent to $\sqrt{nh^{d+4}}\sup_{x\in\mathcal{M}_h}  B_n(x) \leq a_n $ and $\sup_{x\in\mathcal{M}_h}  \wh\lambda_{r+1}(x) < b_n$. Under some regularity assumptions one can show that

(i) the distribution of $\sqrt{nh^{d+4}}\sup_{x\in\mathcal{M}_h}  B_n(x)$ equals that of $\sup_{g\in\mathcal{F}_h}\mathbb{G}_n(g)$ asymptotically, where $\mathbb{G}_n$ is an empirical process and $\mathcal{F}_h$ is a class of functions, which is induced by linear functionals of second derivatives of kernel density estimation;

(ii) the distribution of $\sup_{g\in\mathcal{F}_h}\mathbb{G}_n(g)$ is asymptotically the same as that of $\sup_{g\in\mathcal{F}_h}\mathbb{B}(g)$, where $\mathbb{B}$ is a locally stationary Gaussian process indexed by $\mathcal{F}_h$;

(iii) the distribution of $\sup_{g\in\mathcal{F}_h}\mathbb{B}(g)$ is derived by applying the extreme value theory of $\chi$-fields indexed by manifolds developed in our companion work Qiao (2019b). 

Then $a_n$ is determined by the above approximations and distributional results and $b_n$ is chosen such that $\sup_{x\in\mathcal{M}_h}  \wh\lambda_{r+1}(x) < b_n$ holds with a probability tending to one.  In fact one can show that $P(\mathcal{M}_h \subset \wh C_{n,h}(a_n,b_n))=e^{-e^{-z}}+o(1)$ with $b_n=0$ and $a_n=\Large[\frac{z+c}{\sqrt{2r\log{(h^{-1})}}} +  \sqrt{2r\log(h^{-1})} \Large]$, for some $c>0$ depending on $f$, $K$, and $\mathcal{M}_h$. This type of result is similar to the confidence bands for univariate probability density functions developed in the classical work of Bickel and Rosenblatt (1973). The derivation for $\mathcal{M}$ is similar except that we have to deal with the bias in the estimation.

The way that we study ridge estimation is naturally connected to the literature of level set estimation (e.g. Hartigan, 1987; Polonik, 1995; Tsybakov, 1997; Polonik and Wang, 2005; Cadre, 2006; Mason and Polonik, 2009), which mainly focuses on density functions and regression functions. Confidence regions for level sets have been studied in Mammen and Polonik (2013), Chen et al. (2017), and Qiao and Polonik (2019). It is clear that technically a ridge is a more sophisticated object to study than a density or regression level set, not only because the former involves the estimation of eigen-decomposition of Hessian and its interplay with gradient, but also a ridge is viewed as the {\em intersection} of level sets of multiple functions if $d-r\geq2$. To our knowledge there are no nonparametric distributional results for the estimation of intersections of density or regression level sets in the literature. 
In addition to the papers mentioned above, previous work on ridge estimation also includes Hall et al. (1992), Wegman et al. (1993), Wegman and Luo (2002), Cheng et al. (2004), Arias-Castro et al. (2006), Ozertem and Erdogmus (2011), Genovese et al. (2017), and Li and Ghosal (2019).  

The manuscript is organized as follows. We first introduce our notation and assumptions in Section~\ref{notationassump}. In Section~\ref{mhconfden} we develop the asymptotic confidence regions for $\mathcal{M}_h$ following the procedures listed above. Specifically, steps (i)-(iii) are established in Proposition~\ref{ridgenessest}, and Theorems~\ref{gassianapprox}, and \ref{confidenceregion}, respectively. In Section~\ref{biascorrect}, we use bias correction methods to extend the results to asymptotic confidence regions for $\mathcal{M}$. The confidence regions involve unknown surface integrals on ridges. In Section~\ref{unknownest} we show the asymptotic validity of the confidence regions with these unknown quantities replaced by their plug-in estimators. For technical reasons, the consideration of critical points on ridges are deferred until Section~\ref{Generalization}, where we also discuss different choices of $b_n$. The proofs are given in Section~\ref{proofssec} and the supplementary material.

\section{Notation and assumptions}\label{notationassump}

We first give the notion used in the manuscript. For a real matrix $A$ and compatible vectors $u$, $v$, denote $\langle u,v\rangle_A = u^TAv$. Also we write $\langle u,u\rangle_A=\|u\|_A^2$ and $\|u\|$ is the Euclidian norm. Let $\|A\|_F$ and $\|A\|_{\max}$ be the Frobenius and max norms, respectively. Let $A^+$ be the Moore-Penrose pseudoinverse of $A$ (see page 36, Magnus and Neudecker, 2007), which always exists and is unique. For a positive integer $m$, let $\mathbf{I}_m$ be the $m\times m$ identity matrix. For a vector field $W: \mathbb{R}^m\mapsto \mathbb{R}^n$ let $\mathbf{R}(W)$ denote the matrix given by $\mathbf{R}(W) := \int_{\mathbb{R}^m}
W(x)W(x)^T dx \subset \mathbb{R}^{n\times n}$, assuming the integral is well defined. For a symmetric $d\times d$ matrix $A$, $\text{vec} (A)$ vectorizes $A$ by stacking the columns of $A$ into a $d^2\times 1$ column vector, while $\text{vech} (A)$ only vectorize the lower triangular part of $A$ into a $d(d+1)/2\times 1$ column vector. The duplication matrix is such that $\text{vec} (A) = D \;\text{vech} (A)$. The matrix $D$ does not depend on $A$ and is unique for dimension $d$ (and we have suppressed $d$ in the notation). For example, when $d=2$ and $A=\begin{pmatrix} 
a_{11} & a_{12} \\
a_{12} & a_{22} 
\end{pmatrix}$, using the above notation we have
\begin{align*}
\text{vech} (A) = (a_{11},a_{12},a_{22})^T,\;\; \text{vec} (A) = (a_{11},a_{12},a_{12},a_{22})^T,\;\; \text{and} \;\; D= \begin{pmatrix} 
1 & 0 & 0 & 0 \\
0 & 1 & 1 & 0\\
0 & 0 & 0 & 1
\end{pmatrix}^T.
\end{align*}

For a smooth function $K:\mathbb{R}^d\mapsto\mathbb{R}$, let $\nabla K$ and $\nabla^2 K$ be its gradient and Hessian, respectively, and we denote $d^2 K =\text{vech}\nabla^2 K$. Let $\mathbb{Z}_+$ be the set of non-negative integers. For $m\in\mathbb{Z}_+$, we use $\mathscr{H}_m$ to denote the $m$-dimensional Hausdorff measure. Let $\mathscr{B}(x,t)=\{y\in\mathbb{R}^d:\; \|y-x\|\leq t\}$ be the ball centered at $x$ with radius $t>0$. For a set $M\subset\mathbb{R}^d$ and $\epsilon>0$, let $M\oplus\epsilon= \cup_{x\in M}\mathscr{B}(x,\epsilon)$, which is the $\epsilon$-enlarged set of $M$.  For $m\in\mathbb{Z}_+$, let $\mathbb{S}^m=\{x\in\mathbb{R}^{m+1}:\; \|x\|=1\}$ be the unit $m$-sphere. For any subset $\mathcal{A}\subset\mathbb{R}^d$, let $\mathbf{1}_{\mathcal{A}}$ be the indicator function of $\mathcal{A}$.

Given an i.i.d. sample $X_1,\cdots X_n$ from the probability density function $f$ on $\mathbb{R}^d$, denote the kernel density estimator 
\begin{align}\label{kde}
\wh f(x) = \wh f_{n,h}(x)= \frac{1}{nh^d} \sum_{i=1}^n K\left( \frac{x-X_i}{h}\right),\; x\in\mathbb{R}^d,
\end{align}
where $h>0$ is a bandwidth and $K$ is a twice differentiable kernel density function on $\mathbb{R}^d$. The notation $h$ is used as a default bandwidth unless otherwise indicated, and we suppress the subscripts $n,h$ in the kernel density estimator and all quantities induced by it (so that $\wh V=\wh V_{n,h}$ and $\wh\lambda_{r+1}=\wh\lambda_{r+1,n,h}$ for example). Let $f_h(x) = \mathbb{E}\wh f(x)$ and let $v_{1,h}(x),\cdots,v_{d,h}(x)$ be unit eigenvectors of $\nabla^2 f_h(x)$, with corresponding eigenvalues $\lambda_{1,h}(x)\geq\lambda_{2,h}(x)\geq \cdots \geq \lambda_{d,h}(x)$. Also write $V_h(x) = (v_{r+1,h}(x),\cdots, v_{d,h}(x))$. We focus on ridge estimation on a compact subset $\mathcal{H}$ of $\mathbb{R}^d$, which is assumed to be known. For simplicity, suppose that $\mathcal{H}$ is the hypercube $[0,1]^d$, and all the ridge definitions $\mathcal{M}$, $\wh{\mathcal{M}}$ and $\mathcal{M}_h$ are restricted on $\mathcal{H}$, such as $\mathcal{M}_h=\{x\in\mathcal{H}: V_h(x)^T\nabla f_h(x)=0,\; \lambda_{r+1,h}(x)<0\}$.
%

For $\gamma=(\gamma_1,\cdots,\gamma_d)^T\in\mathbb{Z}_+^d$, let $|\gamma| = \gamma_1+\cdots+\gamma_d$. For a function $g:\mathbb{R}^d\mapsto\mathbb{R}$ with $|\gamma|$th partial derivatives, define 
\begin{align}\label{multiindex}
g^{(\gamma)}(x) = \frac{\partial^{|\gamma|} }{\partial^{\gamma_1}x_1\cdots \partial^{\gamma_d}x_d} g(x),\; x\in\mathbb{R}^d.
\end{align}
Let $\rho_1=(3,0,\cdots,0)^T\in\mathbb{Z}_+^d$ and $\rho_2=(2,1,0,\cdots,0)^T\in\mathbb{Z}_+^d$. Define $a_K= \frac{\int_{\mathbb{R}^{d}}[K^{(\rho_1)}(s)]^2ds}{\int_{\mathbb{R}^{d}}[K^{(\rho_2)}(s)]^2ds}.$ If $d\geq 3$, let $\rho_3=(1,1,1,0,\cdots,0)^T\in\mathbb{Z}_+^d$ and define $b_K= \frac{\int_{\mathbb{R}^{d}}[K^{(\rho_3)}(s)]^2ds}{\int_{\mathbb{R}^{d}}[K^{(\rho_2)}(s)]^2ds}.$ Let $\mathbf{R}:=\mathbf{R}(d^2K)$. 
%
For $\delta>0$, define 
\begin{align}\label{neighborhood}
\mathcal{N}_{\delta}(\mathcal{M}) = \{x\in\mathcal{H}: \|\nabla f(x)^T V(x)\|\leq \delta,\; \lambda_{r+1}(x)<0\},
\end{align}
which is a small neighborhood of $\mathcal{M}$ when $\delta$ is small. For a bandwidth $h>0$, let $\gamma_{n,h}^{(k)} = \sqrt{\frac{\log{n}}{nh^{d+2k}}},$ which is the rate of convergence of $\sup_{x\in\mathbb{R}^d}|\wh f^{(\gamma)}(x)-f_h^{(\gamma)}(x)|$ for $|\gamma|=k\in\mathbb{Z}_+$ under standard assumptions. We use the following assumptions in the construction of confidence regions for ridges.\\[-5pt]

\hspace{-10pt}{\bf Assumptions}: \\[3pt]
({\bf F1}) $f$ is four times continuously differentiable on $\mathcal{H}$.\\[3pt]
({\bf F2}) There exists $\delta_0>0$ such that $\mathcal{N}_{\delta_0}(\mathcal{M})\subset \mathcal{H}$ and the following is satisfied. When $d-r=1$, we require that $\|\nabla(\nabla f(x)^T v_d(x))\|>0$ for all $x\in\mathcal{N}_{\delta_0}(\mathcal{M})$; When $d-r\geq 2$, we require that $\nabla(\nabla f(x)^T v_i(x))$, $i=r+1,\cdots,d$ are linearly independent for all $x\in\mathcal{N}_{\delta_0}(\mathcal{M})$.\\[3pt]
({\bf F3}) $\{x\in\mathcal{H}:\; \lambda_{r+1}(x)=0,\; V(x)^T\nabla f(x)=0\}=\emptyset$.\\[3pt]
({\bf F4})  For $x\in\mathcal{N}_{\delta_0}(\mathcal{M})$, the smallest $d-r$ eigenvalues of $\nabla^2 f(x)$ are simple, i.e., $\lambda_{r}(x) > \lambda_{r+1}(x)>\cdots>\lambda_d(x)$. In particular, $\lambda_{r}(x) > \lambda_{r+1}(x)$ for $x\in\mathcal{H}$.\\[3pt]
({\bf K1}) The kernel function $K$ is a spherically symmetric probability density function on $\mathbb{R}^d$ with $\mathscr{B}(0,1)$ as its support. It has continuous partial derivatives up to order 4. \\ [3pt]
({\bf K2}) For any open ball $\mathcal{S}$ with positive radius contained in $\mathscr{B}(0,1)$, the component functions of $\mathbf{1}_{\mathcal{S}}(s)d^2 K(s)$ are linearly independent. \\[3pt]
({\bf K3}) We require $a_K>1$ if $d=2$, or $a_Kb_K>1$ if $d\geq3$.\\

%
\begin{remark}\label{discuss} $\;$\\[-10pt]

{\em
(i) Note that ridges are defined using the second derivatives of densities. Assumption ({\bf F1}) requires the existence of two additional orders of derivatives. This is similar to other work on the distributional results of ridge estimation (see Chen et al., 2015 and Qiao and Polonik, 2016). 

(ii) Assumption ({\bf F2}) guarantees that the ridge has no flat parts, which is comparable to the margin assumption in the literature of level set estimation (Polonik, 1995). In addition, as we consider ridges as intersections of level sets when $d-r\geq 2$, this assumption guarantees the transversality of the intersecting manifolds. Assumption ({\bf F2}) holds, e.g. if $f$ satisfies assumptions (A1) and (P1) in Chen et al. (2015) (see their Lemma 2). 

(iii) Assumptions ({\bf F3})-({\bf F4}) exclude some scenarios that are on the boundary of the class of density functions we consider (note that these assumptions only exclude some equalities). Here we give some brief discussion of the implications of these assumptions.\\[-17pt] 
\begin{itemize}
\item[a)] Assumption ({\bf F3}) avoids the existence of some degenerate ridge points. Such points have zero first and second directional derivatives along $v_{r+1}$ and so they are almost like ridge points. This assumption has been used in Genovese et al. (2014), Chen et al. (2015) and Qiao and Polonik (2016).
\item[b)] Assumption ({\bf F4}) requires that the smallest $d-r$ eigenvalues of $\nabla^2 f(x)$ for $x\in\mathcal{M}$ all have multiplicity one, for the following technical consideration. When an eigenvalue is repeated, the corresponding eigenvectors might have discontinuity with respect to a small perturbation of the Hessian matrix (e.g., $\nabla^2 \wh f(x)-\nabla^2 f(x)$). 
\end{itemize}

(iv) Assumptions ({\bf K1})-({\bf K3}) are for the kernel function $K$. In particular ({\bf K2}) can guarantee that ${\mathbf{R}}$ is positive definite. In general one can show that $a_K\geq1$ and $b_K\leq1$ (see Lemma~\ref{kernelratio}). So ({\bf K3}) excludes some cases on the boundary of the class of kernel functions we consider. Some properties of the kernel functions can be found in Lemma~\ref{kernelratio}. One can show that the following kernel density function is an example that satisfies ({\bf K1})-({\bf K3}):
\begin{align*}
K(x) = c_{d}(1-\|x\|^2)^5 \mathbf{1}_{\mathscr{B}(0,1)}(x),\; x\in\mathbb{R}^d,
\end{align*}
where $c_d$ is a normalizing constant.
}
\end{remark}

\section{Main Results}

In the literature, the following assumption or even stronger ones are imposed to get distributional results for density ridge estimation. See assumption (P2) of Chen et al. (2015) and assumption (F7) of Qiao and Polonik (2016).\\[3pt]
\leavevmode{\parindent=1em\indent} ({\bf F5}) $\|\nabla f(x)\|\neq 0$, for all $x\in\mathcal{M}$.\\[3pt]
In other words, it is assumed that $\mathcal{M}$ does not contain any critical points of $f$. This assumption excludes many important scenarios in practice because ({\bf F5}) implies that $f$ does not have local modes on $\mathcal{H}$.

Our confidence regions for $\mathcal{M}_h$ and $\mathcal{M}$ eventually do not require assumption ({\bf F5}). But the critical points and regular points on the ridges need to be treated in different ways, because for critical points the estimation is mainly determined by the gradient of $f$, while the estimation of regular ridge points depends on both the gradient and Hessian. It is known that the estimation of Hessian has a slower rate of convergence than the critical points using kernel type estimators, so the estimation behaves differently on regular ridge and critical points. To deal with this issue, the strategy we use is to construct confidence regions for the set of critical points and regular ridge points individually and then combine them (see Section~\ref{Generalization}). For convenience we will first exclude critical points from our consideration and tentatively assume ({\bf F5}).

%
%

\subsection{Asymptotic confidence regions for $\mathcal{M}_h$}\label{mhconfden}

Given any $0<\alpha<1$, we first study how to determine $a_n$ and $b_n$ to make $\wh C_{n,h}(a_n,b_n)$ an asymptotic $100(1-\alpha)\%$ confidence region for $\mathcal{M}_h$. The following lemma shows some basic properties of $\mathcal{M}$ as well as $\mathcal{M}_h$. For any subset $\mathcal{L}\subset\mathbb{R}^d$ and $x\in\mathbb{R}^d$, let $d(x,\mathcal{L})=\inf_{y\in\mathcal{L}}\|x-y\|$. A point $u\in\mathcal{L}$ is called a normal projection of $x$ onto $\mathcal{L}$ if $\|x-u\| = d(x,\mathcal{L})$. For $x\in\mathcal{L}$, let $\Delta(\mathcal{L},x)$ denote the {\em reach} of ${\cal L}$ at $x$ (Federer, 1959), which is the largest $r\geq0$ such that each point in $\mathscr{B}(x,r)$ has unique normal projection onto $\mathcal{L}$. The reach of $\mathcal{L}$ is defined as $\Delta(\mathcal{L}):=\inf_{x\in\mathcal{L}} \Delta(\mathcal{L},x)$, which reflects the curvature of $\mathcal{L}$ if it is a manifold. Recall that $\mathcal{N}_{\delta}(\mathcal{M})$ defined in (\ref{neighborhood}) is a small neighborhood of $\mathcal{M}$ when $\delta$ is small. 
\begin{lemma}\label{lambda2bound}
Under assumptions ({\bf F1}), ({\bf F2}), ({\bf F3}), ({\bf F4}) and ({\bf K1}), we have\\
(i) $\mathcal{M}$ is an $r$-dimensional compact manifold with positive reach.\\
When $h$ is small enough, we have \\
(ii) $\mathcal{M}_h\subset \mathcal{N}_{\delta_0}(\mathcal{M})$, where $\delta_0>0$ is given in ({\bf F2});\\
(iii) $\inf_{x\in \mathcal{M}_h} [\lambda_{j-1,h}(x) - \lambda_{j,h}(x) ] > \beta_0$, $j=r+1,\cdots,d$, and $\sup_{x\in \mathcal{M}_h} \lambda_{r+1,h}(x) < -\beta_0$ for some constant $\beta_0>0$ that does not depend on $h$;\\
(iv) $\mathcal{M}_h$ is an $r$-dimensional manifold with $\Delta(\mathcal{M}_h)>\beta_1$ for some constant $\beta_1>0$ that does not depend on $h$.
\end{lemma}

\begin{remark}\label{zerob}{\em
Property (iii) states that $\lambda_{r+1,h}$ is uniformly bounded away from zero on $\mathcal{M}_h$. As we show in Lemma~\ref{lambda2}, $\wh\lambda_{r+1}$ is a strongly uniform consistent estimator of $\lambda_{r+1,h}$ under our assumptions, that is, $\sup_{x\in\mathcal{H}}|\wh\lambda_{r+1}(x)-\lambda_{r+1,h}(x)|=o(1)$ almost surely, which implies that with probability one $\wh\lambda_{r+1}$ has the same sign as $\lambda_{r+1,h}$ on $\mathcal{M}_h$ for large $n$. This allows us to use $b_n=0$ in $\wh C_{n,h}(a_n,b_n)$, and focus on the behavior of $\wh V(x)^T\nabla\wh f(x)$ on $\mathcal{M}_h$ to choose $a_n$ so that $\wh C_{n,h}(a_n,b_n)$ in (\ref{confregform}) is an asymptotic confidence region for $\mathcal{M}_h$. Also see Section~\ref{Generalization} for different choices of $b_n$.}
\end{remark}

%
By VV we mainly mean the behavior of $\wh V(x)^T\nabla\wh f(x) = \wh V(x)^T\nabla\wh f(x) - V_h(x)^T\nabla f_h(x)$ for $x\in\mathcal{M}_h$. The following proposition shows the the asymptotic normality of this difference, which can be uniformly approximated by a linear combination of $d^2 \wh f(x) - d^2 f_h(x)$. This is not surprising because the difference depends on the estimation of eigenvectors of the Hessian, which has a slower rate of convergence than the estimation of the gradient. Note that each unit eigenvector has two possible directions. Without loss of generality, for $i=d-r,\cdots,d$, suppose that we fix the orientations of $\wh v_i(x)$, $v_{i,h}(x)$ and $v_i(x)$ in such a way that they vary continuously for $x$ in a neighborhood of $\mathcal{M}$ and have pairwise acute angles. 

For two matrices $A$ and $B$, let $A\otimes B$ be the Kronecker product between $A$ and $B$ (cf. page 31, Magnus and Neudecker, 2007). Recall that $D$ is the duplication matrix. For $i=r+1,\cdots,d$, let 
\begin{align}\label{pseudoinv}
m_i(x) = D^T\left(v_i(x)\otimes \sum_{j=1}^r\left[\frac{v_j(x)^T\nabla f(x)}{\lambda_i(x) - \lambda_j(x)} v_j(x)\right] \right),
\end{align}
which are $d(d+1)/2$ dimensional column vectors. 
Let $M(x)=(m_{r+1}(x),\cdots m_d(x))$, which is a $[d(d+1)/2]\times (d-r)$ matrix. Recall that $V_h(x)^T\nabla f_h(x)=0$ for $x\in\mathcal{M}_h$. The following result shows the asymptotic behavior of $\wh V(x)^T\nabla\wh f(x)$ on $\mathcal{M}_h$.

\begin{proposition}\label{pointwisenormal}
Under assumptions ({\bf F1}) - ({\bf F5}), ({\bf K1}), and ({\bf K2}), as $\gamma_{n,h}^{(2)}\rightarrow0$ and $h\rightarrow0$, we have 
\begin{align}\label{uniformapp}
\sup_{x\in\mathcal{M}_h} \big\|\wh V(x)^T\nabla\wh f(x) - M(x)^T[d^2 \wh f(x) - d^2 f_h(x)] \big\|  = O_p\left( \gamma_{n,h}^{(1)} + (\gamma_{n,h}^{(2)})^2 \right),
\end{align}
and for $x\in\mathcal{N}_{\delta_0}(\mathcal{M})$,
\begin{align}\label{diffclt}
\sqrt{nh^{d+4}} M(x)^T[d^2 \wh f(x) - d^2 f_h(x)] \rightarrow_D \mathscr{N}_{d-r}(0,f(x)\Sigma(x)),\;\; \text{as } n\rightarrow\infty,
\end{align}
where $\Sigma(x) =M(x)^T\mathbf{R} M(x)$ is a positive definite matrix for $x\in\mathcal{N}_{\delta_1}(\mathcal{M})$ for some constant $\delta_1>0$, and $x\in\mathcal{M}_h$, when $h$ is small enough. 
%
\end{proposition}

%
%
%
For a positive definite matrix $A$, let $A^{1/2}$ be its square root such that $A^{1/2}$ is also positive definite and $A=A^{1/2}A^{1/2}$. It is known that $A^{1/2}$ is uniquely defined. The asymptotic normality result in (\ref{diffclt}) suggests that we can standardize $\wh V(x)^T\nabla\wh f(x)$ by left multiplying the matrix $Q(x):=[f(x) \Sigma(x)]^{-1/2}$, which is unknown and can be further estimated by a plug-in estimator $Q_n(x):=[\wh f(x)\wh\Sigma(x)]^{-1/2}$ as specified below. Let $\wh\Sigma(x) =\wh M(x)^T\mathbf{R} \wh M(x)$ with $\wh M(x)=(\wh m_{r+1}(x),\cdots \wh m_d(x))$, where $$\wh m_i(x) = D^T\left(\wh v_i(x)\otimes \sum_{j=1}^r\left[\frac{\wh v_j(x)^T\nabla \wh f(x)}{\wh \lambda_i(x) - \wh \lambda_j(x)}\wh v_j(x)\right] \right) .$$ 
%
Let 
\begin{align}\label{Bnx}
B_n(x)= \| Q_n(x) \wh V(x)^T\nabla\wh f(x) \| = \| \wh V(x)^T\nabla\wh f(x) \|_{[\wh f(x)\wh\Sigma(x)]^{-1}}. 
\end{align}
We consider the following form of confidence regions for $\mathcal{M}_h$, which is slightly more formal than (\ref{confregform}). For any $a_n\geq 0$ and $b_n\in\mathbb{R}$, let
\begin{align}\label{hatCnh}
\wh C_{n,h}(a_n,b_n)=\left\{x\in\mathcal{H}:\;   \sqrt{nh^{d+4}} B_n(x) \leq a_n, \text{and}\; \wh\lambda_{r+1}(x) <b_n \right\}. 
\end{align}

We first consider $b_n=0$ for the reason given in Remark~\ref{zerob} and for simplicity write $\wh C_{n,h}(a_n)=\wh C_{n,h}(a_n,0)$. Alternative choices of $b_n$ are given in Section~\ref{Generalization}. 
For a given $0<\alpha<1$, we want to find a sequence $a_{n,h,\alpha}$ such that $\mathbb{P}(\mathcal{M}_h \subset \wh C_{n,h}(a_{n,h,\alpha})) \rightarrow 1-\alpha$, that is, $\wh C_{n,h}(a_{n,h,\alpha})$ is an asymptotic $100(1-\alpha)\%$ confidence region for $\mathcal{M}_h$. Let 
\begin{align}\label{Dnx}
&D_n(x) = \| Q(x)M(x)^T( d^2\wh f(x) - d^2 f_h(x))\|.
\end{align}
The following proposition indicates that the extreme value behaviors of $B_n(x)$ and $D_n(x)$ on $\mathcal{M}_h$ are close, and hence the sequence $a_{n,h,\alpha}$ can be determined by the extreme value distribution of $\sqrt{nh^{d+4}}\sup_{x\in\mathcal{M}_h} D_n(x)$.

\begin{proposition} \label{ridgenessest}
Under assumptions ({\bf F1}) - ({\bf F5}), ({\bf K1}), and ({\bf K2}), as $\gamma_{n,h}^{(2)}\rightarrow0$ and $h\rightarrow0$, we have that
\begin{align}
 & \sup_{x\in\mathcal{M}_h} D_n(x) = O_p (\gamma_{n,h}^{(2)}), \label{ridgenessest0}\\
&\sup_{x\in \mathcal{M}_h} B_n(x) - \sup_{x\in\mathcal{M}_h} D_n(x) 
 = O_p\left( (\gamma_{n,h}^{(2)})^2  + \gamma_{n,h}^{(1)} \right).\label{ridgenessest2}
\end{align}
%
\end{proposition}

\begin{remark}\label{equivl}

{\em 
%
When $r=1$, $m_i(x)$, $i=2,\cdots,d$ in (\ref{pseudoinv}) can be simplified to
\begin{align*}
m_i(x) = \frac{\|\nabla f(x)\|}{\lambda_i(x) - \lambda_1(x)}  D^T\left(v_i(x)\otimes  v_1(x)\right).
\end{align*}
Correspondingly, we can replace $\wh m_i(x)$ in $B_n(x)$ by 
\begin{align*}
\wt m_i(x) = \frac{\|\nabla \wh f(x)\|}{\wh\lambda_i(x) - \wh\lambda_1(x)}  D^T\left(\wh v_i(x)\otimes  \wh v_1(x)\right),
\end{align*}
and the conclusion in this proposition is not changed, following the same proof of this proposition.
}
\end{remark}

We need to find the asymptotic distribution of $ \sqrt{nh^{d+4}} \sup_{x\in\mathcal{M}_h}D_n(x)$. In particular, we will show that there exists $\beta_h$ such that for any $z\in\mathbb{R}$,
\begin{align*}
\mathbb{P} \left\{ \sqrt{2\log(h^{-1})} \left( \sqrt{nh^{d+4}} \sup_{x\in\mathcal{M}_h} D_n(x) -\beta_h\right) \leq z\right\} \rightarrow e^{-e^{-z}}.
\end{align*}
To this end, we will represent $\sqrt{nh^{d+4}}D_n(x)$ as an empirical process and approximate its supremum by the extreme value of a Gaussian process defined on a class of functions.

For $z\in\mathbb{R}^{d-r}\backslash\{0\}$, let $A(x,z) = M(x)Q(x)z.$ Notice that $\sqrt{f(x)}\|A(x,z)\|_{\mathbf{R}}=\|z\|$. Let $g_{x,z}(\cdot)=\frac{1}{\sqrt{h^{d}}}\left\langle A(x,z), d^2K\left(\frac{x-\cdot}{h}\right) \right\rangle$, and define the class of functions
\begin{align}\label{funcclass}
%
\mathcal{F}_h = \left\{g_{x,z}(\cdot):\;x\in\mathcal{M}_h, z\in \mathbb{S}^{d-r-1}\right\}. 
\end{align}
Consider the local empirical process $\{\mathbb{G}_n(g_{x,z}):\; g_{x,z}\in\mathcal{F}_h\}$, where
\begin{align*}
\mathbb{G}_n(g_{x,z}) = \frac{1}{\sqrt{n}}\sum_{i=1}^n\left[g_{x,z}(X_i)-\mathbb{E}g_{x,z}(X_1)\right].
\end{align*}
%
%
Due to the elementary result $\|v\|=\sup_{z\in\mathbb{S}^{d-r-1}} v^Tz$ for any $v\in \mathbb{R}^{d-r}$, it is clear that $\sqrt{nh^{d+4}} D_n(x) = \sup_{z\in\mathbb{S}^{d-r-1}} \mathbb{G}_n(g_{x,z})$. Hence
%
 %
%
%
%
%
%
%
\begin{align}\label{dnxequiv}
 \sqrt{nh^{d+4}}  \sup_{x\in\mathcal{M}_h}  D_n(x) = \sup_{g_{x,z}\in\mathcal{F}_h} \mathbb{G}_n(g_{x,z}).
\end{align}
Using similar arguments as given in Chernozhukov et al. (2014), the supremum of the empirical process in (\ref{dnxequiv}) can be approximated by the supremum of a Gaussian process, as shown in the following theorem. Let $\mathbb{B}$ be a centered Gaussian process on $\mathcal{F}_h$ such that for all $g_{x,z}, g_{\tilde x,\tilde z}\in\mathcal{F}_h$,
\begin{align*}
\mathbb{E}(\mathbb{B}(g_{x,z})\mathbb{B}(g_{\tilde x,\tilde z})) {=} \text{Cov}(g_{x,z}(X_1),\; g_{\tilde x,\tilde z}(X_1)). 
\end{align*}

\begin{theorem}\label{gassianapprox}
Under assumptions ({\bf F1}) - ({\bf F5}), ({\bf K1}), and ({\bf K2}), as $\gamma_{n,h}^{(0)}\log^{4}n\rightarrow0$ and $h\rightarrow0$ we have
\begin{align}\label{gaussapp}
\sup_{t>0} \left|\mathbb{P}\left( \sqrt{nh^{d+4}}  \sup_{x\in\mathcal{M}_h}  D_n(x) <t\right) - \mathbb{P}\left(\sup_{g\in\mathcal{F}_h}\mathbb{B}(g)<t\right)\right| = o(1).
%
\end{align}
\end{theorem}

\begin{remark}\label{smoothness}$\;$\\[-10pt]

{\em
(i) In the derivation of the asymptotic distribution of the maximal deviation of density function estimation, Bickel and Rosenblatt (1973) use a sequence of Gaussian approximations. When extending the idea to multivariate density function estimation, Rosenblatt (1976) imposes an assumption that requires $f$ to be $d$ times continuously differentiable in order to use the Rosenblatt transformation (1952). This type of assumption is further used in related work for Gaussian approximation to maximal deviation in multivariate function estimation (see, e.g., Konakov and Piterbarg, 1984). In fact if one is willing to impose a similar assumption in our context (that is, $f$ is $d+2$ times continuously differentiable, because ridges are defined using up to the second derivatives of $f$), then it can be verified that the Gaussian process $\mathbb{B}(g_{x,z})$ has the following representation:
\begin{align*}
\mathbb{B}(g_{x,z}) \stackrel{d}{=} \int_{\mathbb{R}^d} g_{x,z}(s) d\mathbf{B}(M(s)),
\end{align*}
where $\mathbf{B}$ is the d-dimensional Brownian bridge, and $M$ is the Rosenblatt transformation. Instead of following the approach in Rosenblatt (1976) using a sequence of Gaussian approximations (where the last one is stationary in the context of density estimation), we directly find out the limiting extreme value distribution of $\mathbb{B}(g_{x,z})$, which is shown to be  locally stationary (see Definition~\ref{def-loc-stat}). This allows us to use a less stringent smoothness condition on $f$. 
%
%

(ii) Let $w_{x}(\cdot)=\frac{1}{\sqrt{h^{d}}} Q(x)^T M(x)^Td^2K\left(\frac{x-\cdot}{h}\right)$, so that $g_{x,z}(\cdot)=z^Tw_{x}(\cdot)$. Also let $S_h(x)=(S_{1,h}(x),\cdots,S_{d-r,h}(x))^T$ be a vector of centered Gaussian fields indexed by $\mathcal{M}_h$ such that for $x,\wt x\in\mathcal{M}_h$, $\mathbb{E}(S_h(x)S_h(\wt x)^T)=\text{Cov}(w_{x}(X_1),w_{\tilde x}(X_1))$. Then it is clear that $\sup_{g\in\mathcal{F}_h}\mathbb{B}(g) = \sup_{x\in\mathcal{M}_h}\|S_h(x)\|$, where $\|S_h(x)\|^2\sim \chi_{d-r}^2$ for any fixed $x\in\mathcal{M}_h$, because $\text{Cov}(S_h(x),S_h(x))=\mathbf{I}_{d-r}$, i.e., $S_{1,h}(x),\cdots,S_{d-r,h}(x)$ are independent when $d-r\geq2$. Note the standardization in $S_{h}(x)$ is only pointwise, and if $x-\wt x=o(h)$ then $S_{i,h}(x)$ and $S_{j,h}(\wt x)$ are dependent in general for $i\neq j$ when $d-r\geq2$. Overall $\|S_h(x)\|^2$ is a $\chi^2$ field indexed by $\mathcal{M}_h$, as a sum of squares of Gaussian fields with {\em cross dependence}, whereas independence of the Gaussian fields is usually assumed in the literature of extreme value theory for $\chi^2$ fields (see e.g. Piterbarg, 1994). This dependence structure has an effect on the final extreme value distribution result (see remark after Theorem~\ref{confidenceregion}).
%
}
\end{remark}

To find the distribution of $\sup_{g\in\mathcal{F}_h}\mathbb{B}(g)$, it is critical to calculate the covariance structure of $\mathbb{B}(g_{x,z})$, $g_{x,z}\in\mathcal{F}_h$, and verify it has the desired properties to apply Theorem~\ref{ProbMain}. For $g_{x,z},g_{\tilde x,\tilde z}\in\mathcal{F}_h$ (which means $x,\wt x\in\mathcal{M}_h$ and $z,\wt z\in \mathbb{S}^{d-r-1}$), let $r_h(x,\wt x,z,\wt z)$ be the correlation coefficient of $\mathbb{B}(g_{x,z})$ and $\mathbb{B}(g_{\tilde x,\tilde z})$.

\begin{proposition}\label{covstructure}
Let $\Delta x = \wt x - x$ and $\Delta z = \wt z-z$. Under assumptions ({\bf F1}) - ({\bf F5}), ({\bf K1}), and ({\bf K2}), as $h\rightarrow0$, $\Delta z\rightarrow0$ and $\Delta x/h\rightarrow 0$, we have
\begin{align}\label{covariancestructure}
r_h(x,\wt x,z,\wt z) = 1 -\frac{1}{2} \|\Delta z\|^2 - \frac{1}{2h^2} \Delta x^T\Omega(x,z)\Delta x    + o\left(\left\| \frac{\Delta x}{h}\right\|^2 + \|\Delta z\|^2 \right),
\end{align}
where 
\begin{align}\label{omegaxz}
\Omega(x,z) = \int_{\mathbb{R}^d} \nabla d^2 K(u)^T A(x,z) A(x,z)^T \nabla d^2 K(u) du,
\end{align}
and the o-term in (\ref{covariancestructure}) is uniform in $x,\wt x\in\mathcal{M}_h,$ $z,\wt z\in \mathbb{S}^{d-r-1}$, and $h\in(0,h_0]$ for some $h_0>0$.
\end{proposition}

\begin{remark}\label{dr1}
{\em
When $d-r=1$, we have $z,\wt z\in\{1,-1\}$ and $\Delta z\equiv0$, and then (\ref{covariancestructure}) should be understood as
\begin{align}\label{covariancestructure2}
r_h(x,x+\Delta x,z,z+\Delta z) = 1 - \frac{1}{2h^2} \Delta x^T\Omega(x,1)\Delta x    + o\left(\left\| \frac{\Delta x}{h}\right\|^2 \right).
\end{align}
}
\end{remark}


To construct a confidence region for $\mathcal{M}_h$, we will use the distribution of $\sup_{g\in\mathcal{F}_h}\mathbb{B}(g)$, which is also the extreme value distribution of $\chi$-fields indexed by manifolds, as indicated in Remark~\ref{smoothness}(ii). The distribution depends on the geometry of the manifolds, through a surface integral on the manifolds specifically defined below, which is originated from Theorem~\ref{ProbMain}.  For a differentiable submanifold $\mathcal{S}$ of $\mathbb{R}^d$, at each $u\in\mathcal{S},$ let $T_u\mathcal{S}$ denote the tangent space at $u$ to $\mathcal{S}$. Let $\Lambda(T_u\mathcal{S})$ be a matrix with orthonormal columns that span $T_u\mathcal{S}$. For an  $n \times r$ matrix $M$ with $ r \le n,$ we denote by $\|M\|^2_r$ the sum of squares of all minor determinants of order $r$. For any nice (meaning the following is well-defined) set $\mathcal{A}\subset\mathcal{H}$, define 
\begin{align*}
c_h^{(d,r)}(\mathcal{A}) = \log\left\{ \frac{r^{(d-2)/2}}{2\pi^{d/2}} \int_{\mathbb{S}^{d-r-1}}  \int_{\mathcal{M}_h\cap \mathcal{A}} \|\Omega(x,z)^{1/2} \Lambda({T_x\mathcal{M}_h})\|_r d\mathscr{H}_r(x)d\mathscr{H}_{d-r-1}(z) \right\}.
\end{align*}
For simplicity we write $c_h^{(d,r)} = c_h^{(d,r)}(\mathcal{H}) $. 
%
%
For $z,c\in\mathbb{R}$, let
\begin{align}\label{bhz}
b_h(z,c) = \frac{z}{\sqrt{2r\log{(h^{-1})}}} +  \sqrt{2r\log(h^{-1})} + \frac{1}{\sqrt{2r\log(h^{-1})}} \left[  \frac{d-2}{2} \log\log(h^{-1}) +c \right].
\end{align}
For any $\alpha\in(0,1)$, let $z_\alpha = - \log[-\log(1-\alpha)]$ so that $e^{-e^{-z_\alpha}}=1-\alpha$. The following theorem gives an asymptotic confidence region for $\mathcal{M}_h$.

\begin{theorem}\label{confidenceregion}
Under assumptions ({\bf F1}) - ({\bf F5}), ({\bf K1}) - ({\bf K3}), as $\gamma_{n,h}^{(2)}\log{n}\rightarrow0$ and $h\rightarrow0$, we have
\begin{align}\label{dnxextreme}
\lim_{h\rightarrow 0} \mathbb{P}\left(\sup_{g\in\mathcal{F}_h}\mathbb{B}(g)< b_h(z,c_h^{(d,r)}) \right) = e^{-e^{-z}}.
\end{align}
This implies that for any $\alpha\in(0,1)$, as $n\rightarrow\infty$,
\begin{align}\label{asympconf}
\mathbb{P}\left(\mathcal{M}_h \subset \wh C_{n,h}( b_h(z_\alpha,c_h^{(d,r)}))\right) \rightarrow 1-\alpha, 
\end{align}
where $\wh C_{n,h}$ is defined in (\ref{hatCnh}). 

\end{theorem}

\begin{remark}
{\em We discuss the quantity $c_h^{(d,r)}$. When $d-r = 1$, we have $\mathbb{S}^0=\{-1,1\}$ and $\mathscr{H}_0$ is the counting measure, and so
\begin{align*}
c_h^{(d,r)} = \log\left\{ \frac{r^{(d-2)/2}}{\pi^{d/2}}  \int_{\mathcal{M}_h} \|\Omega(x,1)^{1/2} \Lambda(T_x\mathcal{M}_h)\|_r d\mathscr{H}_r(x) \right\} .
\end{align*}
When $d-r \geq 2$, for any $x\in\mathcal{M}_h$, $ \int_{\mathbb{S}^{d-r-1}} \|\Omega(x,z)^{1/2} \Lambda(T_x\mathcal{M}_h)\|_r d\mathscr{H}_{d-r-1}(z)$ is a hyperelliptic integral. Note the cross dependence in the Gaussian fields discussed in Remark~\ref{smoothness}(ii) is also reflected in $c_h^{(d,r)}$, where the integrals on $\mathcal{M}_h$ and $\mathbb{S}^{d-r-1}$ are not independent. } 
\end{remark}

The confidence regions for $\mathcal{M}_h$ given in (\ref{asympconf}) is a theoretical result depending on the unknown quantity $c_h^{(d,r)}$. In what follows we address a few important questions: (i) confidence regions for $\mathcal{M}$ by correcting the bias (Section~\ref{biascorrect}); (ii) data-driven confidence regions for $\mathcal{M}_h$ and $\mathcal{M}$ by consistently estimating $c_h^{(d,r)}$ (Section~\ref{underconfidenceregion}); (iii) different choices of $b_n$ and modified confidence regions for $\mathcal{M}_h$ and $\mathcal{M}$ when assumption ({\bf F5}) is relaxed (Section~\ref{Generalization}).

\subsection{Asymptotic confidence regions for $\mathcal{M}$}\label{biascorrect} 
We consider asymptotic confidence regions for $\mathcal{M}$ in this section. The difference between $\mathcal{M}$ and $\mathcal{M}_h$ is attributed to the bias in kernel type estimation. In Section~\ref{mhconfden} we focused on $\mathcal{M}_h$ by only considering the stochastic variation $B_n$, which is of order $O_p(\gamma_{n,h}^{(2)})$. As we show in Lemma~\ref{biasasymptotic} below, the bias part in ridge estimation is of order $O(h^2)$. Usually there are two approaches to deal with the bias in kernel type estimation: implicit bias correction using undersmoothing bandwidth and explicit bias correction (see e.g. Hall, 1992). The former makes the bias asymptotically negligible compared with the stochastic variation in the estimation, while the latter directly debiases the estimator by estimating the higher order derivatives in the leading terms of the bias using additional kernel estimation, which also means that the latter usually requires stronger assumptions on the smoothness of the underlying functions (see e.g. Xia, 1998). We use both methods to construct asymptotic confidence regions for $\mathcal{M}$.

The next lemma gives the asymptotic form of the bias in ridge estimation. Let $\mu_K = \int_{\mathbb{R}^d} s_1^2 K(s)ds$, where $s=(s_1,\cdots,s_d)^T$. Let $\Delta_L$ be the Laplacian operator, that is $\Delta_L \xi(x)=\sum_{i=1}^d\frac{\partial^2 \xi(x) }{\partial x_i^2}$, for a twice differentiable function $\xi$ on $\mathbb{R}^d$. If $\xi$ is a vector-valued function, then $\Delta_L$ applies to each element of $\xi$.

\begin{lemma}\label{biasasymptotic}
Under assumptions ({\bf F1}) - ({\bf F4}) and ({\bf K1}), as $h\rightarrow0$, 
\begin{align*}
V_h(x)^T\nabla f_h(x) - V(x)^T\nabla f(x) 
= \frac{1}{2}h^2\mu_K \beta(x) + R_h.
\end{align*}
where $\beta(x) = \{M(x)^T[\Delta_Ld^2 f(x)]\nabla f(x) + V(x)^T [\Delta_L\nabla f(x)]\}$ and $R_h=o(h^2),$ uniformly in $x\in\mathcal{N}_{\delta_0}(\mathcal{M})$. When both $f$ and $K$ are six times continuously differentiable, we have $R_h=O(h^4)$, uniformly in $x\in\mathcal{N}_{\delta_0}(\mathcal{M})$. 
\end{lemma}

Undersmoothing requires to use a small bandwidth $h$ such that $\gamma_{n,h}^{(4)}\rightarrow\infty$. One can also explicitly correct the bias by using a debiased estimator. For a bandwidth $l>0$, let
\begin{align*}
\wh \beta_{n,l}(x) = \{\wh M_{n,l}(x)^T[\Delta_Ld^2 \wh f_{n,l}(x)]\nabla \wh f_{n,l}(x) + \wh V_{n,l}(x)^T [\Delta_L\nabla \wh f_{n,l}(x)]\}, 
\end{align*}
where we have brought the subscripts $n,l$ back to the kernel estimators to show their dependence on a different bandwidth $l$. For $a_n\geq0$ and $b_n\in\mathbb{R}$, let
\begin{align}\label{hatCnhbc}
&\wh C_{n,h,l}^{\text{bc}}(a_n,b_n) \nonumber\\
=&\left\{x\in\mathcal{H}:\;   \sqrt{nh^{d+4}} \left\| Q_n(x)\left[\wh V(x)^T \nabla \wh f(x) -\frac{1}{2}h^2\mu_K\wh \beta_{n,l}(x) \right] \right\| \leq a_n, \text{and}\; \wh\lambda_{r+1}(x) <b_n \right\},
\end{align}
and denote $\wh C_{n,h,l}^{\text{bc}}(a_n)=\wh C_{n,h,l}^{\text{bc}}(a_n,0)$ for simplicity. Define
\begin{align*}
c^{(d,r)} = \log\left\{ \frac{r^{(d-2)/2}}{2\pi^{d/2}} \int_{\mathbb{S}^{d-r-1}}  \int_{\mathcal{M}} \|\Omega(x,z)^{1/2} \Lambda(T_x\mathcal{M})\|_r d\mathscr{H}_r(x)d\mathscr{H}_{d-r-1}(z) \right\},
%
\end{align*}
where we simply replace the domain of integration $\mathcal{M}_h$ by $\mathcal{M}$ in $c_h^{(d,r)}$.

\begin{theorem}\label{underconfidenceregion}
Suppose assumptions ({\bf F1}) - ({\bf F5}), ({\bf K1}) - ({\bf K3}) hold and $\gamma_{n,h}^{(2)}\log{n}\rightarrow0$ and $h\rightarrow0$, for any $\alpha\in(0,1)$ we have the following.\\
(i) {\em Undersmoothing}: As $\gamma_{n,h}^{(4)}\rightarrow\infty$,
\begin{align}\label{asympconfbias}
\mathbb{P}\left(\mathcal{M} \subset \wh C_{n,h}( b_h(z_\alpha,c^{(d,r)}))\right) \rightarrow 1-\alpha.
\end{align}
(ii) {\em Explicit bias correction}: Assume both $f$ and $K$ are six times continuously differentiable. As $h/l\rightarrow0$ and $\gamma_{n,h}^{(4)}/l^2\rightarrow\infty$,
\begin{align}\label{asympconfbiascorr}
\mathbb{P}\left(\mathcal{M} \subset \wh C_{n,h,l}^\text{bc}( b_h(z_\alpha,c^{(d,r)}))\right) \rightarrow 1-\alpha.
\end{align}
\end{theorem}

\begin{remark}\label{hausdorffbias}
{\em
We emphasize that the method in (i) is feasible here because we only require $\gamma_{n,h}^{(2)}\log{n}\rightarrow0$ and $h\rightarrow0$ for $n$ and $h$ in the results in Section~\ref{mhconfden}. As a comparison, the Hausdorff distance based approach for $\mathcal{M}_h$ developed in Chen et al. (2015) requires a oversmoothing bandwidth such that $\gamma_{n,h}^{(4)}\rightarrow0$, which implies that the bias dominates the stochastic variation in ridge estimation using the Hausdorff distance if a second order kernel is used, and hence the implicit bias correction approach using undersmoothing bandwidth is not applicable in their method. }
\end{remark}

\subsection{Estimating the unknowns}\label{unknownest}

The surface integrals $c_h^{(d,r)}$ and $c^{(d,r)}$ are unknown quantities that need to be estimated in order to make the confidence regions in Theorems~\ref{confidenceregion} and \ref{underconfidenceregion} computable with data. For a bandwidth $l>0$, we use the following plug-in estimators. 
Let 
\begin{align*}
&\wh A_{n,l}(x,z) = \wh M_{n,l}(x)[\wh f_{n,l}(x)\wh \Sigma_{n,l}(x)]^{-1/2}z,\\
&\wh\Omega_{n,l}(x,z) = \int_{\mathbb{R}^d} \nabla d^2 K(u)^T \wh A_{n,l}(x,z) \wh A_{n,l}(x,z)^T \nabla d^2 K(u) du,\\
& \wh{\mathcal{M}}_{n,l} = \{x\in\mathcal{H}:\; \wh V_{n,l}(x)^T\nabla \wh f_{n,l}(x) = 0, \; \wh\lambda_{r+1,n,l}(x) < 0\}.
%
%
\end{align*}
Note that the bandwidth $l$ here is not necessarily the same one as used for explicit bias correction in Section~\ref{biasasymptotic}. But we do need a similar condition for them so the same bandwidth $l$ is used for simplicity. For any nice set $\mathcal{A}\subset\mathcal{H}$, let 
\begin{align*}
\wh c_{n,l}^{(d,r)}(\mathcal{A}) = \log\left\{ \frac{r^{(d-2)/2}}{2\pi^{d/2}} \int_{\mathbb{S}^{d-r-1}}  \int_{\wh{\mathcal{M}}_{n,l}\cap \mathcal{A} } \|\wh\Omega_{n,l}(x,z)^{1/2} \Lambda(T_x\wh{\mathcal{M}}_{n,l})\|_r d\mathscr{H}_r(x)d\mathscr{H}_{d-r-1}(z) \right\}.
\end{align*} 
For simplicity we denote $\wh c_{n,l}^{(d,r)} = \wh c_{n,l}^{(d,r)}(\mathcal{H})$. To prove the confidence regions for $\mathcal{M}_h$ and $\mathcal{M}$ are still valid after replacing $b_h(z_\alpha,c_h^{(d,r)})$ and $b_h(z_\alpha,c^{(d,r)})$ by $b_h(z,\wh c_{n,l}^{(d,r)})$, we need to show that $\wh c_{n,l}^{(d,r)}$ is a consistent estimator of $ c_h^{(d,r)}$ and $c^{(d,r)}$. The proof uses similar ideas as in Qiao (2019a), who focuses on surface integral estimation on density level sets, which are $(d-1)$-dimensional manifolds embedded in $\mathbb{R}^d$. Since we view density ridges as intersections of $d-r$ level sets (in a broad sense to include $d-r=1$; see the introduction), the methods in Qiao (2019a) are extended in our proof. The data-driven confidence regions are given in the following corollary. 

\begin{corollary}\label{pluginregion}
Suppose assumptions ({\bf F1}) - ({\bf F5}), ({\bf K1}) - ({\bf K3}) hold. Also assume that $\gamma_{n,h}^{(2)}\log{n}\rightarrow0$, $\gamma_{n,l}^{(4)}\rightarrow0$, $h\rightarrow0$ and $l\rightarrow 0$. For any $\alpha\in(0,1)$ we have the following.\\
(i) {\em For $\mathcal{M}_h$}:
\begin{align}\label{asympconfin}
\mathbb{P}\left(\mathcal{M}_h \subset \wh C_{n,h}( b_h(z_\alpha, \wh c_{n,l}^{(d,r)} ))\right) \rightarrow 1-\alpha.
\end{align}
(ii) {\em For $\mathcal{M}$ using undersmoothing}: as $\gamma_{n,h}^{(4)}\rightarrow\infty$,
\begin{align}\label{asympconfbiasin}
\mathbb{P}\left(\mathcal{M} \subset \wh C_{n,h}( b_h(z_\alpha, \wh c_{n,l}^{(d,r)} ))\right) \rightarrow 1-\alpha.
\end{align}
(iii) {\em For $\mathcal{M}$ using explicit bias correction}: Assume that both $f$ and $K$ are six time continuously differentiable. As $h/l\rightarrow0$ and $\gamma_{n,h}^{(4)}/l^2\rightarrow\infty$, 
\begin{align}\label{asympconfbiascorrin}
\mathbb{P}\left(\mathcal{M} \subset \wh C_{n,h,l}^\text{bc}( b_h(z_\alpha,\wh c_{n,l}^{(d,r)} ))\right) \rightarrow 1-\alpha.
\end{align}

\end{corollary}


\subsection{Further improvements related to eigenvalues and critical points}\label{Generalization}

We have considered the confidence regions in the form of $\wh C_{n,h}(a_n,b_n)$ defined in (\ref{hatCnh}) and $\wh C_{n,h,l}^{\text{bc}}(a_n,b_n)$ defined in (\ref{hatCnhbc}) for some $a_n>0$ and $b_n=0$. So far our main focus has been on the determination of $a_n$, after the justification for the choice $b_n=0$ given in Remark~\ref{zerob}. In fact, one can use some nonpositive $b_n$ as the upper bound of $\wh\lambda_{r+1}$, to potentially make the confidence regions more efficient. This is because $\sup_{x\in\mathcal{M}}\lambda_{r+1}(x)$ is strictly bounded away 0 under assumption (\textbf{F3}), which allows us to choose a nonpositive $b_n$ such that $\sup_{x\in\mathcal{M}}\wh\lambda_{r+1}(x)<b_n$ holds with probability tending to one under our assumptions, as stated in the following proposition. 
%
%
For $a,b\in\mathbb{R}$, denote $a\wedge b=\min(a,b)$. Let $\nu_n$ be a sequence such that $\nu_n\rightarrow\infty$ and define 
\begin{align}
&\zeta_n^0 = \left[\sup_{x\in\wh{\mathcal{M}}} \wh{\lambda}_{r+1}(x) + \nu_n\gamma_{n,h}^{(2)}\right]\wedge 0,\\
&\zeta_n =  \left[\sup_{x\in\wh{\mathcal{M}}} \wh{\lambda}_{r+1}(x) + \nu_n\left(\gamma_{n,h}^{(2)} +h^2\right)\right]\wedge 0.
\end{align}

\begin{proposition}\label{lambdaprop}
Suppose assumptions ({\bf F1}) - ({\bf F4}), ({\bf K1}) - ({\bf K3}) hold. Also assume that $\gamma_{n,h}^{(2)}\rightarrow0$ for $d-r=1$ and $\gamma_{n,h}^{(3)}\rightarrow0$ for $d-r\geq 2$. Then we have
\begin{align}
&\mathbb{P}\left(\sup_{x\in\mathcal{M}_h} \wh\lambda_{h,r+1}(x) \geq \zeta_n^0 \right) \rightarrow 0,\label{lambdah}\\
&\mathbb{P}\left(\sup_{x\in\mathcal{M}} \wh\lambda_{r+1}(x) \geq \zeta_n \right) \rightarrow 0 \label{lambdanh}.
\end{align}
\end{proposition}
\begin{remark}\label{bnnew}
{\em
The result in Proposition~\ref{lambdaprop} immediately implies that we can use $\zeta_n^0$ to replace 0 as $b_n$ in the confidence regions we construct in Corollary~\ref{pluginregion} for $\mathcal{M}_h$ (and use $\zeta_n$ for $\mathcal{M}$), if we additionally assume $\gamma_{n,h}^{(3)}\rightarrow0$ for $d-r\geq 2$. }
\end{remark}

So far we have imposed assumption ({\bf F5}) to exclude critical points on ridges from our consideration. The reason is that the estimation of critical points behaves differently from regular points in our approach. Below we remove assumption ({\bf F5}), that is, we allow the existence of points $x$ such that $\|\nabla f(x)\| = 0$ on $\mathcal{M}$. 
For $0<\eta<1$, let $\mathcal{K}_{h,\eta} = \{x\in\mathcal{H}:\; \|\nabla f_h(x)\| \leq h^{\eta}\}.$ Note that $\mathcal{M}_h = (\mathcal{M}_h \cap \mathcal{K}_{h,\eta})\cup (\mathcal{M}_h \cap \mathcal{K}_{h,\eta}^\complement)$. When $h$ is small, the set $\mathcal{M}_h \cap \mathcal{K}_{h,\eta}$ is a small neighborhood near all the critical points on the ridge $\mathcal{M}_h$, and $\mathcal{M}_h \cap\mathcal{K}_{h,\eta}^\complement$ is the set of remaining points on the ridge. Our strategy is to construct two regions to cover $\mathcal{M}_h \cap \mathcal{K}_{h,\eta}$ and $\mathcal{M}_h \cap \mathcal{K}_{h,\eta}^\complement$ separately and then combine them. 
For a sequence $\mu_n\rightarrow\infty$ such that $h\mu_n\rightarrow 0$, let $\mathcal{E}_{n,\eta} = \{x\in\mathcal{H}:\; \|\nabla \wh f(x)\| \leq \mu_n\gamma_{n,h}^{(1)} + h^\eta \}$ and
\begin{align*}
&\mathcal{G}_{n,\eta}^0 = \mathcal{E}_{n,\eta}\cap \{x\in\mathcal{H}:\; \wh\lambda_{r+1}(x)<\zeta_{n}^0\},\\
&\mathcal{G}_{n,\eta} = \mathcal{E}_{n,\eta}\cap \{x\in\mathcal{H}:\; \wh\lambda_{r+1}(x)<\zeta_{n}\}.
\end{align*}
%
%
Then $\mathcal{G}_{n,\eta}^0$ covers $\mathcal{M}_h \cap \mathcal{K}_{h,\eta}$ with a large probability. The following theorem gives the confidence regions for $\mathcal{M}_h$ and $\mathcal{M}$ without the assumption ({\bf F5}), where we also incorporate a new choice for $b_n$ as discussed in Remark~\ref{bnnew}.

\begin{theorem}\label{generalization}
Suppose assumptions ({\bf F1}) - ({\bf F4}), ({\bf K1}) - ({\bf K3}) hold. Also we assume that $\gamma_{n,h}^{(2)}\log{n}\rightarrow0$ for $d-r=1$ and $\gamma_{n,h}^{(3)}\rightarrow0$ for $d-r\geq 2$; $\gamma_{n,l}^{(4)}\rightarrow0$ and $l\rightarrow0$. Suppose $0<\eta<1$, $\nu_n\rightarrow\infty$, $\mu_n\rightarrow\infty$ and $h\mu_n\rightarrow 0$. For any $\alpha\in(0,1)$ we have the following.\\ 
(i) {\em For $\mathcal{M}_h$}:
\begin{align}
\mathbb{P}\left( \mathcal{M}_h \subset [\wh C_{n,h}(b_{h}(z_\alpha,\wh c_{h,l}^{(d,r)}( \mathcal{E}_{n,\eta}^{\complement})),\;\zeta_n^0) \cup \mathcal{G}_{n,\eta}^0 ] \right) \rightarrow 1-\alpha. \label{confregzero2}
\end{align}
(ii) {\em For $\mathcal{M}$ using undersmoothing}: as $\gamma_{n,h}^{(4)}\rightarrow\infty$,
\begin{align}
\mathbb{P}\left( \mathcal{M} \subset [\wh C_{n,h}(b_{h}(z_\alpha,\wh c_{n,l}^{(d,r)}( \mathcal{E}_{n,\eta}^\complement)),\;\zeta_n) \cup \mathcal{G}_{n,\eta} ] \right) \rightarrow 1-\alpha.
\end{align}
(iii) {\em For $\mathcal{M}$ using explicit bias correction}: Assume that both $f$ and $K$ are six time continuously differentiable. As $h/l\rightarrow0$ and $\gamma_{n,h}^{(4)}/l^2\rightarrow\infty$, 
\begin{align}
\mathbb{P}\left( \mathcal{M} \subset [\wh C_{n,h,l}^{\text{bc}}(b_{h}(z_\alpha,\wh c_{n,l}^{(d,r)}( \mathcal{E}_{n,\eta}^\complement)),\;\zeta_n) \cup \mathcal{G}_{n,\eta} ] \right) \rightarrow 1-\alpha.
\end{align}

\end{theorem}

\begin{remark}$\;$\\[-10pt]

{
\em
(i) The results in this theorem still hold if $\zeta_n^0$ and $\zeta_n$ are replaced by 0 as discussed in Remark~\ref{zerob}. 

(ii) We use two sequences $\mu_n\rightarrow\infty$ and $\nu_n\rightarrow\infty$ in the construction of the confidence regions. One may choose $\mu_n=h^{-\mu}$ and $\nu_n=h^{-\nu}$ for some $0<\mu<1$ and $\nu>0$ to satisfy the assumptions in the theorem. The need for using these tuning parameters $\mu$, $\nu$ as well as $\eta$ also reflects the fact that the concept of ridges involves multiple components, i.e., the eigenvectors and eigenvalues of the Hessian and the gradient (see (\ref{condition1}) and (\ref{condition2})). These components have different asymptotic behaviors and roles in the estimation. The way that we use the tuning parameters captures the gradient and eigenvalues with a high probability, which allows us to be dedicated to the asymptotic behaviors related to the eigenvectors. See Section~\ref{discusssec} for the discussion of an alternative approach.

(iii) The assumption ({\bf F2}) implies that all the critical points on ridges are isolated, but it is clear from the proof that the result in this theorem also holds when the condition in assumption ({\bf F2}) is weakened to hold only on regular points on ridges. For example, the ridge may have flat parts and $\mathcal{E}_{n,\eta}$ is envisioned as tubes around these flat parts. In this case, however, it may be worth considering finer confidence regions for these flat parts by treating them as level sets of the gradient and using similar ideas as we do for regular ridge points.
%
}
\end{remark}

\section{Discussion}\label{discusssec}
In this manuscript we develop asymptotic confidence regions for density ridges. We treat ridges as the intersections of some level sets and use the VV based approach. The construction of our confidence regions is based on Gaussian approximation of suprema of empirical processes and the extreme value distribution of suprema of $\chi$-fields indexed by manifolds. It is known that the rate of convergence of this type of extreme value distribution is slow. As an alternative approach, we are working on developing a bootstrap procedure using the VV idea for the confidence regions. 

Apparently our approach can also be used for the construction of confidence regions for the intersections of multiple functions in general (such as density function and regression functions). It's known that estimating such intersections has applications in econometrics. See, e.g. Bugni (2010).

Ridges points are defined through the two conditions given in (\ref{condition1}) and (\ref{condition2}), which have different roles in the construction of our confidence regions. We choose a conservative way to dealing with the uncertainty in the estimation related to condition (\ref{condition2}), which allows us to focus on quantifying the uncertainty in estimating condition (\ref{condition1}). Alternatively, one may also find the marginal distributions in estimating these two conditions and combine them using the Bonferroni method, or even find their joint distributions to construct confidence regions. We leave the exploration of this idea to future work. \\ 

%
%
\section{Proofs}\label{proofssec}
{\bf Proof of Lemma~\ref{lambda2bound}}
\begin{proof}
Note that with assumption ({\bf F4}) both $\lambda_{r+1}$ and $V^T\nabla f$ are continuous functions on $\mathcal{N}_{\delta_0}(\mathcal{M})$. Under assumption ({\bf F3}), we can write $\mathcal{M} = \{x\in\mathcal{H}:\; V(x)^T\nabla f(x)=0,\; \lambda_{r+1}(x)\leq 0\}$, which is a compact set. The claim that $\mathcal{M}$ is $r$-dimensional manifold is a consequence of the implicit function theorem. Under assumption ({\bf F2}), the claim that $\mathcal{M}$ has positive reach follows from Theorem 4.12 in Federer (1959). This is the assertion (i).

Next we show assertion (ii). 
%
Let $\delta_{\text{gap}}:=\inf_{x\in\mathcal{H}}[\lambda_r(x) - \lambda_{r+1}(x)].$ Since $\mathcal{H}$ is compact, $\delta_{\text{gap}}>0$ due to assumption (\textbf{F4}). Lemma~\ref{lambda2} implies that $\inf_{x\in\mathcal{H}}[\lambda_{r,h}(x) - \lambda_{r+1}(x)] = \delta_{\text{gap}} + O(h^2) \geq \frac{1}{2}\delta_{\text{gap}}$ when $h$ is small enough. Then using Davis-Kahan theorem (von Luxburg, 2007) leads to 
\begin{align}\label{vvdiff}
\sup_{x\in\mathcal{H}} \| V(x)V(x)^T - V_h(x)V_h(x)^T \|_F \leq \frac{2\sqrt{2} \sup_{x\in\mathcal{H}} \|\nabla^2 f(x) - \nabla^2 f_h(x)\|_F}{\delta_{\text{gap}}} = O(h^2),
\end{align}
by using Lemma~\ref{lambda2}.

Noticing that $V(x)^TV(x) = \mathbf{I}_{d-r}(x)$, we can write
\begin{align}\label{vnfmh}
\sup_{x\in\mathcal{M}_h} \|V(x)^T\nabla f(x)\| & = \sup_{x\in\mathcal{M}_h} \|V(x)^T[V(x)V(x)^T\nabla f(x) - V_h(x)V_h(x)^T\nabla f_h(x)]\| \nonumber\\
& \leq \sqrt{d-r} \sup_{x\in\mathcal{H}} \|V(x)V(x)^T\nabla f(x) - V_h(x)V_h(x)^T\nabla f_h(x)\| \nonumber\\
&= O(h^2),
\end{align}
where we use (\ref{vvdiff}) and Lemma~\ref{lambda2}.

Let $\mathcal{M}^{(1)} = \{x\in\mathcal{H}: V(x)^T\nabla f(x)=0\}$ and $\mathcal{M}^{(2)} = \{x\in\mathcal{H}: \lambda_{r+1}<0\}$. Then $\mathcal{M} = \mathcal{M}^{(1)}\cap \mathcal{M}^{(2)}$. For any $\delta>0$, let $\mathcal{N}_{\delta}(\mathcal{M}^{(1)})=\{x\in\mathcal{H}:\; \|V(x)^T\nabla f(x)\| \leq\delta\}$. Note that (\ref{vnfmh}) implies that for any fixed $0<\delta\leq\delta_0$, $\mathcal{M}_h \subset \mathcal{N}_{\delta}(\mathcal{M}^{(1)})$ when $h$ is small enough.

It suffices to show $\mathcal{M}_h\subset \mathcal{M}^{(2)} $, when $h$ is small enough. Since $\mathcal{M}^{(1)}$ is a compact set and due to ({\bf F3}), there exists $\beta_0>0$ such that $\inf_{x\in \mathcal{M}^{(1)}} |\lambda_{r+1}(x)| \geq 4\beta_0.$ There exists $\delta_1$ with $0<\delta_1\leq\delta_0$ such that $\inf_{x\in \mathcal{N}_{\delta_1}(\mathcal{M}^{(1)})} |\lambda_{r+1}(x)| \geq 2\beta_0,$ which further implies that $\inf_{x\in \mathcal{M}_h} |\lambda_{r+1}(x)| \geq 2\beta_0,$ when $h$ is small enough. Then we must have $\sup_{x\in \mathcal{M}_h} \lambda_{r+1}(x) \leq -2\beta_0,$ since if there exists $x_0\in\mathcal{M}_h$ such that $\lambda_{r+1}(x_0) \geq 2\beta_0$, then Lemma~\ref{lambda2} would lead to 
%
\begin{align*}
\lambda_{r+1,h}(x_0)  \geq  \lambda_{r+1}(x_0) - |\lambda_{r+1}(x_0) - \lambda_{r+1,h}(x_0)| \geq 2\beta_0 + O(h^2) \geq \beta_0,
\end{align*}
when $h$ is small, which contradicts the definition of $\mathcal{M}_h$. Hence $\mathcal{M}_h\subset [\mathcal{N}_{\delta_0}(\mathcal{M}^{(1)})\cap \mathcal{M}^{(2)}] = \mathcal{N}_{\delta_0}(\mathcal{M})$, when $h$ is small enough. This is the assertion (ii). 

For assertion (iii), we have shown $\sup_{x\in \mathcal{M}_h} \lambda_{r+1,h}(x) \leq -\beta_0$ above. Using a similar argument, we get $\inf_{x\in \mathcal{M}_h} [\lambda_{j-1,h}(x) - \lambda_{j,h}(x) ] > \beta_0$, $j=r+1,\cdots,d$, when $h$ is small. 

To show that $\mathcal{M}_h$ is an $r$-dimensional manifold and has positive reach when $h$ is small, we use a similar argument as given in the proof of assertion (i). The key is to show that $f_h$ satisfies a similar property as in the assumption ({\bf F2}) for $f$, when $h$ is small. Let $l_i(x) = \nabla(\nabla f(x)^T v_{r+i}(x))$, $i=1,\cdots,d-r$ and $L(x)=(l_{1}(x),\cdots,l_{d-r}(x))$. Then ({\bf F2}) is equivalent to $\inf_{x\in\mathcal{M}} \text{det}(L(x)^TL(x)) >0$. Since $\mathcal{N}_{\delta_0}(\mathcal{M})$ is a compact set (or we can replace $\delta_0$ by a smaller value if necessary), we can find $\epsilon_0>0$ such that
\begin{align}\label{epsilon0const}
\inf_{x\in\mathcal{N}_{\delta_0}(\mathcal{M})} \text{det}(L(x)^TL(x)) \geq \epsilon_0.
\end{align}
Let $l_{i,h}(x) = \nabla(\nabla f_h(x)^T v_{h,r+i}(x))$, $i=1,\cdots,d-r$ and $L_h(x)=(l_{h,1}(x),\cdots,l_{h,d-r}(x))$. With (\ref{epsilon0const}) we have
\begin{align}
& \inf_{x\in\mathcal{N}_{\delta_0}(\mathcal{M})} \text{det}(L_h(x)^TL_h(x)) \nonumber \\
 \geq & \inf_{x\in\mathcal{N}_{\delta_0}(\mathcal{M})} \text{det}(L(x)^TL(x)) - \sup_{x\in\mathcal{N}_{\delta_0}(\mathcal{M})} | \text{det}(L(x)^TL(x))- \text{det}(L_h(x)^TL_h(x))| \nonumber\\
\geq & \epsilon_0 - O(h^2),\label{marginassump}
\end{align}
where we use Lemma~\ref{lambda2} given above and Theorem 3.3 of Ipsen and Rehman (2008), the latter gives a perturbation bound for matrix determinants. This then implies that there exists $\epsilon_1>0$ such that for $h$ small enough, $\inf_{x\in\mathcal{N}_{\delta_0}}\|l_{i,h}(x) \|>\epsilon_1.$ Following Lemma~\ref{lambda2}, $l_{i,h}(x)$ has a Lipschitz constant $C<\infty$ for all $x\in\mathcal{N}_{\delta_0}$. Also it is clear that there exists $\delta_1>0$ such that $\mathcal{M}_h\oplus\delta_1\subset \mathcal{N}_{\delta_0}(\mathcal{M})$ when $h$ is small enough. Define sets $\mathcal{M}_{i,h}=\{x\in\mathbb{R}^d: \; \nabla f_h(x)^T v_{h,r+i}(x) =0, \;  \lambda_{r+1,h}(x)<0\}.$ Note that $\mathcal{M}_{h}=\cap_{i=1}^{d-r} \mathcal{M}_{i,h}.$ Using Lemma 4.11 in Federer (1959), when $h$ is small enough such that $\mathcal{M}_h\subset \mathcal{N}_{\delta_0}(\mathcal{M})$, we have for $i=1,\cdots,d-r$,
\begin{align}\label{reachsingle}
\inf_{u\in \mathcal{M}_{h}} \Delta(\mathcal{M}_{i,h},u) \geq \min(\delta_1/2,\;  \epsilon_1/C).
\end{align}
Using (\ref{reachsingle}) and Theorem 4.10 in Federer (1959), we conclude the assertion (iv) by using a deduction argument, similar to the proof of Theorem 4.12 in Federer (1959).
%
\hfill$\square$ 
\end{proof}

%
%
%
%

\hspace{-13pt}{\bf Proof of Proposition~\ref{pointwisenormal}}

\begin{proof}
For any symmetric matrix $M$, let $G_i:\mathbb{R}^{d(d+1)/2}\mapsto \mathbb{R}^{d}$, for $i=1,\cdots,d$ be a vector field that maps $\text{vech} (M)$ to the $i$th unit eigenvector of $M$, such that $v_i(x)=G_i(d^2f(x))$, $v_{i,h}(x)=G_i(d^2f_h(x))$ and $\wh v_i(x)=G_i(d^2\wh f(x))$. Recall that the sign of $G_i(\text{vech} (M))$ is assumed to be determined such that $G_i(\text{vech} (M))$ is continuous as a function of a symmetric matrix $M$ when the $i$th eigenvalue of $M$ is simple. 
Also recall that $D$ is the duplication matrix. It has been shown on page 181, Magnus and Neudecker (1988) that for  $i=r+1,\cdots,d$,
\begin{align}
\nabla  G_i(d^2 f(x)) = (v_i(x)^T\otimes (\lambda_i(x) \mathbf{I}_d - \nabla^2 f(x))^+)D. 
\end{align}
Note that $\lambda_i(x) \mathbf{I}_d - \nabla^2 f(x) = \sum_{j\neq i}[(\lambda_i(x) - \lambda_j(x))v_j(x)v_j(x)^T]$. Due to the uniqueness property of pseudoinverse (page 37, Magnus and Neudecker, 1988), it is easy to verify that $(\lambda_i(x) \mathbf{I}_d - \nabla^2 f(x))^+ = \sum_{j\neq i}\left[\frac{1}{\lambda_i(x) - \lambda_j(x)}v_j(x)v_j(x)^T\right]$, and so $m_i(x)=\nabla G_i(d^2 f(x))^T\nabla f(x)$ for $m_i(x)$ in (\ref{pseudoinv}), $i=r+1,\cdots,d$. Using Taylor expansion, we have
\begin{align}\label{eigenvecdiff}
[\wh V(x)-V_h(x) ]^T& =  \begin{pmatrix}
    [G_{r+1}(d^2 \wh f(x)) - G_{r+1}(d^2 f_h(x))]^T \\
    \vdots\\
    [G_d(d^2 \wh f(x)) - G_d(d^2 f_h(x))]^T
  \end{pmatrix} \nonumber\\
 &= \begin{pmatrix}
   \Big(d^2 \wh f(x) - d^2 f_h(x)\Big)^T \nabla  G_{r+1}(d^2  f_h(x))^T\\
    \vdots\\
    \Big(d^2 \wh f(x) - d^2 f_h(x)\Big)^T\nabla  G_d(d^2 f_h(x))^T
  \end{pmatrix} +  O_p\left( (\gamma_{n,h}^{(2)})^2 \right),
\end{align}
where the order of the $O_p$-term is due to Lemma~\ref{lambda2} and the second-order derivatives of eigenvectors of symmetric matrices (see Dunajeva, 2004). Therefore using Lemma~\ref{lambda2} again we have
\begin{align}\label{ridgenessdiff}
& \wh V(x)^T\nabla\wh f(x) - V_h(x)^T\nabla f_h(x)\nonumber\\
%
%
= & [\wh V(x) -V_h(x)]^T\nabla f_h(x) +  \wh V(x)^T[\nabla\wh f(x) - \nabla f_h(x) ]\nonumber\\
= & M(x)^T\Big(d^2 \wh f(x) - d^2 f_h(x)\Big)  + O_p\left( \gamma_{n,h}^{(1)} + (\gamma_{n,h}^{(2)})^2\right),
\end{align}
where the $O_p$-term is uniform in $x\in\mathcal{M}_h$. This is (\ref{uniformapp}). Then (\ref{diffclt}) follows from Theorem 3 of Duong et al. (2008), which says 
\begin{align*}
\sqrt{nh^{d+4}}[d^2 \wh f(x) - d^2 f_h(x)] \rightarrow_D \mathscr{N}_{d(d+1)/2}(0,f(x)\mathbf{R}), \text{ as } n\rightarrow \infty.
\end{align*}

Next we show that $\Sigma(x)$ is positive definite for $x$ in a neighborhood of $\mathcal{M}$. Since $\Sigma(x)$ is a symmetric matrix, it suffices to show that $\lambda_{\min}(\Sigma(x)) >0$, where $\lambda_{\min}$ is the smallest eigenvalue of a symmetric matrix. First note that $\mathbf{R}$ is a positive definite matrix because for any $b\in\mathbb{R}^{d(d+1)/2}\backslash\{0\}$, $b^T\mathbf{R}b = \int_{\mathbb{R}^d} [b^T d^2K(u)]^2du>0$ using assumptions (\textbf{K1}) and (\textbf{K2}).  Denote $W(x)=(w_{r+1}(x),\cdots,w_d(x))$ with $w_i(x)=v_i(x)\otimes \sum_{j=1}^r \left[\frac{v_j(x)^T \nabla f(x)}{\lambda_i(x) - \lambda_j(x)}v_j(x)\right]$, $i=r+1,\cdots,d$, so that $M(x)=D^TW(x)$. Note that
\begin{align}
\lambda_{\min}(\Sigma(x)) & = \lambda_{\min}(W(x)^TD\mathbf{R}D^TW(x)) \nonumber\\
& = \inf_{a\in\mathbb{S}^{d-r-1}} a^TW(x)^TD\mathbf{R}D^TW(x)a \nonumber\\
& \geq \lambda_{\min}(\mathbf{R})  \inf_{a\in\mathbb{S}^{d-r-1}} \|D^TW(x)a\|^2.
%
\end{align}
Recall that $D^+$ is the pseudoinverse of $D$. It has full row rank and hence we have $0<\lambda_{\max}(D^+(D^+)^T)<\infty$, where $\lambda_{\max}$ is the largest eigenvalue of a symmetric matrix. We have $\|(D^+)^TD^TW(x)a\|^2\leq \lambda_{\max}(D^+(D^+)^T)\|D^TW(x)a\|^2$ and therefore
\begin{align}
\lambda_{\min}(\Sigma(x)) & \geq \frac{\lambda_{\min}(\mathbf{R})}{\lambda_{\max}(D^+(D^+)^T)}  \inf_{a\in\mathbb{S}^{d-r-1}} \|(D^+)^TD^TW(x)a\|^2.
\end{align}
Let $K_{d^2}$ be the $d^2\times d^2$ commutation matrix such that for any $d\times d$ matrix $A$, $K_{d^2}\text{vec}(A)=\text{vec}(A^T).$ It is known from Theorem 12 on page 57 of Magnus and Neudecker (1988) that $(D^+)^TD^T = \frac{1}{2}(\mathbf{I}_{d^2} + K_{d^2})$. Also $K_{d^2}$ has such a property that for any $p,q\in\mathbb{R}^d$, $K_{d^2}(p\otimes q)=q\otimes p$ (Theorem 9, page 55, Magnus and Neudecker, 1988). Therefore
\begin{align}
\lambda_{\min}(\Sigma(x)) & \geq \frac{\lambda_{\min}(\mathbf{R})}{4\lambda_{\max}(D^+(D^+)^T)}  \inf_{a\in\mathbb{S}^{d-r-1}} \|[W(x)+W^*(x)]a\|^2,
\end{align}
where $W^*(x) = (w_{r+1}^*(x),\cdots,w_d^*(x))$ with $w_i^*(x)=\sum_{j=1}^r \left[\frac{v_j(x)^T \nabla f(x)}{\lambda_i(x) - \lambda_j(x)}v_j(x)\right] \otimes v_i(x)$, $i=r+1,\cdots,d$. Denote $W^\dagger(x)=[W(x),W^*(x)]$ and $a^\dagger=(a^T,a^T)^T,$ so that $[W(x)+W^*(x)]a=W^\dagger(x)a^\dagger$. Note that the columns of $W^\dagger(x)$ are orthogonal and $\|w_i(x)\|^2=\|w_i^*(x)\|^2=\sum_{j=1}^r \left[\frac{v_j(x)^T \nabla f(x)}{\lambda_i(x) - \lambda_j(x)}\right]^2,$ $i=r+1,\cdots,d$. So we have
\begin{align}\label{minegienvalue}
\lambda_{\min}(\Sigma(x)) & \geq \frac{\lambda_{\min}(\mathbf{R})}{2\lambda_{\max}(D^+(D^+)^T)}  \lambda_{\min}[W^\dagger(x)^TW^\dagger(x)]  \nonumber\\
& = \frac{\lambda_{\min}(\mathbf{R})}{2\lambda_{\max}(D^+(D^+)^T)}  \min_{i\in\{r+1,\cdots,d\}}\sum_{j=1}^r \left[\frac{v_j(x)^T \nabla f(x)}{\lambda_i(x) - \lambda_j(x)}\right]^2.
\end{align}
Under the assumptions (\textbf{F3})-(\textbf{F5}), there exists $\delta_1>0$ such that $\inf_{x\in\mathcal{N}_{\delta_1}(\mathcal{M})} \lambda_{\min}(\Sigma(x))>0.$ In view of Lemma~\ref{lambda2bound}, we conclude that  $\Sigma(x)$ is positive definite for $x\in\mathcal{M}_h$, when $h$ is small enough.
%
\hfill$\square$

\end{proof}

\hspace{-12pt}{\bf Proof of Theorem~\ref{confidenceregion}}

To prove Theorem~\ref{confidenceregion}, we need the following lemma.
\begin{lemma}\label{posidefi}
Suppose assumptions ({\bf F1}) - ({\bf F5}), ({\bf K1}) - ({\bf K3}) hold. There exists $\delta_1>0$ such that for $x\in\mathcal{N}_{\delta_1}(\mathcal{M})$ and $z\in\mathbb{R}^d\backslash\{0\}$, $\Omega(x,z)$ in (\ref{omegaxz}) is positive definite.
\end{lemma}

\begin{proof}[\emph{ \normalfont PROOF OF LEMMA~\ref{posidefi}.}]
We need to introduce some notation first. Recall that for any $d\times d$ symmetric matrix $A$, $\text{vech}(A)$ is a half-vectorization of $A$, that is, it vectorizes only the lower triangular part of $A$ (including the diagonal of $A$). Let $\text{diag}(A)$ be the vector of the diagonal entries of $A$ and $\text{vech}_s(A)$ be the vectorization of the {\em strictly} lower triangular portion of $A$, which can be obtained from $\text{vech}(A)$ by eliminating all diagonal elements of $A$. Let $\text{dvech}(A)$ be a vectorization of the lower triangular portion of $A$, such that $\text{dvech}(A)= (\text{diag}(A)^T,\text{vech}_s(A)^T)^T$. Let $Q$ be a $[d(d+1)/2]\times[d(d+1)/2]$ matrix such that $\text{dvech}(A) = Q\,\text{vech}(A)$. Note that $Q$ is nonsingular. 

Let $\mathcal{I} = \mathcal{I}^d \cup \mathcal{I}^o$, where $\mathcal{I}^d=\{1,2,\cdots,d\}$ and $\mathcal{I}^o=\{d+1,d+2,\cdots,d(d+1)/2\}$, that is, $\mathcal{I}^d$ and $\mathcal{I}^o$ are index sets for $\text{diag}(A)$ and $\text{vech}_s(A)$ in $\text{dvech}(A)$, respectively. Suppose that $a_{l,m}$ is the element of $A$ at the $l$th row and $m$th column, for $1\leq l,m\leq d$. Define the map $\pi=(\pi_1,\pi_2): \mathcal{I}\mapsto \in \mathcal{I}^d\times \mathcal{I}^d$ such that the $k$th element of $\text{dvech}(A)$ is $a_{\pi_1(k),\pi_2(k)}$, $k\in\mathcal{I}$. For $k_1,k_2\in \mathcal{I}$, let $\pi_{\Delta}(k_1,k_2) = \{\pi_1(k_1), \pi_2(k_1)\}\Delta \{\pi_1(k_2), \pi_2(k_2)\}$, where $\Delta$ denotes symmetric difference, i.e., $A\Delta B=(A\backslash B)\cup (B\backslash A)$ for any two sets $A$ and $B$. For $i,j\in\mathcal{I}^d$, let $\pi_q^{-1}(i)=\{k\in\mathcal{I}:\pi_q(k)=i\}$, $q=1,2$, and  
\begin{align*}
\pi^{-1}(i,j)=
\begin{cases}
\pi_1^{-1}(i)\cap \pi_2^{-1}(j) & \text{ if } i\geq j,\\
\pi_1^{-1}(j)\cap \pi_2^{-1}(i) & \text{ if } i < j.
\end{cases}
\end{align*}
Note that $\pi^{-1}(i,j)=\pi^{-1}(j,i)$. Let $\pi_{\cup}^{-1}(i) = \pi_1^{-1}(i)\cup\pi_2^{-1}(i).$ Let $\delta(i,j)$ be the Kronecker delta. If $\mathcal{J}$ is a set, then let $\delta(i,\mathcal{J})=\mathbf{1}_\mathcal{J}(i)$, which is an indicator function regarding whether $i \in \mathcal{J}$. 
%
Let $\wt A(x,z): = A(x,z)^TQ^{-1}=(t_1(x,z),\cdots,t_{d(d+1)/2}(x,z)).$ Then we can write $\Omega(x,z) = \left[\int \nabla d^2 K(u)^T Q^T \wt A(x,z)^T \wt A(x,z) Q\nabla d^2 K(u) du \right],$ where we suppose that $\Omega_{i,j}(x,z)$ is at the $i$th row and $j$th column.
Let $\eta: \mathcal{I}^d\times\mathcal{I}^d\mapsto \mathbb{Z}_+^d$ be a map such that for $(l,m)\in\mathcal{I}^d\times\mathcal{I}^d$, $\frac{\partial^2 K(u)}{\partial u_l\partial u_m} = K^{(\eta(l,m))}(u)$, $u\in\mathbb{R}^d$ (see (\ref{multiindex})). Let $\Omega=(\Omega_{i,j})$. Then
\begin{align}\label{Omegaij}
\Omega_{i,j} = \sum_{(k_1,k_2)\in \mathcal{I}\times\mathcal{I}} w_{k_1k_2}^{(i,j)} t_{k_1} t_{k_2}, 
\end{align}
where $$w_{k_1k_2}^{(i,j)} =\int_{\mathbb{R}^d} \left[ \frac{\partial}{\partial u_i}  K^{(\eta(\pi(k_1)))}(u) \right] \left[ \frac{\partial}{\partial u_j} K^{(\eta(\pi(k_2)))}(u) \right] du.$$
Next we will show that we can write 
\begin{align}\label{Pmatrix}
%
\Omega(x,z) =\int_{\mathbb{R}^{d}}[K^{(\rho_2)}(s)]^2ds P(x,z) ,
\end{align}
where $P=(p_{ij})$ is a $d\times d$ matrix and $P$ is positive definite given the assumptions in this lemma. When $d=2$, it follows direct calculation using Lemma~\ref{kernelratio} that $P$ is given by
\begin{align*}
&p_{11} = a_K t_1^2 +t_2^2 +t_3^2 +2t_1t_3,\\
&p_{12} = p_{21} = 2t_1t_2 + 2t_2t_3,\\
&p_{22} = a_Kt_3^2 + t_2^2 +t_1^2 + 2t_1t_3.
\end{align*}
It is clear from Proposition~\ref{pointwisenormal} that $P(x,z)$ is a positive definite matrix for $x\in\mathcal{N}_{\delta_1}(\mathcal{M})$ for some $\delta_1>0$ and $z\in\mathbb{R}^d\backslash\{0\}$, when we assume that $a_K>1$. We consider $d\geq3$ below. Note that $w_{k_1k_2}^{(i,j)}\in\{\int [K^{(\rho_q)}(u)]^2du: q=1,2,3\}\cup\{0\}$ and we can determine the value of $w_{k_1k_2}^{(i,j)}$ using Lemma~\ref{kernelratio}. 
We split our discussion into two cases: $i=j$ and $i\neq j$. When $i=j$, the values of $w_{k_1k_2}^{(i,j)}$ can be determined by the following diagram. \\

\tikzstyle{level 1}=[level distance=4.5cm, sibling distance=3cm]
\tikzstyle{level 2}=[level distance=2.5cm, sibling distance=2cm]
\tikzstyle{level 3}=[level distance=4.5cm, sibling distance=2cm]

\tikzstyle{bag} = []
\tikzstyle{end} = [circle, minimum width=3pt,fill, inner sep=0pt]

\begin{tikzpicture}[grow=right, sloped,thick,scale=0.8, every node/.style={scale=0.8}]
\node[bag] {$w_{k_1k_2}^{(i,j)}=$}
    child {
            child {
                node[end, label=right:
                    {$0$}] {}
                edge from parent
                node[above] {$k_1\neq k_2$}
            }
            child {
		child{
			node[end, label=right:
                    	{$\int [K^{(\rho_2)}(u)]^2du$}] {}
			edge from parent
                		node[above] {$k_1\in \pi_\cup^{-1}(i)$}
		}
		child{
			node[end, label=right:
                    	{$\int [K^{(\rho_3)}(u)]^2du$}] {}
			edge from parent
                		node[above] {$k_1\notin \pi_\cup^{-1}(i)$}
		}
                edge from parent
                node[above] {$k_1=k_2$}
            }
            edge from parent 
            node[above] {$k_1,k_2\in\mathcal{I}^o$}
    }
                    child {
                node[end, label=right:
                    {$0$}] {}
                edge from parent
                node[above] {$k_1\in\mathcal{I}^d \;\&\; k_2\in\mathcal{I}^o$}
                node[below]  {or $k_2\in\mathcal{I}^d \;\&\; k_1\in\mathcal{I}^o$}
            }
    child {
            child {
            		child{
		    	node[end, label=right:
                    	{$\int [K^{(\rho_3)}(u)]^2du$}] {}
	                edge from parent
                	        node[above] {  $\;\;$   $k_1\neq i$ $\&$ $k_2\neq i$}
		   	 }
            	    child{
	    		node[end, label=right:
                    	{$\int [K^{(\rho_2)}(u)]^2du$}] {}
	                edge from parent
                	        node[above] {$k_1= i$ or $k_2=i$}
	    	    }
                edge from parent
                node[above] {$k_1\neq k_2$}
            }
                    child {
                    	child{
			node[end, label=right:
                    	{$\int [K^{(\rho_2)}(u)]^2du$}] {}
	                edge from parent
                	        node[above] {$k_1\neq i$}
			}
        			child{
			node[end, label=right:
                    	{$\int [K^{(\rho_1)}(u)]^2du$}] {}
	                edge from parent
                	        node[above] {$k_1=i$}
			}
                edge from parent
                node[above] {$k_1=k_2$}
            	}
        edge from parent         
            node[above] {$k_1,k_2\in\mathcal{I}^d$}
    };
\end{tikzpicture}
%

\hspace{-17pt}When $i\neq j$, the values of $w_{k_1k_2}^{(i,j)}$ can be determined by the following diagram. \\

\tikzstyle{level 1}=[level distance=4.5cm, sibling distance=3.5cm]
\tikzstyle{level 2}=[level distance=4.5cm, sibling distance=2cm]
\tikzstyle{level 3}=[level distance=3cm, sibling distance=2cm]

\begin{tikzpicture}[grow=right, sloped,thick,scale=0.8, every node/.style={scale=0.8}]
\node[bag] {$w_{k_1k_2}^{(i,j)}=$}
    child {
            	child{
			node[end, label=right:
                    	{$\int [K^{(\rho_3)}(u)]^2du$}] {}
			edge from parent         
                		node[above] {$\;\;\;$ $\pi_{\Delta}(k_1,k_2) = \{i,j\}$}
		}
		child{
	             	node[end, label=right:
                 	{$0$}] {}
		         edge from parent         
                		node[above] {$\pi_{\Delta}(k_1,k_2)\neq \{i,j\}$}
		}
            edge from parent 
            node[above] {$k_1,k_2\in\mathcal{I}^o$}
    }
        child {
        		child{
			child{
				node[end, label=right:
                    		{$\int [K^{(\rho_2)}(u)]^2du$}] {}
				edge from parent         
                			node[above] {$k_2\in\{i,j\}$}
			}
			child{
				node[end, label=right:
                    		{$\int [K^{(\rho_3)}(u)]^2du$}] {}
				edge from parent         
                			node[above] {$k_2\notin\{i,j\}$}
			}
			edge from parent         
                		node[above] {$k_1= \pi^{-1}(i,j)$}
		}
		child{
			node[end, label=right:
                 	{$0$}] {}
		         edge from parent         
                		node[above] {$\;\;\;\;\;\;\;$ $\pi^{-1}(i,j)\notin\{k_1,k_2\}$}
		}
		child{
			child{
				node[end, label=right:
                    		{$\int [K^{(\rho_2)}(u)]^2du$}] {}
				edge from parent         
                			node[above] {$k_1\in\{i,j\}$}
			}
			child{
				node[end, label=right:
                    		{$\int [K^{(\rho_3)}(u)]^2du$}] {}
				edge from parent         
                			node[above] {$k_1\notin\{i,j\}$}
			}
			edge from parent         
                		node[above] {$k_2= \pi^{-1}(i,j)$}
		}
                edge from parent
                node[above] {$k_1\in\mathcal{I}^d \;\&\; k_2\in\mathcal{I}^o$}
                node[below]  {or $k_2\in\mathcal{I}^d \;\&\; k_1\in\mathcal{I}^o$}
    }
        child {
             	node[end, label=right:
                 {$0$}] {}
                edge from parent         
                node[above] {$k_1,k_2\in\mathcal{I}^d$}
    };
\end{tikzpicture}    
$\;$\\

\hspace{-17pt} Plugging these values of $w_{k_1k_2}^{(i,j)}$ into (\ref{Omegaij}) we can show that the matrix $P$ in (\ref{Pmatrix}) is given by
\begin{align}\label{pijexpress}
p_{ij} = 
\begin{cases}
\Bigg[\sum\limits_{(k_1,k_2)\in\mathcal{I}^d\times\mathcal{I}^d} a_K^{\delta(i,k_1)\delta(i,k_2)} b_K^{(1-\delta(i,k_1))(1-\delta(i,k_2))(1-\delta(k_1,k_2))}t_{k_1}t_{k_2}   & \text{ if } i=j,\\
\quad\; +  \sum\limits_{k\in\mathcal{I}^o} b_K^{1-\delta(k,\pi_\cup^{-1}(i))} t_k^2 \Bigg]& \\[15pt]
2\sum\limits_{k\in\mathcal{I}^d}b_K^{1-\delta(k,\{i,j\})}t_kt_{\pi^{-1}(i,j)} + b_K\sum\limits_{k_1,k_2\in\mathcal{I}^o: \pi_{\Delta}(k_1,k_2)=\{i,j\}}t_{k_1}t_{k_2} & \text{ if } i \neq j.\\
\end{cases}
\end{align}

To prove that $P$ is positive definite under the given conditions, we will show that there exists a matrix $L$ such that  
\begin{align}\label{PmatrixDecomp}
P = LL^T + S,
\end{align}
where $S=(a_K-1/b_K)\text{diag}(t_1^2, t_2^2,\cdots, t_d^2).$
%
%
%
The matrix $L$ is in the form of $L = (L_1,L_2,L_3)$ and construction of three matrices $L_1,L_2$, $L_3$ is as follows. First $L_1=(l_{ij}^{(1)})$ is a $d\times d$ matrix where
\begin{align*}
l_{ij}^{(1)} = 
\begin{cases}
\frac{1}{\sqrt{b_K}} t_i + \sqrt{b_K} \sum\limits_{k\in\mathcal{I}^d\backslash\{i\}}t_k &  \text{ if } i=j \\
\sqrt{b_K}t_{\pi^{-1}(i,j)} &  \text{ if } i \neq j
\end{cases}.
\end{align*}
$L_2=(l_{ij}^{(2)})$ is a $d\times [d(d-1)(d-2)/6]$ matrix. Each of its $d \choose 3$ columns is constructed in the following way. Let $v=(v_1,\cdots,v_d)^T$ be a generic column. For any $1\leq j_1<j_2<j_3\leq d$,
\begin{align}\label{l1column}
v_i = 
\begin{cases}
\sqrt{b_K}t_{\pi^{-1}(j_2,j_3)}&  \text{ if }i= j_1\\
\sqrt{b_K}t_{\pi^{-1}(j_1,j_3)}&  \text{ if }i= j_2\\
\sqrt{b_K}t_{\pi^{-1}(j_1,j_2)}&  \text{ if } i= j_3\\
0 & \text{otherwise}
\end{cases}.
\end{align}
$L_3=(l_{ij}^{(3)})$ is a $d\times [d(d-1)]$ matrix. Each pair of its $2\times{d \choose 2}$ columns are constructed in the following way. Let $v^{(1)}=(v_1,\cdots,v_d)^T$ and $v^{(2)}=(v_1,\cdots,v_d)^T$ be a pair of generic columns. For any $1\leq j_1<j_2\leq d$,
\begin{align}\label{paircolumns}
v_i^{(1)} = 
\begin{cases}
\sqrt{1-b_K}t_{j_2} &  \text{ if }i= j_1\\
\sqrt{1-b_K}t_{\pi^{-1}(j_1,j_2)} &  \text{ if }i= j_2\\
0 & \text{otherwise}
\end{cases},\;
v_i^{(2)} = 
\begin{cases}
\sqrt{1-b_K}t_{\pi^{-1}(j_1,j_2)} &  \text{ if }i= j_1\\
\sqrt{1-b_K}t_{j_1} &  \text{ if }i= j_2\\
0 & \text{otherwise}
\end{cases}.
\end{align}

It is straightforward to verify that (\ref{PmatrixDecomp}) holds. The explicit expressions of $P$, $L$ and $S$ when $d=3$ are given as an example in the appendix in the supplementary material. 
%
%
%
%
%
%

To show that $P$ is positive definite, we only need to show that $L$ is full rank unless $t_k=0$ for all $k\in\mathcal{I}$.  
%
This can be seen from the following procedure. Let $e_i$ be the $i$th standard basis vector of $\mathbb{R}^d$, that is, its $i$th element is 1 and the rest are zeros. Denote the diagonal matrix $\frac{1}{\sqrt{b_K}}(t_1e_1,\cdots,t_de_d)$ by $\wt L_1$. Also denote $\wt L=(\wt L_1,L_2,L_3)$. Below we show that there exists a non-singular $d\times d$ matrix $M$ such that $\wt L=LM$, which implies that $L$ and $\wt L$ have the same rank. Here $M$ can be constructed by finding a sequence of elementary column operations on $L$, which transform $L_1$ into $\wt L_1$. Let $l_i^{(1)}$ and $l_i^{(3)}$ be the $i$th columns of $L_1$ and $L_3$, respectively. The transformation is achieved by simply noticing that  
\begin{align*}
l_i^{(1)} - \sum_{k:\; l_{ik}^{(3)} \in \mathcal{I}_d\backslash\{i\}}\sqrt{\frac{b_K}{1-b_K}} l_k^{(3)} = \frac{1}{\sqrt{b_K}} t_i e_i.
\end{align*}

Below we will show that if $t_1$, $t_2$, $\cdots$ $t_{d(d+1)/2}$ are not all zero, then there exists at least one column of $\wt L_1$, $L_2$ or $L_3$ in the form of $\sqrt{b_K}t_ke_i$ or $\sqrt{1-b_K}t_ke_i$ for some $t_k\neq 0$, for all $i=1,2,\cdots,d$, which implies that $L$ is full rank.  This is trivially true if none of $t_1,\cdots,t_d$ is zero. Now assume there is at least one of $t_1,\cdots,t_d$ is zero. Without loss of generality assume $t_1=0$ and we would like to show that there exists at least one column of $L_2$ or $L_3$ in the form of 
\begin{align}\label{basiscondition}
\sqrt{b_K}t_ke_1 \text{ or } \sqrt{1-b_K}t_ke_1,
\end{align}
for some $t_k\neq 0$. In the construction of the paired columns $v^{(1)}$ and $v^{(2)}$ of $L_3$ given in (\ref{paircolumns}), take $j_1=1$ and let $j_2$ be any integer such that $1<j_2\leq d$. If $t_{\pi^{-1}(1,j_2)}\neq 0$ then $v^{(2)}$ satisfies (\ref{basiscondition}); otherwise if $t_{j_2}\neq 0$ then $v^{(1)}$ satisfies (\ref{basiscondition}). If neither $v^{(1)}$ nor $v^{(2)}$ satisfies (\ref{basiscondition}), then we must have $t_{\pi^{-1}(1,k)}=t_k= 0$ for all $k\in\mathcal{I}^d$ (note that $t_{\pi^{-1}(1,1)}=t_1$), which is what we assume for the rest of the proof. Now we consider the columns in $L_2$. For $v$ given in (\ref{l1column}) we take $j_1=1$ and let $j_2$ and $j_3$ be any two integers satisfying $1<j_2<j_3\leq d$. Then there must exist $t_{\pi^{-1}(j_2,j_3)}\neq 0$ so that $v$ satisfies (\ref{basiscondition}), unless $t_k=0$ for all $k\in\mathcal{I}$, because $\mathcal{I} =\mathcal{I}^d  \cup \mathcal{I}^o$ and $\mathcal{I}^o = \{\pi^{-1}(i,j):\; 1\leq i < j \leq d\}$. 
%
%
\hfill$\square$

\end{proof}

\begin{proof}[$\;\;\;\;$\emph{ \normalfont PROOF OF THEOREM~\ref{confidenceregion}.}]
We first consider the case $d-r\geq 2$ and then briefly discuss the case $d-r=1$ at the end of the proof. For $g\in\mathcal{F}_h$, let $\sigma_{g} = \sqrt{\text{Var}(\mathbb{B}(g)) }$. First we want to show 
\begin{align}\label{normalizedgaussian}
\lim_{h\rightarrow 0} \mathbb{P}\left(\sup_{g\in\mathcal{F}_h} \sigma_{g}^{-1}\mathbb{B}(g)< b_h(z,c_h^{(d,r)}) \right) =  e^{-e^{-z}}.
\end{align}
We need to show that $B(x,z):=\sigma_{g_{x,z}}^{-1}\mathbb{B}(g_{x,z})$ for $g_{x,z}\in\mathcal{F}_h$ satisfies the conditions of the Gaussian fields in Theorem~\ref{ProbMain} in the supplementary material. Note that $r_h(x,\tilde x,z,\tilde z)$ is the covariance function between $B(x,z)$ and $B(\tilde x,\tilde z)$. Proposition~\ref{covstructure} and Lemma~\ref{posidefi} verify that $B(x,z)$, $(x,z)\in\mathcal{M}_h\times\mathbb{S}^{d-r-1}$ is local equi-$(\alpha_1,D_{x,z}^h,\alpha_2,B_{x,z})$-stationary (see the appendix in the supplementary material for the definition), where 
\begin{align}\label{specifica}
\alpha_1=\alpha_2=2, \; D_{x,z}^h = \frac{1}{\sqrt{2}}\Omega(x,z)^{1/2}, \text{ and } B_{x,z}=\frac{1}{\sqrt{2}}\mathbf{I}_{d-r}.
\end{align}

Note that (\ref{SupGauss2}) in Theorem~\ref{ProbMain} is clearly satisfied, simply because the kernel function $K$ is assumed to have bounded support in assumption ({\bf K1}). 
We only need to verify that $r_h$ satisfy (\ref{SupGauss1}). Recall that $\mathscr{B}(u,1)$ denotes a ball with center $u$ and unit radius. 
For any $\lambda\in\mathbb{R}$, let $\kappa(\lambda; X_1,x,\wt x,z,\wt z,h) = g_{x,z}(X_1) - \lambda g_{\tilde x,\tilde z}(X_1)$ and
\begin{align*}
\zeta(\lambda; X_1,x,\wt x,z,\wt z,h) = [\kappa(\lambda; X_1,x,\wt x,z,\wt z,h) - \mathbb{E} \kappa(\lambda; X_1,x,\wt x,z,\wt z,h)]^2.
\end{align*}
Then obviously $\mathbb{E}\zeta(\lambda; X_1,x,\wt x,z,\wt z,h) = \mathbb{E}[\kappa(\lambda; X_1,x,\wt x,z,\wt z,h)^2] - [\mathbb{E} \kappa(\lambda; X_1,x,\wt x,z,\wt z,h)]^2.$ Denote $B(x,\wt x, h)=\mathscr{B}(x,h)\cup\mathscr{B}(\wt x,h).$ Using the boundedness of the support of $K$ and the Cauchy-Schwarz inequality we have
\begin{align*}
&[\mathbb{E} \kappa(\lambda; X_1,x,\wt x,z,\wt z,h)]^2 \\
=&  \frac{1}{ h^d } \left\{\int_{\mathbb{R}^d} \left[\left\langle A(x,z), d^2K\left(\frac{x-s}{h}\right) \right\rangle -\lambda \left\langle A(\wt x,\wt z), d^2K\left(\frac{\wt x-s}{h}\right) \right\rangle\right] f(s)ds\right\}^2\\
=& \frac{1}{ h^d } \left\{\int_{B(x,\tilde x, h)} \left[\left\langle A(x,z), d^2K\left(\frac{x-s}{h}\right) \right\rangle -\lambda \left\langle A(\wt x,\wt z), d^2K\left(\frac{\wt x-s}{h}\right) \right\rangle\right] \sqrt{f(s)} \sqrt{f(s)} ds\right\}^2\\
\leq & \mathbb{E} [\kappa(\lambda; X_1,x,\wt x,z,\wt z,h)^2] F(x,\wt x,h),
\end{align*}
where $F(x,\wt x,h) = \int_{B(x,\tilde x, h)} f(s)ds = O(h^d)$, uniformly in $x$ and $\tilde x$. This implies that there exists $h_0>0$ such that for all $0<h\leq h_0$,
\begin{align}\label{zetabound}
\mathbb{E}\zeta(\lambda; X_1,x,\wt x,z,\wt z,h) \geq \frac{1}{2} \mathbb{E} [\kappa(\lambda; X_1,x,\wt x,z,\wt z,h)^2].
\end{align}
Denote $\Delta x = \wt x-x$ and $\Delta z=\wt z-z$. Using the bounded support of $K$ in assumption ({\bf K1}) again we have
%
%
%
%
%
\begin{align}
& \mathbb{E} [\kappa(\lambda; X_1,x,\wt x,z,\wt z,h)^2] \nonumber\\
\geq & \mathbb{E} \Big\{  \mathbf{1}_{\mathscr{B}(x,h)\backslash \mathscr{B}(\tilde x,h)}(X_1)\kappa(\lambda;X_1,x,\wt x,z,\wt z,h)^2\Big\}  \nonumber\\
=& \mathbb{E} \Big\{  \mathbf{1}_{\mathscr{B}(x,h)\backslash \mathscr{B}(\tilde x,h)}(X_1)\kappa(0;X_1,x,\wt x,z,\wt z,h)^2\Big\} \nonumber\\
=& \mathbb{E}\Big\{  \mathbf{1}_{\mathscr{B}(x,h)\backslash \mathscr{B}(\tilde x,h)}(X_1) \left[ g_{x,z}(X_1) \right]^2  \Big\} \nonumber\\
=& \int_{\mathscr{B}(0,1)\backslash \mathscr{B}(\Delta x/h,1)} \langle A(x,z), d^2  K(s)\rangle^2 f(x-hs)ds \nonumber\\
=& f(x)\int_{\mathscr{B}(0,1)\backslash \mathscr{B}(\Delta x/h,1)} \langle A(x,z), d^2  K(s)\rangle^2 ds +O(h),\label{InfIneq}
\end{align}
where in the last step we use a Taylor expansion for $f(x-hs)$ and the O-term is uniform in $x,\wt x\in\mathcal {M}_h$ and $z,\wt z\in\mathbb{S}^{d-r-1}$.

Note that for any $\delta>0$, if $\|\Delta x\|>h\delta$, then the set $\mathscr{B}(0,1)\backslash \mathscr{B}(\Delta x/h,1)$ contains a ball $\mathscr{B}^*$ with radius $\delta/2$. It follows that for any $x\in\mathcal{M}_h$ and $z\in\mathbb{S}^{d-r-1}$,
\begin{align*}
\inf_{\|\Delta x\|>h\delta}\int_{\mathscr{B}(0,1)\backslash \mathscr{B}(\Delta x/h,1)} \langle A(x,z), d^2  K(s)\rangle^2 ds\geq \int_{\mathscr{B}^*} \langle A(x,z), d^2  K(s)\rangle^2 ds. 
\end{align*}
Recall that $A(x,z)=M(x)Q(x)z$. As we have shown in the proof of Proposition~\ref{pointwisenormal}, $M(x)Q(x)$ is full rank for all $x\in\mathcal{N}_{\delta_0}(\mathcal{M})$ for some $\delta_0>0$. Then under assumptions (\textbf{K2}), there exists a constant $C>0$ such that the Lebesgue measure of $\{s\in {\mathscr B}^*: \langle A(x,z), d^2  K(s)\rangle^2 >C\}$ is positive for all $x\in\mathcal{N}_{\delta_0}(\mathcal{M})$ and $z\in\mathbb{S}^{d-r-1}$, because the sets $\mathcal{N}_{\delta_0}$ and $\mathbb{S}^{d-r-1}$ are compact. Therefore, 
\begin{align*}
 \inf_{x\in\mathcal{M}_h, z\in\mathbb{S}^{d-r-1}}\inf_{\|\Delta x\|>h\delta}\int_{\mathscr{B}(0,1)\backslash \mathscr{B}(\Delta x/h,1)} \langle A(x,z), d^2  K(s)\rangle^2 ds>0.
\end{align*}
%
%
%
which by (\ref{zetabound}) and (\ref{InfIneq}) further implies that for some $h_0>0$,
\begin{align}\label{InfIneq2}
\inf\limits_{\substack{x,\tilde x\in\mathcal {M}_h, z,\tilde z\in\mathbb{S}^{d-r-1} \\ \|\Delta x\|>h\delta, \|\Delta z\|>\delta, 0<h\leq h_0}} \mathbb{E}\zeta(\lambda;X_1,x,\wt x,z,\wt z,h) >0.
\end{align}
%
 %
%
%
%
%
Note that $\mathbb{E}\zeta(\lambda;X_1,x,\wt x,z,\wt z,h) =  \lambda^2\sigma_{g_{\tilde x,\tilde z}}^2 - 2\lambda\text{Cov}(g_{\tilde x,\tilde z}(X_1),g_{x,z}(X_1)) + \sigma_{g_{x,z}}^2,$
which is a quadratic polynomial in $\lambda$ and its
%
discriminant is given by
\begin{align*}
\sigma(x,\wt x,z,\wt z,h) = 4\text{Cov}(g_{\tilde x,\tilde z}(X_1),g_{x,z}(X_1)) - 4\sigma_{g_{x,z}}^2\sigma_{g_{\tilde x,\tilde z}}^2.
\end{align*}
%
%
Then (\ref{InfIneq2}) implies that
\begin{align*}
&\sup\limits_{\substack{x,\tilde x\in\mathcal {M}_h, z,\tilde z\in\mathbb{S}^{d-r-1} \\ \|\Delta x\|>h\delta, \|\Delta z\|>\delta, 0<h\leq h_0}}\sigma(x,\wt x,z,\wt z,h)<0,\\[-25pt]
\end{align*}
or equivalently,\\[-25pt]
%
\begin{align*}
&\sup\limits_{\substack{x,\tilde x\in\mathcal {M}_h, z,\tilde z\in\mathbb{S}^{d-r-1} \\ \|\Delta x\|>h\delta, \|\Delta z\|>\delta, 0<h\leq h_0}} |r_h(x,\wt x,z,\wt z)|<1.
\end{align*}

With $\beta_h =\sqrt{2r\log(h^{-1})} + \frac{1}{\sqrt{2r\log(h^{-1})}} \left[  \frac{d-2}{2} \log\log(h^{-1}) +c_h^{(d,r)} \right]$, applying Theorem~\ref{ProbMain}, we get 
\begin{align}\label{normalizedextreme}
\lim_{h\rightarrow 0} \mathbb{P} \left\{ \sqrt{2r\log{(h^{-1})}} \left( \sup_{g\in\mathcal{F}_h}\sigma_{g}^{-1}\mathbb{B}(g) -\beta_h\right) \leq z \right\} =  e^{-e^{-z}},
\end{align}
where in the calculation of $c_h^{(d,r)}$ we use (\ref{specifica}) and $H_{m}^{(2)}=\pi^{-m/2}$ for any $m\in\mathbb{Z}^+$, which is a well-known fact for Pickands' constant (cf. page 31, Piterbarg, 1996). This is just (\ref{normalizedgaussian}). \\

For $g_{x,z}\in\mathcal{F}_h$ we have
\begin{align*}
&\sigma_{g_{x,z}}^2  \\
=&\mathbb{E} [g_{x,z}(X_1)^2] - [\mathbb{E}g_{x,z}(X_1)]^2\\ 
=& \frac{1}{h^d} \int_{\mathbb{R}^d} \left\langle A(x,z), d^2K\left(\frac{x-u}{h}\right) \right\rangle^2 f(u)du- \frac{1}{h^d} \left[\int_{\mathbb{R}^d} \left\langle A(x,z), d^2K\left(\frac{x-u}{h}\right) \right\rangle f(u)du\right]^2\\
=& \int_{\mathbb{R}^d} \left\langle A(x,z), d^2K\left(u\right) \right\rangle^2 f(x-hu)du - h^d \left[ \int_{\mathbb{R}^d} \left\langle A(x,z), d^2K\left(u\right) \right\rangle f(x-hu)du\right]^2 \\
=& 1 + O(h^2),
\end{align*}
where the O-term is uniform in $x$ and $z$. Note that (\ref{normalizedextreme}) implies that $\sup_{g\in\mathcal{F}_h}|\sigma_{g}^{-1}\mathbb{B}(g)| = O_p(\sqrt{\log{h^{-1}}})$ and hence
\begin{align*}
\left|\sup_{g\in\mathcal{F}_h}\mathbb{B}(g) - \sup_{g\in\mathcal{F}_h}\sigma_{g}^{-1}\mathbb{B}(g) \right| \leq \sup_{g\in\mathcal{F}_h}|(\sigma_{g}-1)\sigma_{g}^{-1}\mathbb{B}(g)| =O_p(h^2\sqrt{\log{h^{-1}}}). 
\end{align*}

We then get (\ref{dnxextreme}) by using (\ref{normalizedextreme}). As a direct consequence of Theorem~\ref{gassianapprox}, for $D_n$ defined in (\ref{Dnx}) we have
\begin{align}\label{dnxasympt}
\mathbb{P}\left(\sqrt{nh^{d+4}}\sup_{x\in\mathcal{M}_h} D_n(x)\leq b_h(z,c_h^{(d,r)})\right) \rightarrow e^{-e^{-z}}.
\end{align}

Next we show (\ref{asympconf}). 
%
%
%
It follows from Proposition~\ref{lambda2} and Lemma~\ref{lambda2bound} that 
\begin{align}\label{probabilityone}
\mathbb{P}(\mathcal{M}_h \subset \{x\in\mathcal{H}: \; \wh\lambda_{r+1}(x)<0\}) \rightarrow 1.
\end{align}
Let $\wh C_{n,h}^*(a)=\{x\in\mathcal{H}:\;   \sqrt{nh^{d+4}} B_n(x) \leq a \}.$ Then by (\ref{probabilityone}) we get 
\begin{align}
\sup_{a\geq 0} \left|\mathbb{P}(\mathcal{M}_h \subset \wh C_{n,h}(a)) -  \mathbb{P}(\mathcal{M}_h \subset \wh C_{n,h}^*(a)) \right| \rightarrow 0.
\end{align}
Furthermore it is clear that $\mathbb{P}(\mathcal{M}_h  \subset \wh C_{n,h}^*(a)) = \mathbb{P}(\sqrt{nh^{d+4}}\sup_{x\in\mathcal{M}_h} B_n(x) \leq a )$ for all $a\geq 0$. 
%
%
By applying Proposition~\ref{ridgenessest} and (\ref{dnxasympt}), we finish the proof of (\ref{asympconf}) for the case $d-r\geq 2$. When $d-r=1$, the covariance structure of $\mathbb{B}$ is simplified (see Remark~\ref{dr1}). Then instead of using Theorem~\ref{ProbMain}, we apply the main theorem in Qiao and Polonik (2018). The rest of the proof is similar to the above. 
 \hfill$\square$
\end{proof}

\input{Reference}

\newpage
\begin{center}
{\Large \bf  Supplementary Material for ``Asymptotic Confidence Regions for Density Ridges''}\\
\end{center}

\begin{center}
{BY WANLI QIAO}
\end{center}

This supplementary material presents the proofs of some theoretical results that are not shown in Section~\ref{proofssec} due to the page constraints (Appendix A), as well as some miscellaneous results (Appendix B).

%
%

\section*{Appendix A: Technical proofs}
We need the following basic lemma to prove some of the results in the manuscript.
 
\begin{lemma}\label{lambda2}
Under assumptions ({\bf F1}), ({\bf F4}) and ({\bf K1}), as $n\rightarrow\infty$ and $h\rightarrow0$, 
\begin{align}
&  \sup_{x\in\mathcal{H}} \|\nabla \wh f(x) - \nabla f_h(x) \| =O_{a.s.}\left(\gamma_{n,h}^{(1)}\right), \;\; \sup_{x\in\mathcal{H}} \| \nabla f_h(x) - \nabla f(x)\| = O \left(h^2 \right), \label{unirate1}\\
&  \sup_{x\in\mathcal{H}} \|\nabla^2 \wh f(x) - \nabla^2 f_h(x) \| =O_{a.s.}\left(\gamma_{n,h}^{(2)}\right), \;\; \sup_{x\in\mathcal{H}} \| \nabla^2 f_h(x) - \nabla^2 f(x) \| = O \left(h^2 \right), \label{unirate2}\\
&\sup_{x\in\mathcal{H}} | \wh\lambda_{r+1}(x) - \lambda_{r+1,h}(x)| = O_{a.s.} \left( \gamma_{n,h}^{(2)}\right), \label{unirate3}\\
&\sup_{x\in\mathcal{H}}|\lambda_{r+1,h}(x) - \lambda_{r+1}(x)| = O \left(h^2 \right), \label{unirate4}\\
&\sup_{x\in\mathcal{N}_{\delta_0}(\mathcal{M})} \|\wh V(x)^T\nabla\wh f(x) - V_h(x)^T\nabla f_h(x)\| = O_{a.s.} \left( \gamma_{n,h}^{(2)}\right), \label{unirate5}\\
&\sup_{x\in\mathcal{N}_{\delta_0}(\mathcal{M})} \|V_h(x)^T\nabla f_h(x) - V(x)^T\nabla f(x) \| = O \left(h^2 \right), \label{unirate6}\\
%
%
%
%
%
&\sup_{x\in\mathcal{N}_{\delta_0}(\mathcal{M})} \|\nabla(\wh V(x)^T\nabla\wh f(x)) - \nabla(V_h(x)^T\nabla f_h(x))\|_{\max} = O_{a.s.} \left( \gamma_{n,h}^{(3)} \right), \label{unirate7}\\
&\sup_{x\in\mathcal{N}_{\delta_0}(\mathcal{M})} \|\nabla(V_h(x)^T\nabla f_h(x)) - \nabla(V(x)^T\nabla f(x)) \|_{\max} = O \left(h^2 \right). \label{unirate8}
\end{align}
Also for $\theta\in\mathbb{Z}_+^d$ with $|\theta|$=3 or 4, we have
\begin{align}\label{unirate9}
\sup_{x\in\mathcal{H}} \|\wh f^{(\theta)}(x) - f_h^{(\theta)}(x) \| =O_{a.s.}\left(\gamma_{n,h}^{(|\theta|)}\right).
\end{align}
If we further assume that both $f$ and $K$ are six times continuously differentiable, then for $\theta\in\mathbb{Z}_+^d$ with $|\theta|$=3 or 4, we have
\begin{align}
\sup_{x\in\mathcal{H}} \| f_h^{(\theta)}(x) - f^{(\theta)}(x)\| = O \left(h^2 \right). \label{unirate10}
\end{align}

\end{lemma}

\begin{proof}
The rate of the strong uniform convergence of the kernel density estimation can be found in e.g. Gin\'{e} and Guillou (2002), and Einmahl and Mason (2005). Their results can be extended to the rates for density derivative estimation as shown in (\ref{unirate1}), (\ref{unirate2}) and (\ref{unirate9}). See Lemmas 2 and 3 in Arias-Castro et al. (2016). Using (\ref{unirate2}), results (\ref{unirate3}) and (\ref{unirate4}) follows from the fact that eigenvalues are Lipschitz continuous on real symmetric matrices (Weyl inequality, cf. page 57, Serre, 2002). The rates of convergence for differences involving eigenvectors in (\ref{unirate5}), (\ref{unirate6}), (\ref{unirate7}) and (\ref{unirate8}) follow from the fact that the last $d-r$ eigenvectors are infinitely differentiable functions of the Hessian for $x\in\mathcal{N}_{\delta_0}(\mathcal{M})$ for some $\delta_0>0$, under assumption ({\bf F4}). \hfill$\square$ 
\end{proof}

%
%

\hspace{-12pt}{\bf Proof of Proposition~\ref{ridgenessest}}
\begin{proof}
Noticing that $Q(x)$ is positive definite for $x\in\mathcal{M}_h$ when $h$ is small enough due to Proposition~\ref{pointwisenormal}, (\ref{ridgenessest0}) immediately follows from Lemmas~\ref{lambda2bound} and \ref{lambda2}. Next we show (\ref{ridgenessest2}). It follows from Lemmas~\ref{lambda2} and \ref{lambda2bound} that 
\begin{align}\label{varmatrixrate}
\sup_{x\in\mathcal{N}_{\delta_0}(\mathcal{M})} \| [\wh f(x)\wh\Sigma(x)]- [f(x) \Sigma(x)] \|_F =O_p\left( \gamma_{n,h}^{(2)} +h^2\right).
\end{align}

Since $Q_n(x) - Q(x) = Q_n(x)\{[\wh f(x)\wh\Sigma(x)]^{1/2}- [f(x) \Sigma(x)]^{1/2}\}Q(x)$, using the perturbation bound theory for square roots of positive definite matrices (Theorem 6.2 in Higham, 2008), we have
\begin{align*}
& \| Q_n(x) - Q(x) \|_F\\
\leq &\|Q_n(x)\|_F \|Q(x)\|_F \| [\wh f(x)\wh\Sigma(x)]^{1/2}- [f(x) \Sigma(x)]^{1/2} \|_F\\
\leq &  \frac{\|Q_n(x)\|_F \|Q(x)\|_F}{\lambda_{\min}([\wh f(x)\wh\Sigma(x))]^{1/2} + [\lambda_{\min}(f(x) \Sigma(x))]^{1/2} } \| [\wh f(x)\wh\Sigma(x)]- [f(x) \Sigma(x)] \|_F.
\end{align*}
Therefore by (\ref{varmatrixrate}) and Proposition~\ref{pointwisenormal} we have 
\begin{align}\label{ridgenessest1}
\sup_{x\in\mathcal{N}_{\delta_0}(\mathcal{M})}  \| Q_n(x) - Q(x) \|_F = O_p\left(\gamma_{n,h}^{(2)} + h^2\right).
\end{align}
Denote $E_n(x) =  \wh V(x)^T\nabla\wh f(x) - V_h(x)^T\nabla f_h(x)$. We have $\sup_{x\in\mathcal{N}_{\delta_0}(\mathcal{M})} \left\| E_n(x)\right\| = O_p( \gamma_{n,h}^{(2)})$ by Lemma~\ref{lambda2}. For any $\mathcal{A}\subset \mathcal{N}_{\delta_0}(\mathcal{M})$, notice that
\begin{align}\label{combeq1}
&\left|\sup_{x\in\mathcal{A}}\| Q_n(x)E_n(x) \| - \sup_{x\in\mathcal{A}}\| Q(x)E_n(x) \|\right| \nonumber\\
%
%
\leq & \sup_{x\in\mathcal{A}} \left\| Q_n(x) - Q(x) \right\|_F \; \sup_{x\in\mathcal{A}} \left\| E_n(x)\right\| 
=  O_p\left( \left(\gamma_{n,h}^{(2)} +h^2\right) \gamma_{n,h}^{(2)} \right),
\end{align}
and
%
%
%
%
\begin{align}\label{combeq2}
&\left| \sup_{x\in\mathcal{A}}\| Q(x)E_n(x) \| - \sup_{x\in\mathcal{A}} D_n(x) \right| \nonumber\\
\leq & \sup_{x\in\mathcal{A}} \|Q(x)\|_F \; \sup_{x\in\mathcal{A}}\| E_n(x) - M(x)^T( d^2\wh f(x) - d^2 f_h(x))\|
=  O_p\left(\gamma_{n,h}^{(1)} + (\gamma_{n,h}^{(2)})^2\right).
\end{align}
where we use Proposition~\ref{pointwisenormal}.
Combining (\ref{combeq1}) and (\ref{combeq2}), we then get (\ref{ridgenessest2}) by noticing that $\sup_{x\in\mathcal{A}}\| Q_n(x)E_n(x) \|= \sup_{x\in\mathcal{M}_h}B_n$ when $\mathcal{A}=\mathcal{M}_h$. \hfill$\square$

\end{proof}

\hspace{-12pt}{\bf Proof of Theorem~\ref{gassianapprox}}
\begin{proof}
The proof is similar to that of Proposition 3.1 in Chernozhukov et al. (2014) so we only give a sketch. Define $\mathcal{G}_h=\{\left\langle A(x,z), d^2K\left(\frac{x-\cdot}{h}\right) \right\rangle:\; x\in\mathcal{M}_h, z\in \mathbb{S}^{d-r-1}\}$. In other words, $\mathcal{G}_h$ is created by multiplying $h^{d/2}$ with the functions in $\mathcal{F}_h$ defined in (\ref{funcclass}). Under the assumption (\textbf{K1}), for $\beta\in\mathbb{Z}_+^d$ with $|\beta|=2$ and $0<\delta_1\leq\delta_0$, the class of functions $\{K^{(\beta)}\left(\frac{x-\cdot}{h}\right):\; x\in\mathbb{R}^d\}$ is VC type (see Vaat and Wellner, 1996). By Proposition~\ref{pointwisenormal}, for some $\delta_1>0$, $\sup_{x\in\mathcal{N}_{\delta_1}(\mathcal{M}),z\in\mathbb{S}^{d-r-1}} \|A(x,z)\|_F<\infty$. Hence $\mathcal{G}_h$ is VC type when $h$ is small enough, following from Lemma A.6 in Chernozhukov et al. (2014). It is clear that $\mathcal{G}_h$ is pointwise measurable and has a bounded envelope. Also following standard calculation one can show that $\sup_{g\in\mathcal{G}_h}\mathbb{E}|g(X_1)|^3 = O(h^d)$ and $\sup_{g\in\mathcal{G}_h}\mathbb{E}|g(X_1)|^4 = O(h^d)$.
Applying Corollary 2.2 in Chernozhukov et al. (2014) with parameters $\gamma=\gamma_n=(\log n)^{-1}$, $b=O(1)$, and $\sigma=\sigma_n=h^{d/2}$, we have 
\begin{align}\label{ChernozhukovApp0}
\left|   \sup_{g\in\mathcal{G}_h} \mathbb{G}_n(g)- \sup_{g\in\mathcal{G}_h}\mathbb{D}(g)\right|  = O_p \left( n^{-1/6}h^{d/3}\log n + n^{-1/4}h^{d/4}\log^{5/4}n + n^{-1/2}\log^{3/2}n\right),
\end{align}
where $\mathbb{D}$ is a centered Gaussian process on $\mathcal{G}_h$ such that $\mathbb{E}(\mathbb{D}(p)\mathbb{D}(\wt p)) {=} \text{Cov}(p(X_1),\; \wt p(X_1))$ for all $p, \wt p\in\mathcal{G}_h$. For $g\in\mathcal{F}_h$, note that $h^{d/2}g\in\mathcal{G}_h$. Let $\mathbb{B}(g)=h^{-d/2}\mathbb{D}(h^{d/2}g)$. Due to the rescaling relationship between $\mathcal{F}_h$ and $\mathcal{G}_h$, from (\ref{ChernozhukovApp0}) we get

%
%
%

\begin{align}\label{ChernozhukovApp}
&\left|   \sup_{g\in\mathcal{F}_h} \mathbb{G}_n(g)- \sup_{g\in\mathcal{F}_h}\mathbb{B}(g)\right| & = O_p \left( \frac{\log n}{(nh^{d})^{1/6}} + \frac{\log^{5/4}n}{(nh^{d})^{1/4}} + \frac{\log^{3/2}n}{(nh^{d})^{1/2}}\right)
& = o_p(\log^{-1/2}n),
\end{align}
due to the assumption $\gamma_{n,h}^{(0)}\log^{4}n=o(1)$.  
%
%
%
Since $ \mathbb{E} \left[ \sup_{g\in\mathcal{F}_h}|\mathbb{B}(g)| \right] = O(\sqrt{\log{n}})$ (by Dudley's inequality for Gaussian processes, c.f. van der Vaart and Wellner, 1996, Corollary 2.2.8), by applying Lemma 2.4 in Chernozhukov et al. (2014), (\ref{ChernozhukovApp}) leads to $$\sup_{t>0} \left|\mathbb{P}\left( \sup_{g\in\mathcal{F}_h} \mathbb{G}_n(g)<t\right) - \mathbb{P}\left(\sup_{g\in\mathcal{F}_h}\mathbb{B}(g)<t\right)\right| = o(1).$$
Then using (\ref{dnxequiv}) we obtain (\ref{gaussapp}). \hfill$\square$

\end{proof}

\hspace{-12pt}{\bf Proof of Proposition~\ref{covstructure}}
\begin{proof}
We only give the proof for $d-r\geq2$ below and for $d-r=1$ the arguments are similar. Let $\wt x= x+\Delta x$, $\check x= x+\frac{\Delta x}{2}$, $\wt z=z+\Delta z$, and $\check z=z+\frac{\Delta z}{2}$. Let 
\begin{align*}
A_h^-(u; \check x,\Delta x,\check z,\Delta z)= \left\langle A\left(\check x-\frac{\Delta x}{2},\check z-\frac{\Delta z}{2}\right), d^2K\left(u - \frac{\Delta x}{2h}\right) \right\rangle,\\ 
A_h^+(u; \check x,\Delta x,\check z,\Delta z) = \left\langle A\left(\check x+\frac{\Delta x}{2},\check z+\frac{\Delta z}{2}\right), d^2K\left(u + \frac{\Delta x}{2h}\right) \right\rangle .
\end{align*}
Let $\xi(u) := \xi(u,\check x,\check z) =\left\langle A\left(\check x,\check z\right), d^2K\left(u \right)\right\rangle,$ and denote $\xi_{\check x}= \frac{\partial \xi}{\partial \check x} $, $\xi_{\check z}=\frac{\partial \xi}{\partial \check z}$, $\xi_{u}=\frac{\partial \xi}{\partial u}$, $\xi_{\check z\check z}=\frac{\partial^2 \xi}{\partial \check z \otimes \partial \check z}$, $\xi_{uu}=\frac{\partial^2 \xi}{\partial u \otimes \partial u}$ and $\xi_{\check z u}=\frac{\partial^2 \xi}{\partial \check z \otimes \partial u}$. Also let  
\begin{align*}
&\xi^{(1)}(u) := \xi^{(1)}(u,\check x,\Delta x,\check z,\Delta z)= \left\langle \frac{\Delta x}{2}, \xi_{\check x} \right\rangle + \left\langle \frac{\Delta z}{2}, \xi_{\check z} \right\rangle + \left\langle \frac{\Delta x}{2h}, \xi_{u} \right\rangle,\\
&\xi^{(2)}(u) := \xi^{(2)}(u,\check x,\Delta x,\check z,\Delta z)=\frac{1}{2} \|\Delta z\|_{\xi_{\check z\check z}}^2 +  \frac{1}{2} \left\|\frac{\Delta x}{h}\right\|_{\xi_{uu}}^2  + \left\langle \frac{\Delta z}{2}, \frac{\Delta x}{2h} \right\rangle_{\xi_{\check z u}}.
\end{align*}
Taking $h\rightarrow 0$, $\Delta x/h\rightarrow0$ and $\Delta z\rightarrow 0$, and using Taylor expansion we have
\begin{align}
& A_h^-(u; \check x,\Delta x,\check z,\Delta z)= \xi(u) - \xi^{(1)}(u) + \xi^{(2)}(u) + o\left(\left\| \Delta x/h\right\|^2 + \|\Delta z\|^2 \right),\label{Ahminus}\\
& A_h^+(u; \check x,\Delta x,\check z,\Delta z) =  \xi(u) + \xi^{(1)}(u) + \xi^{(2)}(u) + o\left(\left\| \Delta x/h\right\|^2 + \|\Delta z\|^2 \right).\label{Ahplus}
%
%
\end{align}
Here and in the Taylor expansions throughout this proof, the o-terms are uniform in $x,\wt x\in\mathcal{M}_h,$ $z,\wt z\in \mathbb{S}^{d-r-1}$, and $h\in(0,h_0]$ for some $h_0>0$, due to assumptions (\textbf{F1}) - (\textbf{F5}) and (\textbf{K1}).
%
%
%
%
Note that for all $g_{x,z}, g_{\tilde x,\tilde z}\in\mathcal{F}_h$, by (\ref{Ahminus}) and (\ref{Ahplus}) and using change of variables $u=(\check x -s)/h$ we have the following calculation.
\begin{align}\label{calcser1}
&\mathbb{E}[g_{x,z}(X_1)g_{\tilde x, \tilde z}(X_1)] \nonumber\\
=& \frac{1}{h^{d} } \int_{\mathbb{R}^d} \left\langle A(x,z), d^2K\left(\frac{x-s}{h}\right) \right\rangle \left\langle A(\wt x,\wt z), d^2K\left(\frac{\tilde x-s}{h}\right) \right\rangle f(s)ds \nonumber\\
=& \int_{\mathbb{R}^d} A_h^-(u; \check x,\Delta x,\check z,\Delta z) A_h^+(u; \check x,\Delta x,\check z,\Delta z) f(\check x-hu)du \nonumber\\
=& \int_{\mathbb{R}^d} [\xi(u)^2 - \xi^{(1)}(u)^2+2\xi(u)\xi^{(2)}(u)] f(\check x-hu)du + o\left(\left\| \Delta x/h\right\|^2 + \|\Delta z\|^2 \right), 
%
%
%
%
\end{align}
and
\begin{align}\label{calcser2}
&\mathbb{E}[g_{x,z}(X_1)^2]\nonumber \\
=& \frac{1}{h^{d} } \int_{\mathbb{R}^d} \left\langle A(x,z), d^2K\left(\frac{x-s}{h}\right) \right\rangle^2 f(s)ds \nonumber\\
=& \int_{\mathbb{R}^d} A_h^-(u;\check x,\Delta x,\check z,\Delta z)^2 f(\check x-hu)du \nonumber\\
=& \int_{\mathbb{R}^d} [\xi(u)^2 - 2\xi(u)\xi^{(1)}(u) + \xi^{(1)}(u)^2 + 2\xi(u)\xi^{(2)}(u)] f(\check x-hu)du + o\left(\left\| \Delta x/h\right\|^2 + \|\Delta z\|^2 \right),
%
\end{align}
and
\begin{align}\label{calcser3}
&\mathbb{E}[g_{\tilde x,\tilde z}(X_1)^2] \nonumber\\
=& \frac{1}{h^{d} } \int_{\mathbb{R}^d} \left\langle A(\wt x,\wt z), d^2K\left(\frac{\wt x-s}{h}\right) \right\rangle^2 f(s)ds \nonumber\\
=& \int_{\mathbb{R}^d} A_h^+(u; \check x,\Delta x,\check z,\Delta z)^2 f(\check x-hu)du \nonumber\\
=& \int_{\mathbb{R}^d} [\xi(u)^2 + 2\xi(u)\xi^{(1)}(u) + \xi^{(1)}(u)^2 + 2\xi(u)\xi^{(2)}(u)] f(\check x-hu)du + o\left(\left\| \Delta x/h\right\|^2 + \|\Delta z\|^2 \right).
%
%
%
\end{align}
Using (\ref{Ahminus}) and (\ref{Ahplus}) again, and noticing that $\mathbb{E}(g_{x,z}(X_1)) = \sqrt{h^{d}}  \int A_h^-(u; \check x,\Delta x,\check z,\Delta z)  f(\check x-hu)du$ and $\mathbb{E}(g_{\tilde x,\tilde z}(X_1)) = \sqrt{h^{d}}  \int A_h^+(u; \check x,\Delta x,\check z,\Delta z)  f(\check x-hu)du$, we have
\begin{align}
&\mathbb{E}(g_{x,z}(X_1))  \mathbb{E}(g_{\tilde x,\tilde z}(X_1)) 
= h^d\left[\int_{\mathbb{R}^d} \xi(u) f(\check x - hu)du\right]^2 
%
%
 + o\left(\left\| \Delta x/h\right\|^2 + \|\Delta z\|^2 \right), \label{calcser4}\\
 & \left[ \mathbb{E}(g_{x,z}(X_1)) \right]^2 = h^d\left[\int_{\mathbb{R}^d} \xi(u) f(\check x - hu)du\right]^2 -h^d\left[\int_{\mathbb{R}^d} \xi(u) f(\check x - hu)du\right] \nonumber\\
 &\hspace{3cm}\times \left[\int \xi^{(1)}(u) f(\check x - hu)du\right]+ o\left(\left\| \Delta x/h\right\|^2 + \|\Delta z\|^2 \right), \label{calcser5}\\
  & \left[ \mathbb{E}(g_{\tilde x,\tilde z}(X_1)) \right]^2 = h^d\left[\int_{\mathbb{R}^d} \xi(u) f(\check x - hu)du\right]^2  + h^d\left[\int_{\mathbb{R}^d} \xi(u) f(\check x - hu)du\right] \nonumber\\
  &\hspace{3cm}\times \left[\int_{\mathbb{R}^d} \xi^{(1)}(u) f(\check x - hu)du\right]+ o\left(\left\| \Delta x/h\right\|^2 + \|\Delta z\|^2 \right). \label{calcser6}
\end{align}
Let 
\begin{align*}
&M_0 = \int_{\mathbb{R}^d} \xi(u)^2 f(\check x - hu)du - h^d\left[\int_{\mathbb{R}^d} \xi(u) f(\check x - hu)du\right]^2,\\
%
&M_1 =  \int_{\mathbb{R}^d} \xi(u) \xi^{(1)}(u)f(\check x - hu)du -  h^d\left[\int_{\mathbb{R}^d} \xi(u) f(\check x - hu)du\right] \left[\int_{\mathbb{R}^d} \xi^{(1)}(u) f(\check x - hu)du\right],\\
%
%
&M_2 =  \int_{\mathbb{R}^d}  \xi^{(1)}(u)^2 f(\check x - hu)du,\\
%
%
&M_3 = 2\int_{\mathbb{R}^d} \xi(u) \xi^{(2)}(u)f(\check x - hu)du .
%
\end{align*}
It follows from the calculations in (\ref{calcser1}) - (\ref{calcser6}) that
\begin{align*}
& \text{Var}(g_{x,z}(X_1)) \text{Var}(g_{\tilde x,\tilde z}(X_1))\\
=& (M_0-2M_1+M_2+M_3)(M_0+2M_1+M_2+M_3) + o\left(\left\| \Delta x/h\right\|^2 + \|\Delta z\|^2 \right)\\
=& (M_0^2 - 4M_1^2 + 2M_0M_2 + 2M_0M_3) + o\left(\left\| \Delta x/h\right\|^2 + \|\Delta z\|^2 \right).
\end{align*}
%
%
%
%
Similarly,
\begin{align*}
\text{Cov}(g_{x,z}(X_1) ,\;g_{\tilde x,\tilde z}(X_1)) 
= (M_0 - M_2 + M_3 ) + o\left(\left\| \Delta x/h\right\|^2 + \|\Delta z\|^2 \right).
\end{align*}
So using Taylor expansion we get
\begin{align}\label{rhytildey}
r_h(x,\tilde x,z,\tilde z) =& \frac{\text{Cov}(g_{x,z}(X_1) ,\;g_{\tilde x,\tilde z}(X_1))}{\sqrt{\text{Var}(g_{x,z}(X_1)) \text{Var}(g_{\tilde x,\tilde z}(X_1))}} \nonumber\\
= & 1 +  2\frac{M_1^2}{M_0^2} - 2\frac{M_2}{M_0}  + o\left(\left\| \Delta x/h\right\|^2 + \|\Delta z\|^2 \right).
\end{align}
Below we will find the leading terms in the expansion of $M_0$, $M_1$ and $M_2$. For $M_0$, we have 
\begin{align}\label{M0approx}
M_0 = & \int_{\mathbb{R}^d} \xi(u)^2 f(\check x-hu)du\{1+o(1)\}\nonumber\\
= & \int_{\mathbb{R}^d} \left\langle A\left(\check x,\check z\right), d^2K\left(u\right) \right\rangle^2f(\check x-hu)du\{1+o(1)\}\nonumber\\
=& \int_{\mathbb{R}^d} \left\langle A\left(\check x,\check z\right), d^2K\left(u\right) \right\rangle^2f(\check x)du \{1+o(1)\}\nonumber\\
= & 1+o(1).
\end{align}
Next we consider $M_1$. Since the kernel $K$ is symmetric, we have
\begin{align*}
f(\check x) \int_{\mathbb{R}^d} \xi(u) \left\langle \frac{\Delta x}{2h}, \xi_{u} (u)\right\rangle  du = 0.
\end{align*}
%
Let $J(x,u)=[f(x)\Sigma(x)]^{-1/2}M(x)^Td^2K(u)$. Then $\xi(u)=\langle\check z,J(\check x,u)\rangle$ and $\xi_{\check z}(u)=J(\check x,u)$. 
%
Notice that $\int_{\mathbb{R}^d} J(\check x,u)J(\check x,u)^T du = [f(\check x)]^{-1}\mathbf{I}_{d(d+1)/2}$. Therefore
\begin{align}\label{intinner1}
f(\check x) \int_{\mathbb{R}^d} \xi(u)  \left\langle \frac{\Delta z}{2}, \xi_{\check z}(u) \right\rangle du = \Delta z^T\check z = \Delta z^T \left(z+\frac{\Delta z}{2}\right).
\end{align}
Since we only consider $z,z+\Delta z\in\mathbb{S}^{d-r-1}$, we have $\|z\|^2 = \|z+\Delta z\|^2$, that is, $\|z\|^2 = \|z\|^2 + 2 \Delta z^T z + \|\Delta z\|^2$, which implies that 
$\Delta z^T z = -\frac{1}{2} \|\Delta z\|^2$ and furthermore 
\begin{align}\label{intinner2}
f(\check x) \int_{\mathbb{R}^d} \xi(u)  \left\langle \frac{\Delta z}{2}, \xi_{\check z}(u) \right\rangle du = 0.
\end{align}
Hence using (\ref{intinner1}) and (\ref{intinner2}) and a Taylor expansion, we have
\begin{align}\label{M1approx}
M_1^2 
%
%
=o\left(\left\| \Delta x/h\right\|^2 + \|\Delta z\|^2 \right).
\end{align}
Also note that
\begin{align}\label{M2approx}
M_2 = &\int_{\mathbb{R}^d} \left[ \left\langle \frac{\Delta z}{2}, \xi_{\check z}(u) \right\rangle^2 + \left\langle \frac{\Delta x}{2h}, \xi_{u} (u)\right\rangle^2\right] f(\check x - hu)du + o\left(\left\| \Delta x/h\right\|^2 + \|\Delta z\|^2\right)\nonumber\\
=& \frac{1}{4} \|\Delta z\|^2 + \frac{1}{4}  \Delta x^T\Omega(x,z)\Delta x  + o\left(\left\| \Delta x/h\right\|^2 + \|\Delta z\|^2\right),
%
\end{align}
%
%
where $\Omega(x,z)$ is given in (\ref{omegaxz}).
%
%
%
%
%
%
Therefore using (\ref{rhytildey}), (\ref{M0approx}), (\ref{M1approx}) and (\ref{M2approx}), we get (\ref{covariancestructure}). \hfill $\square$
\end{proof}

\hspace{-12pt}{\bf Proof of Lemma~\ref{biasasymptotic}}
\begin{proof}
It follows from a similar derivation for (\ref{eigenvecdiff}) that $[V_h(x) -V(x)]^T = J_h(x)^T + R_{h1}(x)^T$, where $R_{h1}(x)^T=O(h^4)$ and
\begin{align*}
J_h(x)^T = \begin{pmatrix}
   \Big(d^2 f_h(x) - d^2 f(x) \Big)^T \nabla  G_{r+1}(d^2  f(x))^T\\
    \vdots\\
    \Big(d^2 f_h(x) - d^2 f(x)\Big)^T\nabla  G_d(d^2 f(x))^T
  \end{pmatrix} = O(h^2).
\end{align*}
Using integration by part, Taylor expansion and the symmetry of $K$, we have
\begin{align*}
d^2 f_h(x) -  d^2 f(x) = \int_{\mathbb{R}^d} K(u) d^2f(x+hu)du - d^2 f(x) = \frac{1}{2}h^2\mu_K \Delta_Ld^2 f(x) + R_{h2}(x),
\end{align*}
where $R_{h2}(x)=o(h^2)$, because of the assumed fourth-order continuous differentiability of $f$ on $\mathcal{H}$ and the compactness of $\mathcal{H}$. Similarly we get $ \nabla f_h(x) - \nabla f(x) = \frac{1}{2}h^2\mu_K \Delta_L \nabla f(x) + R_{h3}(x)$, where $R_{h3}(x)=o(h^2)$. 
%
Therefore, similar to (\ref{ridgenessdiff}), we have
\begin{align}\label{biasrem}
&V_h(x)^T\nabla f_h(x) - V(x)^T\nabla f(x) \nonumber\\
=& [V_h(x)-V(x)]^T\nabla f_h(x) +  V(x)^T[ \nabla f_h(x) - \nabla f(x) ] \nonumber\\
=& M(x)^T\Big( d^2 f_h(x) -  d^2 f(x) \Big) + R_{h1}^T\nabla f_h(x) +  [J_h(x) + V(x)]^T[ \nabla f_h(x) - \nabla f(x) ] \nonumber\\
=& \frac{1}{2}h^2\mu_K \beta(x) + R_h(x) ,
\end{align}
where $R_h(x) = o(h^2)$, uniformly in $x\in\mathcal{N}_{\delta_0}(\mathcal{M})$. When $f$ is six times continuously differentiable, by using higher order Taylor expansions, we have $R_{h2}(x)=O(h^4)$ and $R_{h3}(x)=O(h^4)$, and therefore $R_{h}(x)=O(h^4)$, uniformly in $x\in\mathcal{N}_{\delta_0}(\mathcal{M})$.  \hfill$\square$
\end{proof}

\hspace{-12pt}{\bf Proof of Theorem~\ref{underconfidenceregion}}

\begin{proof}
(i) {\em Undersmoothing}: First recall that $B_n(x)=\| Q_n(x)[\wh V(x)^T\nabla\wh f(x)]\|$. Since $V(x)^T\nabla f(x)=0$ for $x\in\mathcal{M}$, we write $\sup_{x\in\mathcal{M}}B_n(x)=\sup_{x\in\mathcal{M}}\| Q_n(x)[\wh V(x)^T\nabla\wh f(x)- V(x)^T\nabla f(x) ]\|$. We denote $\wt B_n(x)= \| Q_n(x)[\wh V(x)^T\nabla\wh f(x) - V_h(x)^T\nabla f_h(x)]\|$. It is known that $\sup_{x\in\mathcal{M}} \| Q_n(x)\|_F  = O_p(1)$ by (\ref{ridgenessest1}). Also we have that $\sup_{x\in\mathcal{M}} \|V(x)^T\nabla f(x) - V_h(x)^T\nabla f_h(x)\|=O(h^2)$ by using Lemma~\ref{biasasymptotic}. Hence
\begin{align}\label{negldiff}
\left|\sup_{x\in\mathcal{M}}B_n(x) - \sup_{x\in\mathcal{M}}\wt B_n(x)\right| & \leq \sup_{x\in\mathcal{M}} \| Q_n(x)\|_F \sup_{x\in\mathcal{M}} \|V(x)^T\nabla f(x) - V_h(x)^T\nabla f_h(x)\| \nonumber\\
& = O_p(h^2).
\end{align}

%
%
%
%
%
%
%
%
It follows from the same proof for Theorem~\ref{confidenceregion} that
\begin{align*}
\mathbb{P}\left(\sqrt{nh^{d+4}}\sup_{x\in\mathcal{M}} \wt B_n(x)\leq b_h(z_\alpha,c^{(d,r)})\right) \rightarrow 1-\alpha.
\end{align*}
Under the assumption $\gamma_{n,h}^{(4)}\rightarrow\infty$, then (\ref{negldiff}) leads to (\ref{asympconfbias}).


(ii) {\em Explicit bias correction}: Let $B_n^{\text{bc}}(x)=\| Q_n(x)[\wh V(x)^T\nabla\wh f(x) -\frac{1}{2}h^2\mu_K\wh \beta_{n,l}(x)]\|$. Then
\begin{align}\label{negldiff2}
& \left|\sup_{x\in\mathcal{M}}B_n^{\text{bc}}(x) - \sup_{x\in\mathcal{M}}\wt B_n(x)\right| \nonumber\\
 \leq & \sup_{x\in\mathcal{M}} \| Q_n(x)\|_F \sup_{x\in\mathcal{M}} \|V(x)^T\nabla f(x) - V_h(x)^T\nabla f_h(x) -\frac{1}{2}h^2\mu_K\wh \beta_{n,l}(x)\| \nonumber\\
\leq & \sup_{x\in\mathcal{M}} \| Q_n(x)\|_F \left[\frac{1}{2}h^2\mu_K \sup_{x\in\mathcal{M}} \|\wh \beta_{n,l}(x) - \beta(x)\| + \sup_{x\in\mathcal{M}} \|R_h(x)\| \right],
\end{align}
where $R_h(x)$ is given in (\ref{biasrem}) and $\sup_{x\in\mathcal{M}} \|R_h(x)\|=O(h^4).$
It follows from Lemma~\ref{lambda2} that $\sup_{x\in\mathcal{M}}\|\wh \beta_{n,l}(x) - \beta(x)\|=O_p(\gamma_{n,l}^{(4)}) + O(h^2)$, assuming that $f$ is six times continuously differentiable. Then Lemma~\ref{biasasymptotic} leads to that the right-hand side of (\ref{negldiff2}) is of order $ O_p(h^2\gamma_{n,l}^{(4)} +h^2l^2 +h^4)$. 
%
Then following the same arguments in (i), we obtain (\ref{asympconfbiascorr}) using the assumptions $h/l\rightarrow0$ and $\gamma_{n,h}^{(4)}/l^2\rightarrow\infty$.
\hfill$\square$

\end{proof}

\hspace{-12pt}{\bf Proof of Corollary~\ref{pluginregion}}
\begin{proof}
The idea in this proof is similar to that given in Qiao (2019a), in particular the proof of Theorem 3.1 therein and so we only give a sketch of the proof. Using Theorems~\ref{confidenceregion} and \ref{underconfidenceregion}, it suffices to prove $\wh c_{n,l}^{(d,r)} - c_h^{(d,r)} =o_p(1)$ and $\wh c_{n,l}^{(d,r)} - c^{(d,r)} =o_p(1)$. Due to their similarity, we only prove the former, which is equivalent to
\begin{align}\label{surfaceintconsist}
\int_{\wh{\mathcal{M}}_{n,l}} \|\wh\Omega_{n,l}(x,z)^{1/2} \Lambda(T_x\wh{\mathcal{M}}_{n,l})\|_r d\mathscr{H}_r(x) - \int_{\mathcal{M}_h} \|\Omega(x,z)^{1/2} \Lambda(T_x\mathcal{M}_h)\|_r d\mathscr{H}_r(x) =o_p(1).
\end{align}
%
 
First note that with probability one $\wh{\mathcal{M}}_{n,l} \subset \mathcal{N}_{\delta_0}(\mathcal{M})$ for $n$ large enough by using Lemma~\ref{lambda2} and a similar argument as in the proof of Lemma~\ref{lambda2bound}. Also with probability one $ \wh{\mathcal{M}}_{n,l}$ is an $r$-dimensional manifold and has positive reach for $n$ large enough. This can be shown in a way similar to the proof of Lemma~\ref{lambda2bound}, by using Lemma~\ref{lambda2} and Theorem 4.12 in Federer (1959). 

Next we define a normal projection from $\mathcal{M}_h$ to $\wh{\mathcal{M}}_{n,l}$. Let $p_{i,h}(x)=\nabla f_h(x)^Tv_{r+i,h}(x)$, and $l_{i,h}(x)=\nabla p_{i,h}(x)$. Also let $N_h(x)=(N_{1,h}(x),\cdots,N_{d-r,h}(x))$, where $N_{i,h}(x)=l_{i,h}(x)/\|l_{i,h}(x)\|$, $i=1,\cdots,d-r$. Note that $N_{i,h}(x)$, $i=1,\cdots,d-r$ are unit vectors that spans the normal space of $\mathcal{M}_h$. Similarly we define $\wh N_{n,l}$ with its columns spanning the normal space of $\wh{\mathcal{M}}_{n,l}$. For $t=(t_{1},\cdots,t_{d-r})^T$, define $\zeta_x(t)=x+N_h(x)t$ and $t_n(x) = \argmin_t\{\|t\|: \zeta_x(t)\in\wh{\mathcal{M}}_{n,l}\}$. For $x\in\mathcal{M}_h$, let $P_n(x)=\zeta_x(t_n(x))$. Then following similar arguments in the proof of Theorem 1 in Chazal et al. (2007) one can show with probability one $P_n$ is a homeomorphism between $\mathcal{M}_h$ and $\wh{\mathcal{M}}_{n,l}$ when $n$ is large enough. Note that for $x\in\mathcal{M}_h$ we can establish the following system of equations: for $ i=1,\cdots,d-r,$ 
\begin{align}\label{systemequ}
0 = \wh p_i(P_n(x)) - p_{i,h}(x) =  \wh p_i(x) - p_{i,h}(x) + t_n(x)^TN_h(x)^T\nabla p_{i,h}(x) + O(\|t_n(x)\|^2).
\end{align}
Note that under our assumption, $N_h(x)$ is full rank and so $N_h(x)N_h(x)^T$ is positive definite. Then (\ref{systemequ}) yields $t_n(x) = [N_h(x)N_h(x)^T]^{-1} [\wh p(x) - p_h(x)] + O(\|t_n(x)\|^2)$ and therefore $\sup_{x\in\mathcal{M}_h}\|t_n(x) \| = o_p(1).$ Similarly by taking gradient on both sides of (\ref{systemequ}) we can obtain $\sup_{x\in\mathcal{M}_h}\|\nabla t_n(x) \|_F = o_p(1).$\\

Note that the Jacobian of $P_n$ is $J_n(x) = \mathbf{I}_d+\nabla N_h(x) t_n(x) + N_h(x) \nabla t_n(x)$. From the above derivation we have 
\begin{align}
&\sup_{x\in\mathcal{M}_h}\|P_n(x)- x\|=o_p(1),\label{pnconsis}\\
&\sup_{x\in\mathcal{M}_h}\|\wh N_{n,l}(P_n(x)) - N_h(x)\|_F=o_p(1),\label{npnconsist}\\
&\sup_{x\in\mathcal{M}_h}\|J_n(x)- \mathbf{I}_d\|_F=o_p(1).\label{jnconsist}
\end{align}

Also using Lemma~\ref{lambda2} and following similar arguments given in the proof of Proposition~\ref{ridgenessest} one can show that  
\begin{align}
&\sup_{x\in\mathcal{M}_h}\|\wh \Omega_{n,l}(x,z)- \Omega(x,z)\|_F=o_p(1).\label{omeganconsist}
\end{align}

Since $\mathcal{M}_h$ is a compact submanifold embedded in $\mathbb{R}^d$, it admits an atlas $\{(U_\alpha,\psi_\alpha):\alpha\in\mathscr{A}\}$ indexed by a finite set $\mathscr{A}$, where $\{U_\alpha:\alpha\in\mathscr{A}\}$ is an open cover of $\mathcal{M}$, and for an open set $\Omega_\alpha\subset \mathbb{R}^{r}$, $\psi_\alpha: \Omega_\alpha \mapsto U_\alpha$ is a diffeomorphism. We suppress the subscript $\alpha$ in what follows. For any $U$, we write $\wh U =\{P_n(x):x\in U\}$. Then
 \begin{align*}
&\int_{\wh U} \|\wh\Omega_{n,l}(x,z)^{1/2} \Lambda(T_x\wh{\mathcal{M}}_{n,l})\|_r d\mathscr{H}_r(x) - \int_{U} \|\Omega(x,z)^{1/2} \Lambda(T_x\mathcal{M}_h)\|_r d\mathscr{H}_r(x) \\
= & \text{I}_n + \text{II}_n + \text{III}_n, 
 \end{align*}
 where
\begin{align*}
&\text{I}_n = \int_{U} \|\Omega(P_n(x),z)^{1/2} \Lambda(T_x\mathcal{M}_h)\|_r d\mathscr{H}_r(x) - \int_{U} \|\Omega(x,z)^{1/2} \Lambda(T_x\mathcal{M}_h)\|_r d\mathscr{H}_r(x),\\
&\text{II}_n = \int_{U} \|\wh\Omega_{n,l}(P_n(x),z)^{1/2} \Lambda(T_{P_n(x)}\wh{\mathcal{M}}_{n,l})\|_r d\mathscr{H}_r(x) \\
& \hspace{3cm} - \int_{U} \|\Omega(P_n(x),z)^{1/2} \Lambda(T_x\mathcal{M}_h)\|_r d\mathscr{H}_r(x) ,\\
&\text{III}_n = \int_{\wh U} \|\wh\Omega_{n,l}(x,z)^{1/2} \Lambda(T_x\wh{\mathcal{M}}_{n,l})\|_r d\mathscr{H}_r(x) \\
& \hspace{3cm} - \int_{U} \|\wh\Omega_{n,l}(P_n(x),z)^{1/2} \Lambda(T_{P_n(x)}\wh{\mathcal{M}}_{n,l})\|_r d\mathscr{H}_r(x) .
\end{align*} 
Then (\ref{surfaceintconsist}) follows from $\text{I}_n + \text{II}_n + \text{III}_n=o_p(1)$, where $\text{I}_n=o_p(1)$ is due to (\ref{pnconsis}) and that $\Omega(x,z)$ as a function of $x$ is continuous on $\mathcal{N}_{\delta_1}(\mathcal{M})$ for some $0<\delta_1\leq\delta_0$, $\text{II}_n=o_p(1)$ is due to (\ref{npnconsist}) and (\ref{omeganconsist}), and $\text{III}_n=o_p(1)$ is due to (\ref{jnconsist}). We then conclude the proof. \hfill$\square$
\end{proof}

\hspace{-12pt}{\bf Proof of Proposition~\ref{lambdaprop}}
\begin{proof}
We only show the proof of (\ref{lambdah}) and (\ref{lambdanh}) can be proved similarly. Note that $\sup_{x\in\mathcal{M}_h} \wh\lambda_{h,r+1}(x)  - \sup_{x\in\wh{\mathcal{M}}} \wh{\lambda}_{r+1}(x) = \text{I}_n + \text{II}_n + \text{III}_n,$
where $\text{I}_n = \sup_{x\in\mathcal{M}_h} \lambda_{h,r+1}(x)  -  \sup_{x\in\wh{\mathcal{M}}} \lambda_{h,r+1}(x) $, $\text{II}_n = \sup_{x\in\wh{\mathcal{M}}} \lambda_{h,r+1}(x) - \sup_{x\in\wh{\mathcal{M}}} \wh{\lambda}_{r+1}(x) $, $\text{III}_n = \sup_{x\in\mathcal{M}_h} \wh{\lambda}_{r+1}(x) - \sup_{x\in\mathcal{M}_h} \lambda_{h,r+1}(x) $. It suffices to prove 
\begin{align}
\text{I}_n = O_p\left(\gamma_{n,h}^{(2)}\right),\label{In}\\
\text{II}_n + \text{III}_n = O_p\left(\gamma_{n,h}^{(2)}\right).\label{IIn}
%
\end{align}
Note that 
\begin{align*}
\max(\left| \text{II}_n \right|,\; \left| \text{III}_n \right|) \leq \sup_{x\in\wh{\mathcal{M}}\cup \mathcal{M}_h} \left| \wh{\lambda}_{r+1}(x) - \lambda_{h,r+1}(x)  \right| \leq \sup_{x\in\mathcal{H}} \left| \wh{\lambda}_{r+1}(x) - \lambda_{h,r+1}(x)\right|,
\end{align*} 
and hence (\ref{IIn}) follows from Lemma~\ref{lambda2}. Next we show (\ref{In}). For $\delta>0$, let
\begin{align*}
\mathcal{N}_{\delta}(\mathcal{M}_h) = \{x\in\mathcal{H}: \|V_h(x)^T \nabla f_h(x)\|\leq \delta,\; \lambda_{h,r+1}(x)<0\}.
\end{align*}
Using Lemma~\ref{lambda2} and following the proof of Lemma~\ref{lambda2bound}, we can find a constant $C_1>0$ such that $\mathbb{P}(\wh{\mathcal{M}}\subset \mathcal{N}_{d_{n,1}}(\mathcal{M}_h)) \rightarrow 1$ with $d_{n,1}=C_1\gamma_{n,h}^{(2)}$. Given any $u\in\mathbb{R}^{d-r}$, let $\mathcal{M}_h^u = \{x\in\mathcal{H}:\; V_h(x)^T\nabla f_h(x)=u, \lambda_{h,r+1}(x)<0\}.$ Using similar arguments given in the proof of Corollary~\ref{pluginregion} and due to (\ref{marginassump}), bijective normal projections can be established between $\mathcal{M}_h=\mathcal{M}_h^0$ and $\mathcal{M}_h^u$ when $\|u\|$ is small enough. Hence there exists a constant $C_2>0$ such that when both $\|u\|$ and $h$ are small enough we have $\sup_{x \in \mathcal{M}_h^u} d(x, \mathcal{M}_h) \leq C_2 \|u\|.$ This then implies 
\begin{align}\label{hausdorff1}
\mathbb{P}(\wh{\mathcal{M}}\subset \mathcal{M}_h\oplus (C_2d_{n,1}) ) \rightarrow 1.
\end{align}
%
%
Next we will show that for some $d_{n,2} = O(\gamma_{n,h}^{(2)})$, 
\begin{align}\label{hausdorff2}
\mathbb{P}(\mathcal{M}_h \subset \wh{\mathcal{M}}\oplus d_{n,2} ) \rightarrow 1.
\end{align}
For $d-r=1$, (\ref{hausdorff2}) directly follows from Theorem 2 of Cuevas et al. (2006). Next we show (\ref{hausdorff2}) for $d-r\geq 2$.  Let $\wh l_i(x) = \nabla(\nabla \wh f(x)^T \wh v_{r+i}(x))$, $i=1,\cdots,d-r$ and $\wh L(x)=(\wh l_{1}(x),\cdots,\wh l_{d-r}(x))$. Using Lemma~\ref{lambda2} and the assumption $\gamma_{n,h}^{(3)}\rightarrow0$, then similar to (\ref{marginassump}), we have
\begin{align}
&\inf_{x\in\mathcal{N}_{\delta_0}(\mathcal{M})} \text{det}(\wh L(x)^T\wh L(x)) \nonumber \\
\geq & \inf_{x\in\mathcal{N}_{\delta_0}(\mathcal{M})} \text{det}(L(x)^TL(x)) - \sup_{x\in\mathcal{N}_{\delta_0}(\mathcal{M})} | \text{det}(L(x)^TL(x))- \text{det}(\wh L(x)^T\wh L(x))| \nonumber\\
\geq & \epsilon_0 - O_{a.s.}\left( \gamma_{n,h}^{(3)} +h^2 \right),\label{marginassump2}
\end{align}
where $\epsilon_0>0$ is given in (\ref{epsilon0const}). This then allows us to switch the roles between $\wh{\mathcal{M}}$ and $\mathcal{M}_h$ in proving (\ref{hausdorff1}) and we get (\ref{hausdorff2}). Now with (\ref{hausdorff1}) and (\ref{hausdorff2}), and using the Lipschitz continuity of the $(r+1)$th eigenvalue as a function of symmetric matrices (Weyl's inequality, cf. page 57, Serre, 2002), we obtain (\ref{In}) and then conclude the proof.
\hfill$\square$ 

\end{proof}

\hspace{-12pt}{\bf Proof of Theorem~\ref{generalization} }
\begin{proof}
First we show that
\begin{align}
\mathbb{P}\left( \mathcal{M}_h \subset [\mathcal{J}_{n,\eta} \cup \mathcal{G}^0_{n,\eta} ] \right) \rightarrow 1-\alpha. \label{confregzero1}
\end{align}
where $\mathcal{J}_{n,\eta} = \wh C_{n,h}(b_{h}(z_\alpha, c_{h}^{(d,r)}( \mathcal{C}_{n,\eta}^\complement)),\;\zeta_n^0)$. Denote events $E_{n,1} = \{(\mathcal{M}_h \cap \mathcal{K}_{h,\eta}^\complement) \subset \mathcal{J}_{n,\eta}\}$ and $E_{n,2} = \{(\mathcal{M}_h \cap \mathcal{K}_{h,\eta}) \subset \mathcal{G}^0_{n,\eta}\}.$ The following are two basic inequalities that will be used in the proof. 
\begin{align}
&\|\nabla f_h(x)\| + \|\nabla \wh f(x) - \nabla f_h(x) \|  \geq \|\nabla \wh f(x) \| \geq \|\nabla f_h(x)\| -\|\nabla \wh f(x) - \nabla f_h(x) \|,\label{basicrelation}\\
&\|\nabla f_h(x)\| + \|\nabla f(x) - \nabla f_h(x) \|  \geq \|\nabla f(x) \| \geq \|\nabla f_h(x)\| -\|\nabla f(x) - \nabla f_h(x) \| .\label{basicrelation2}
\end{align}
Due to (\ref{basicrelation2}) and Lemma~\ref{lambda2}, there exists a constant $C_1>0$ such that for $\mathcal{K}_{h,\eta}^{\dagger}=\{x\in\mathcal{H}:\; \|\nabla f(x)\| \leq C_1 h^\eta\}$, we have $\mathcal{K}_{h,\eta}^\complement \subset \mathcal{K}_{h,\eta}^{\dagger\complement}$, when $h$ is small enough. For $x\in  \mathcal{K}_{h,\eta}^\complement$, let $\wh N(x)=\|\nabla f(x)\|^{-1}\nabla \wh f(x)$, $N_h(x)=\|\nabla f(x)\|^{-1}\nabla f_h(x)$, and $N(x)=\|\nabla f(x)\|^{-1}\nabla f(x)$. Similar to Lemma~\ref{pointwisenormal}, we can show that 
\begin{align}
&\sup_{x\in  \mathcal{K}_{h,\eta}^\complement} \left\| \wh N(x) - N_h(x) \right\|= O_p\left(\gamma_{n,h}^{(1+\eta)} \right), \label{normgrad1}\\
&\sup_{x\in  \mathcal{K}_{h,\eta}^\complement} \left\| N(x)  - N_h(x) \right\|= O\left(h^{2-\eta}\right).\label{normgrad2}
\end{align}
Denote $M^*(x) = \|\nabla f(x)\|^{-1}M(x)$, $\Sigma^*(x)=M^*(x)^T\mathbf{R}M^*(x)$, $Q^*(x) = [f(x)\Sigma^*(x)]^{-1/2}= \|\nabla f(x)\|[f(x)\Sigma(x)]^{-1/2}$ and $Q_n^*(x) = \|\nabla f(x)\| [\wh f(x) \wh \Sigma(x)]^{-1/2}$. Then for $B_n(x)$ and $D_n(x)$ in (\ref{Bnx}) and (\ref{Dnx}), we have the following equivalent expressions. 
\begin{align}\label{Dnx2}
B_n(x)= \left\| Q_n^*(x)\wh V(x)^T \wh N(x) \right\| \text{ and } D_n(x) = \left\| Q^*(x) M^*(x)^T\left( d^2\wh f(x) - d^2 f_h(x)\right)\right\|.
\end{align}
Denote $E_n^*(x)=\wh V(x)^T \wh N(x) - V_h(x)^T N_h(x)$. Using (\ref{normgrad1}) and (\ref{normgrad2}), the following result can be obtained similar to (\ref{ridgenessdiff}).
\begin{align}\label{ridgenessdiff2}
\sup_{x\in\mathcal{M}_h \cap \mathcal{K}_{h,\eta}^\complement} \| E_n^*(x) - M^*(x)^T\Big(d^2 \wh f(x) - d^2 f_h(x)\Big)\|  = O_p\left( \gamma_{n,h}^{(1+\eta)} + (\gamma_{n,h}^{(2)})^2\right).
\end{align}
Denote $\mathcal{C}=\{x\in\mathcal{H}:\; \nabla f(x)=0\}$, which is the set of critical points of $f$. Note that for any point $x\in\mathcal{M}\cap\mathcal{C}^\complement$, $N(x)$ is a unit vector in the linear subspace spanned by $v_{1}(x),\cdots,v_r(x)$ by (\ref{condition1}). In other words, $\sum_{j=1}^r [v_j(x)^TN(x)]^2=1$. Hence similar to (\ref{minegienvalue}) we get
\begin{align*}
\lambda_{\min}(\Sigma^*(x)) & \geq  \frac{\lambda_{\min}(\mathbf{R})}{2\lambda_{\max}(D^+(D^+)^T)}  \min_{i\in\{r+1,\cdots,d\}}\sum_{j=1}^r \left[\frac{v_j(x)^TN(x)}{\lambda_i(x) - \lambda_j(x)}\right]^2 \\
&\geq \frac{\lambda_{\min}(\mathbf{R})}{2\lambda_{\max}(D^+(D^+)^T)} \min_{j
\in\{1,\cdots,r\}} \min_{i\in\{r+1,\cdots,d\}} \left[\frac{1}{\lambda_i(x) - \lambda_j(x)}\right]^2.
\end{align*}
Due to assumption (\textbf{F4}), we have $\inf_{x\in\mathcal{M}\cap\mathcal{C}^\complement}\lambda_{\min}(\Sigma^*(x))\geq C_2$ for some positive constant $C_2$. Therefore $\inf_{x\in \mathcal{M}_h \cap \mathcal{K}_{h,\eta}^\complement} \lambda_{\min}(\Sigma^*(x)) >C_2/2$ when $h$ is small enough, because of the Lipschitz continuity of $\lambda_{\min}$. 
%
Then using (\ref{ridgenessdiff2}) and following the similar arguments as in the proof of Proposition~\ref{ridgenessest}, we get 
\begin{align}
\sup_{x\in \mathcal{M}_h \cap \mathcal{K}_{h,\eta}^\complement } B_n(x) - \sup_{x\in\mathcal{M}_h \cap \mathcal{K}_{h,\eta}^\complement} D_n(x) 
 = O_p\left( \gamma_{n,h}^{(1+\eta)} + (\gamma_{n,h}^{(2)})^2\right),
\end{align}
which is similar to (\ref{ridgenessest2}). 
%
%
One can then use similar arguments for the Gaussian approximation (Theorem~\ref{gassianapprox}) and the corresponding extreme value distribution (Theorem~\ref{confidenceregion}) to get
\begin{align}\label{confidencut}
\mathbb{P}\left(\sqrt{nh^{d+4}}\sup_{x\in\mathcal{M}_h \cap \mathcal{K}_{h,\eta}^\complement }B_n(x) \leq b_{h}(z, c_{h}^{(d,r)}( \mathcal{C}_{n,\eta}^\complement),\zeta_n^0)\right) \rightarrow e^{-e^{-z}}.
\end{align}
Then Proposition~\ref{lambdaprop} and (\ref{confidencut}) imply that for any $\alpha\in(0,1)$, as $n\rightarrow\infty$,
\begin{align}\label{pen1}
\mathbb{P}(E_{n,1}) \rightarrow 1-\alpha.
\end{align}
%
Note that (\ref{basicrelation}) and Lemma~\ref{lambda2} imply that $\mathbb{P}(\mathcal{M}_h \cap \mathcal{K}_{h,\eta} \subset \mathcal{E}_{n,\eta}) \rightarrow 1$. Also it follows from Proposition~\ref{lambdaprop} that $\mathbb{P}(\mathcal{M}_h \cap \mathcal{K}_{h,\eta} \subset \{x\in\mathcal{H}:\; \wh\lambda_{r+1}(x)<\zeta_{n}^0\}) \rightarrow 1$. Hence $\mathbb{P}(E_{n,2}) \rightarrow 1.$ 
%
%
Combining this with (\ref{pen1}) yields
%
\begin{align}\label{upperconverge}
 \mathbb{P}\left( \mathcal{M}_h \subset [\mathcal{J}_{n,\eta} \cup \mathcal{G}^0_{n,\eta} ] \right) 
\geq & \mathbb{P}(E_{n,1}\cap E_{n,2}) \nonumber\\
= & \mathbb{P}(E_{n,1}) + \mathbb{P}(E_{n,2}) - \mathbb{P}(E_{n,1}\cup E_{n,2})  \rightarrow 1-\alpha.
 %
 %
 %
\end{align}
Denote $\mathcal{K}_{h,\eta}^* = \{x\in\mathcal{H}:\; \|\nabla f_h(x)\| \geq 2\mu_n\gamma_{n,h}^{(1)} + h^\eta\}$ and $\mathcal{J}_{n,\eta}^* = \wh C_{n,h}(b_{h}(z_\alpha, c_{h}^{(d,r)}( \mathcal{C}_{n,\eta}^*)))$. Note that 
\begin{align}\label{subsetineq}
\mathbb{P}\left( \mathcal{M}_h \subset [\mathcal{J}_{n,\eta} \cup \mathcal{G}^0_{n,\eta} ] \right)  \leq \mathbb{P}\left( (\mathcal{M}_h \cap \mathcal{K}_{h,\eta}^*)  \subset [\mathcal{J}_{n,\eta} \cup \mathcal{G}^0_{n,\eta} ] \right).
\end{align}
It follows from (\ref{basicrelation}) and Lemma~\ref{lambda2} that $\mathbb{P}(\mathcal{K}_{h,\eta}^* \cap \mathcal{G}^0_{n,\eta} = \emptyset)\rightarrow 1$. Also note that similar to (\ref{pen1}), we have $\mathbb{P}( (\mathcal{M}_h \cap \mathcal{K}_{h,\eta}^* )  \subset \mathcal{J}_{n,\eta}^* )  \rightarrow 1-\alpha.$ Since $c_{h}^{(d,r)}( \mathcal{K}_{h,\eta}^\complement) - c_{h}^{(d,r)}( \mathcal{K}_{h,\eta}^*) = o(1)$, we have $\mathbb{P}( (\mathcal{M}_h \cap \mathcal{K}_{h,\eta}^* )  \subset \mathcal{J}_{n,\eta})  \rightarrow 1-\alpha.$ Therefore the right-hand side of (\ref{subsetineq}) converges to $1-\alpha$. Combining this with (\ref{upperconverge}) we get (\ref{confregzero1}). Then (\ref{confregzero2}) in assertion (i) is a direct consequence by noticing that $\wh c_{n,l}^{(d,r)}( \mathcal{E}_{n,\eta}^\complement)$ is a consistent estimator of $c_{h}^{(d,r)}( \mathcal{K}_{h,\eta}^\complement)$. Assertions (ii) and (iii) can be proved using similar arguments as above, combined with the proof of Theorem~\ref{underconfidenceregion}.
%
%
 %
 %
\hfill$\square$
\end{proof}

%
%
%
%

\section*{Appendix B: Miscellaneous results}
In this appendix we collect some useful results used in the proof of Theorem~\ref{confidenceregion}. 
%
%
Recall $\rho_1,\rho_2,\rho_3\in\mathbb{Z}_+^d$ defined before the assumptions. 
\begin{lemma}\label{kernelratio}
Let $\alpha,\beta\in\mathbb{Z}_+^d$ with $d\geq2$ and $|\alpha|=|\beta|=3$. Under assumption (\textbf{K1}) we have \\
(i) 
\begin{align*}
\int_{\mathbb{R}^d}K^{(\alpha)}(u) K^{(\beta)}(u)du=
\begin{cases}
\int_{\mathbb{R}^d}[ K^{(\rho_q)}(u) ]^2du & \text{ if } \alpha + \beta \in \{2\gamma:\; \gamma\in\Pi(\rho_q)\}, q=1,2,3\\
0 & \text{ otherwise}
\end{cases},
\end{align*}
where $\Pi(\rho_q)$ is the set of all the permutations of the elements in $\rho_q$, $q=1,2,3$.\\
(ii)  
$a_K\geq 1$; and $b_K\leq 1$ when $d\geq3$.
\end{lemma}

\begin{proof}

(i) For $i=1,\cdots,d$ and $u=(u_1,\cdots,u_d)$, let $\zeta_i: \mathbb{R}^d\mapsto \mathbb{R}^d$ be a map such that $\zeta_i(u)=(u_1,\cdots,-u_i,\cdots,u_d)$. The following is a consequence of the properties for symmetric kernel functions $K$.

$$K^{(\alpha)}(u) K^{(\beta)}(u) = (-1)^{\alpha_i+\beta_i} K^{(\alpha)}(\zeta_i(u)) K^{(\beta)}(\zeta_i(u)),\; i=1,\cdots,d.$$
Therefore if any index in $\alpha+\beta$ is odd, then $\int_{\mathbb{R}^d}K^{(\alpha)}(u) K^{(\beta)}(u)du=0$. Now assume that $\alpha+\beta=2\gamma$ for some $\gamma\in\Pi(\rho_q)$, $q=1,\cdots,d$. Using integration by parts (for twice or none, depending on $\alpha$ and $\beta$), we obtain that $\int_{\mathbb{R}^d}[ K^{(\rho_q)}(u) ]^2du = \int_{\mathbb{R}^d}[ K^{(\gamma)}(u) ]^2du$. Then the result in (i) follows from the fact that $\int_{\mathbb{R}^d}[ K^{(\gamma)}(u) ]^2du=\int_{\mathbb{R}^d}[ K^{(\rho_q)}(u) ]^2du$, for all $\gamma\in\Pi(\rho_q)$, $q=1,\cdots,d$, again due to the spherical symmetry of $K$.

(ii) Let $\rho_4=(1,2,0,\cdots,0)^T\in\mathbb{Z}_+^d$ and $\rho_5=(1,0,2,\cdots,0)^T\in\mathbb{Z}_+^d$, where $\rho_5$ is only defined if $d\geq3$. Notice that 
\begin{align*}
\int_{\mathbb{R}^{d}}[K^{(\rho_2)}(s)]^2ds = \int_{\mathbb{R}^{d}}[K^{(\rho_4)}(s)]^2ds =\int_{\mathbb{R}^{d}}[K^{(\rho_5)}(s)]^2ds.
\end{align*}
We have $a_K\geq 1$ and $b_K\leq 1$ because 
\begin{align*}
\int_{\mathbb{R}^{d}}[K^{(\rho_1)}(s)]^2ds + \int_{\mathbb{R}^{d}}[K^{(\rho_4)}(s)]^2ds &\geq 2\int_{\mathbb{R}^{d}}K^{(\rho_1)}(s)K^{(\rho_4)}(s)ds
=2\int_{\mathbb{R}^{d}}[K^{(\rho_2)}(s)]^2ds,
\end{align*}
%
and
\begin{align*}
\int_{\mathbb{R}^{d}}[K^{(\rho_4)}(s)]^2ds + \int_{\mathbb{R}^{d}}[K^{(\rho_5)}(s)]^2ds &\geq 2\int_{\mathbb{R}^{d}}K^{(\rho_4)}(s)K^{(\rho_5)}(s)ds
=2\int_{\mathbb{R}^{d}}[K^{(\rho_3)}(s)]^2ds.
\end{align*}

\hfill$\square$
\end{proof}

Below we give an expression of $P$ in (\ref{Pmatrix}) when $d=3$ as an example.  Let $P_1$, $P_2$ and $P_3$ be the columns of $P$. Then
\begin{align*}
P_1 = &\begin{pmatrix}
a_Kt_1^2 + t_2^2+t_3^2 + 2t_1t_2+2t_1t_3 + 2b_Kt_2t_3 + t_4^2+t_5^2+b_Kt_6^2\\
2t_1t_4+2t_2t_4+2b_Kt_3t_4+2b_Kt_5t_6\\
2t_1t_5+2b_Kt_2t_5+2t_3t_5+2b_Kt_4t_6
\end{pmatrix},\\
P_2 = &\begin{pmatrix}
2t_1t_4+2t_2t_4+2b_Kt_3t_4+2b_Kt_5t_6\\
t_1^2 + a_Kt_2^2+t_3^2 + 2t_1t_2+2b_Kt_1t_3 + 2t_2t_3 + t_4^2+b_Kt_5^2+t_6^2\\
2b_Kt_1t_6+2t_2t_6+2t_3t_6+2b_Kt_4t_5
\end{pmatrix},\\
P_3 = &\begin{pmatrix}
 2t_1t_5 + 2b_Kt_2t_5+2t_3t_5+2b_Kt_4t_5\\
 2b_Kt_1t_6+2t_2t_6+2t_3t_6+2b_Kt_4t_5\\
  t_1^2 + t_2^2+a_Kt_3^2 + 2b_Kt_1t_2+2t_1t_3 + 2t_2t_3 + b_Kt_4^2+t_5^2+t_6^2
  \end{pmatrix}.
\end{align*}
We continue to use $d=3$ as an example to give the explicit expression of $L$ and $S$ in (\ref{PmatrixDecomp}). We can write $S=(a_K-1/b_K)\text{diag}(t_1^2, t_2^2, t_3^2)$
and $L = (L_1,L_2,L_3)$, where
\begin{align*}
L_{1} & = \sqrt{b_K}\begin{pmatrix}
t_1/b_K + t_2+t_3 & t_4 &t_5\\
t_4 & t_2/b_K+t_1+t_3 & t_6\\
t_5 & t_6 & t_3/b_K+t_1+t_2
\end{pmatrix},\\
L_{2} &= \sqrt{b_K}\begin{pmatrix}
t_6\\t_5\\t_4
\end{pmatrix},\;
L_{3} = \sqrt{1-b_K}\begin{pmatrix}
t_4 & t_2 & t_5 & t_3 & 0 & 0\\
t_1 & t_4 & 0 & 0 & t_6 & t_3\\
0 & 0 & t_1 & t_5 & t_2 & t_6
\end{pmatrix}.
\end{align*}

We need the following definition and probability result which are proved in our companion work Qiao (2019b). Suppose that $n_1$ and $n_2$ are positive integers and $0<\alpha_1,\alpha_2\leq 2$.

\begin{definition}[Local equi-$(\alpha_1, D^h_{t,v}, \alpha_2, B_{t,v})$-stationarity] \label{def-loc-stat} $\;$Let $\{Z_h(t,v), (t,v)\in \mathcal {S}_1\times \mathcal {S}_2, h\in \mathbb{H}\}$ be a class of non-homogeneous random fields, where $\mathbb{H}$ is an index set, $\mathcal {S}_1 \subset \mathbb{R}^{n_1}$ and $\mathcal {S}_1 \subset \mathbb{R}^{n_2}$. We say that this class is locally equi-$(\alpha_1, D^h_{t,v}, \alpha_2, B_{t,v})$-stationary, if the following three conditions hold. For any $t\in \mathcal {S}_1$,  $v\in \mathcal {S}_2$ and $h \in {\mathbb H},$ there exist non-degenerate matrices $D_{t,v}^h$ and $B_{t,v}$ such that 

\begin{alignat}{2}
&(i) \qquad\frac{ [1-r_h(t_1,t_2,v_1,v_2)]}{\|h^{-1}D_{t,v}^h(t_1-t_2)\|^{\alpha_1} + \|B_{t,v}(v_1-v_2)\|^{\alpha_2} }\rightarrow1\quad  \label{CC} \\
&\text{\rm as }\,\frac{\max\{\|t_1-t\|,\|t_2-t\|\}}{h}\rightarrow0 \text { \rm and } \max\{\|v_1-v\|,\|v_2-v\|\} \rightarrow0, \text{ uniformly in } h\in \mathbb{H}, \nonumber\\
&s\in\mathcal {S}_1, \text{ and } u\in\mathcal {S}_2, \text{ and}\nonumber\\
%
&(ii) \qquad 0< \inf\limits_{\substack{h\in \mathbb{H}, (t,v)\in \mathcal{S}_1\times \mathcal{S}_2\\ u\in \mathbb{R}^{n_1}\backslash \{0\}}} \frac{\|D_{t,v}^h\;u\|}{\|u\|}\leq \sup\limits_{\substack{h\in \mathbb{H}, (t,v)\in \mathcal{S}_1\times \mathcal{S}_2 \\ u\in \mathbb{R}^{n_1}\backslash \{0\}}} \frac{\|D_{t,v}^h\;u\|}{\|u\|} <\infty.\label{CCprime}\\
&(iii) \qquad 0< \inf\limits_{\substack{(t,v)\in \mathcal{S}_1\times \mathcal{S}_2\\ u\in \mathbb{R}^{n_2}\backslash \{0\}}} \frac{\|B_{t,v}\;u\|}{\|u\|}\leq \sup\limits_{\substack{(t,v)\in \mathcal{S}_1\times \mathcal{S}_2 \\ u\in \mathbb{R}^{n_2}\backslash \{0\}}} \frac{\|B_{t,v}\;u\|}{\|u\|} <\infty.\label{CCprime2}
\end{alignat}
\end{definition}

We consider $1\leq r_1<n_2$ and $1\leq r_2<n_2$ below. Let $H_{\alpha_i}^{(r_i)}$, $i=1,2$ denote the generalized Pickands' constant of Gaussian fields (see the appendix of Qiao, 2019b). Recall that $\Delta(\mathcal{L})$ is the reach of a manifold $\mathcal{L}$. 
\begin{theorem}\label{ProbMain}
Let $\mathcal{H}\subset\mathbb{R}^{n_1}$ be a compact set and and $\mathcal{L}\subset\mathbb{R}^{n_2}$ be an $r_2$-dimensional compact Riemannian manifold with $\Delta(\mathcal{L})>0$. For fixed $h_0$ with $0<h_0<1$, let $\{Z_h(t,v), (t,v)\in \mathcal{H}\times \mathcal{L}, 0<h\leq h_0\}$ be a class of Gaussian centered locally equi-$(\alpha_1, D_{t,v}^h, \alpha_2, B_{t,v})$-stationary fields with $0<\alpha_1,\alpha_2\leq2,$ and all components of $D_{t,v}^h$ continuous in $h$, $t$ and $v$.  Suppose that $D_{t,v}^h$, $(t,v)\in \mathcal{H}\times\mathcal{L}$, uniformly converges, as $h\downarrow0$, to a matrix field $D_{t,v}$, $(t,v) \in {\mathcal H}\times\mathcal{L},$ with continuous components. Let $\mathcal {M}_{h}\subset\mathcal{H}$ be $r_1$-dimensional compact Riemannian manifolds with $\inf_{0<h\leq h_0}\Delta(\mathcal{M}_{h})>0$, and $\sup_{0<h\leq h_0}\mathscr{H}_{r_1}(\mathcal{M}_{h})<\infty$. For $x>0,$ let 
\begin{align}\label{QDelta}
Q(x)=\sup_{0<h\leq h_0}\{|r_h(t+s,s,v+u,u)|: (t+s,v+u), (s,u)\in\mathcal{M}_{h}\times\mathcal{L}, \|t\|>hx\},
\end{align}
where $r_h$ denotes the covariance function of $Z_h(t,v)$. Suppose that, for any $x>0$, there exists $\eta > 0$ such that
\begin{align}\label{SupGauss1}
Q(x) < \eta<1.
\end{align}
Furthermore, let  $x_0> 0$ be such that for a function $v(\cdot)$ and for $x>x_0,$ we have

\begin{align}\label{SupGauss2}
Q(x)\Big|(\log x)^{2r/\alpha_1}\Big|\leq v(x),
\end{align}
where $v$ is a monotonically decreasing, such that, for some $p > 0,$ $v(x^p)=O(v(x))=o(1)$ and $v(x) x^{p}\to \infty$ as $x\rightarrow\infty$. Let
\begin{align}\label{BetaExp}
\beta_h=&\Big(2r_1\log\frac{1}{h} \Big)^{\frac{1}{2}}  +\Big(2r_1\log\frac{1}{h}\Big)^{-\frac{1}{2}} \nonumber\\
&\hspace{1cm}\times\bigg[\Big(\frac{r_1}{\alpha_1}+\frac{r_2}{\alpha_2}-\frac{1}{2}\Big)\log{\log\frac{1}{h} } +\log\bigg\{\frac{(2r_1)^{\frac{r_1}{\alpha_1}+\frac{r_2}{\alpha_2}-\frac{1}{2}}}{\sqrt{2\pi}}H_{\alpha_1}^{(r_1)}H_{\alpha_2}^{(r_2)}I(\mathcal {M}_h\times\mathcal{L})\bigg\}\bigg],
\end{align}
where $I(\mathcal {M}_h\times\mathcal{L})=\int_{\mathcal{L}}\int_{\mathcal {M}_{h}}\|D_{s,u} M_s\|_{r_1} \| B_{s,u} M_u\|_{r_2}d\mathscr{H}_{r_1}(s)d\mathscr{H}_{r_2}(u)$ with $M_s$ an $n_1\times r_1$ matrix with orthonormal columns spanning $T_s\mathcal {M}_{h},$ and $M_u$ an $n_2\times r_2$ matrix with orthonormal columns spanning $T_u\mathcal {L}.$ Then
\begin{align}\label{ConclRes}
\lim_{h\rightarrow0}\mathbb{P}\left\{\sqrt{2r_1\log\tfrac{1}{h} } \left(\sup_{v\in\mathcal{L}}\sup_{t\in\mathcal {M}_{h}}Z_h(t,v)-\beta_h \right)\leq z\right\}=\exp\{-2\exp\{-z\}\}.
\end{align}
\end{theorem}

\section*{References}
\begin{description}
\itemsep0em 
\item Arias-Castro, E., Mason, D. and Pelletier, B. (2016). On the estimation of the gradient lines of a density and the consistency of the mean-shift algorithm. {\em Journal of Machine Learning} {\bf 17} 1-28.
 \item Chazal, F., Lieutier, A., and Rossignac, J. (2007). Normal-map between normal-compatible manifolds. {\em International Journal of Computational Geometry \& Applications} {\bf 17} 403-421.
\item Chernozhukov, V., Chetverikov, D. and Kato, K. (2014). Gaussian approximation of suprema of empirical processes. {\em Ann. Statist.} {\bf 42}, 1564-1597.
\item Cuevas, A., Gonz\'{a}lez-Manteiga, W., and  Rodr\'{i}guez-Casal, A. (2006). Plug-in estimation of general level sets. {\em Australian \& New Zealand Journal of Statistics} {\bf 48} 7-19.
\item Dunajeva, O. (2004). The second-order derivatives of matrices of eigenvalues and eigenvectors with an application to generalized F-statistic. {\em Linear Algebra and its Applications} {\bf 388} 159-171.
\item Einmahl, U. and Mason, D.M. (2005). Uniform in bandwidth consistency of kernel-type function estimators. {\em Ann. Statist.} {\bf 33} 1380-1403.
\item Federer, H. (1959). Curvature measures. {\em Trans. Amer. Math. Soc.} {\bf 93} 418-491.
\item Gin\'{e}, E. and Guillou, A. (2002). Rates of strong uniform consistency for multivariate kernel density estimators. {\em Ann. Inst. H. Poincar\'{e} Probab. Statist.} {\bf 38} 907-921.
\item Higham, N. (2008). {\em Functions of Matrices - Theory and Computation}. SIAM, Philadelphia.
\item Ipsen, I.C.F. and Rehman, R. (2008). Perturbation bounds for determinants and characteristic polynomials. {\em SIAM J. Matrix Anal. Appl.} {\bf 30} 762-776.
\item Qiao, W. (2019a). Nonparametric estimation of surface integrals on density level sets. {\em arXiv: 1804.03601}.
\item Qiao, W. (2019b). Extremes of locally stationary chi-fields on manifolds, {\em preprint}.
\item Serre, D. (2002). {\em Matrices: Theory and Applications}. Springer-Verlag, New York.
\item van der Vaart, A.W. and Wellner, J.A. (1996). {\em Weak Convergence and Empirical Processes: With Applications to Statistics}. Springer, New York.
\end{description}
\end{document}

%% file: Reference.tex
\section*{References}
\begin{description}
\itemsep0em 
\item Arias-Castro, E., Donoho, D.L. and Huo, X. (2006). Adaptive multiscale detection of filamentary structures in a background of uniform random points. {\em Ann. Statisti.} {\bf 34} 326-349.
\item Baddeley, A.J. (1992). Errors in binary images and $L^p$ version of the Hausdorff metric. {\em Nieuw Archief Voor Wiskunde} {\bf 10} 157-183.
\item Bickel, P. and Rosenblatt, M. (1973). On some global measures of the deviations of density function estimates. {\em The Annals of Statistics}, {\bf 1}, 1071-1095.
%
%
\item Bugni, F. (2010). Bootstrap inference in partially identified models defined by moment inequalities: coverage of the identified set, {\em Econometrica} {\bf 76} 735-753.
\item Cadre, B. (2006). Kernel estimation of density level sets. {\em J. Multivariate Anal.} {\bf 97} 999-1023.
%
%
\item Chen, Y.-C., Genovese, C. and Wasserman, L. (2015). Asymptotic theory for density ridges. {\em The Annals of Statistics}, {\bf 43}(5), 1896-1928.
\item Cheng, M.-Y., Hall, P. and Hartigan, J.A. (2004). Estimating gradient trees. In {\em A Festschrift for Herman Rubin Institute of Mathematical Statistics Lecture Notes - Monograph Series} {\bf 45} 237-249. IMS, Beachwood, OH.
\item Chernozhukov, V., Chetverikov, D. and Kato, K. (2014). Gaussian approximation of suprema of empirical processes. {\em Ann. Statist.} {\bf 42}, 1564-1597.
\item Comaniciu, D. and Meer, P. (2002). A robust approach toward feature space analysis. {\em IEEE Trans. Pattern and Analysis and Machine Intelligence}, {\bf 24} 603-619.
%
%
\item Duong, T., Cowling, A., Koch, I., and Wand, M.P. (2008). Feature significance for multivariate kernel density estimation, {\em Computational Statistics \& Data Analysis}, {\bf 52}, 4225-4242.
\item Eberly, D. (1996). {\em Ridges in Image and Data Analysis}. Kluwer, Boston, MA.
\item Federer, H. (1959). Curvature measures. {\em Trans. Amer. Math. Soc.} {\bf 93} 418-491.
\item Genovese, C. R. and Perone-Pacifico, M. and Verdinelli, I. and Wasserman, L. (2009). On the path density of a gradient field. {\em Ann. Statist.} {\bf 37}, 3236-3271.
\item Genovese, C. R. and Perone-Pacifico, M. and Verdinelli, I. and Wasserman, L. (2012). The geometry of nonparametric filament estimation. {\em J. Amer. Statist. Assoc.} {\bf 107}, {788-799}.
\item Genovese, C. R. and Perone-Pacifico, M. and Verdinelli, I. and Wasserman, L. (2014). Nonparametric ridge estimation. {\em Ann. Statist.} {\bf 42}, 1511-1545.
\item Genovese, C., Perone-Pacifico, M., Verdinelli, I. and Wasserman, L. (2017). Finding singular features. {\em Journal of Computational and Graphical Statistics}, {\bf 26}(3), 598-609.
\item Hall, P. (1992). Effect of Bias estimation on coverage accuracy of bootstrap confidence intervals for a probability density. {\em Ann. Statist.} {\bf 20}, 675-694.
%
\item Hall, P., Qian, W. and Titterington, D. M. (1992). Ridge finding from noisy data. {\em J. Comp. Graph. Statist.}, {\bf 1}, 197-211.
\item Hartigan, J. A. (1987). Estimation of a convex density contour in two dimensions. {\em J. Amer. Statist. Assoc}, {\bf 82}, {267-270}.
\item Konakov, V. D., Piterbarg, V. I. (1984). On the convergence rate of maximal deviation distribution for kernel regression estimate. {\em Journal of Multivariate Analysis}, {\bf 15}, 279-294.
\item Li, W. and Ghosal, S. (2019). Posterior Contraction and Credible Sets for Filaments of Regression Functions. {\em Arxiv: 1803.03898}
\item Magnus, X. and Neudecker, H. (2007). {\em Matrix Differential Calculus with Applications in Statistics and Econometrics}, 3rd edition, John Wiley \& Sons, Chichester.
\item  Mammen, E. and Polonik, W. (2013). Confidence sets for level sets. {\em Journal of Multivariate Analysis} {\bf 122} 202-214.
\item Mason, D.M. and Polonik, W. (2009). Asymptotic normality of plug-in level set estimates. {\em The Annals of Applied Probability} {\bf 19} 1108-1142.
%
%
\item Ozertem, U. and Erdogmus, D. (2011). Locally defined principal curves and surfaces.  {\em Journal of Machine Learning Research}, {\bf 12}, 1249-1286.
\item Piterbarg, V.I. (1994). High excursions for nonstationary generalized chi-square processes. {\em Stochastic Processes and their Applications}, {\bf 53}, 307-337.
\item Piterbarg, V.I. (1996). \emph{Asymptotic Methods in the Theory of Gaussian Processes and Fields}, Translations of Mathematical Monographs, Vol. 148, American Mathematical Society, Providence, RI.
\item Polonik, W. (1995). Measuring mass concentrations and estimating density contour clusters - an excess mass approach. {\em Ann. Statist.} {\bf 23} 855-881.
\item Polonik, W. and Wang, Z. (2005). Estimation of regression contour clusters: an application of the excess mass approach to regression. {\em Journal of Multivariate Analysis}, {\bf 94} 227-249.
%
%
\item Qiao, W. (2019a). Nonparametric estimation of surface integrals on density level sets. {\em arXiv: 1804.03601}.
\item Qiao, W. (2019b). Extremes of locally stationary chi-fields on manifolds, {\em preprint}.
\item Qiao, W. and Polonik, W. (2016). Theoretical analysis of nonparametric filament estimation. {\em The Annals of Statistics}, {\bf 44}(3), 1269-1297.
\item Qiao, W. and Polonik, W. (2018). Extrema of rescaled locally stationary Gaussian fields on manifolds, {\em Bernoulli}, {\bf 24}(3), 1834-1859.
\item Qiao, W. and Polonik, W. (2019). Nonparametric confidence regions for level sets: statistical properties and geometry. {\em Electronic Journal of Statistics}, {\bf 13}(1), 985-1030.
\item Rosenblatt, M. (1976). On the maximal deviation of $k$-dimensional density estimates. {\em Ann. Probab.}, {\bf 4}, 1009--1015.
\item Sousbie, T., Pichon, C., Colombi, S., Novikov, D. and Pogosyan, D. (2008). The 3D skeleton: tracing the filamentary structure of the Universe. {\em Mon. Not. R. Astron. Soc. } {\bf 383} 1655-1670.
%
%
\item Tsybakov, A.B. (1997). Nonparametric estimation of density level sets. {\em Ann. Statist.} {\bf 25} 948-969.
%
%
\item von Luxburg, U. (2007). A tutorial on spectral clustering. {\em Stat. Comput.} {\bf 17} 395-416.
\item Wegman, E. J., Carr, D. B. and Luo, Q. (1993). Visualizing multivariate data. In {\em Multivariate Analysis: Future Directions}, Ed. C. R. Rao, 423-466. North Holland, Amsterdam.
\item Wegman, E.J. and Luo, Q. (2002). Smoothings, ridges, and bumps. In {\em Proceedings of the ASA (published on CD). Development of the relationship between geometric aspects of visualizing densities and density approximators, and a discussion of rendering and lighting models, contouring algorithms, stereoscopic display algorithms, and visual design considerations} 3666-3672. American Statistical Association.
\item Xia, Y. (1998). Bias-corrected confidence bands in nonparametric regression. {\em J. R. Statist. Soc. B} {\bf 60} 797-811.
\end{description}